\documentclass[11pt]{article} 

\usepackage{amsmath}
\usepackage{lmodern}
\usepackage[T1]{fontenc}
\usepackage[babel=true]{microtype}
\usepackage{amsxtra}%
\usepackage{amsfonts}%
\usepackage{amssymb}%
\usepackage{amsthm}
\usepackage{graphicx}
\usepackage{color,hyperref}
\hypersetup{colorlinks,breaklinks,
             linkcolor=blue,urlcolor=blue,
           anchorcolor=blue,citecolor=blue}
       
\usepackage{subfigure}
\usepackage{appendix}
\usepackage[small]{caption}

\usepackage[margin=2.54cm]{geometry}

\newcommand{\eps}{\varepsilon}

\newcommand{\R}{\mathbb{R}}

\newcommand{\C}{\mathbb{C}}

\renewcommand{\phi}{\varphi}

\newcommand{\mcl}{\mathcal{L}}
\newcommand{\spn}{\text{\span}}

\usepackage{enumitem}
\setitemize{noitemsep,topsep=0pt,parsep=2pt,partopsep=0pt}
\renewcommand{\Re}{\mathrm{Re} \,}

\def\XXint#1#2#3{{\setbox0=\hbox{$#1{#2#3}{\int}$ }
		\vcenter{\hbox{$#2#3$ }}\kern-.6\wd0}}

\newtheorem{thm}{Theorem}

\newtheorem*{thm*}{Theorem}

\newtheorem{prop}{Proposition}
\newtheorem{lemma}[prop]{Lemma}
\newtheorem{corollary}[prop]{Corollary}

\newtheorem{hyp}{Hypothesis}

\newtheorem{remark}[prop]{Remark}

\numberwithin{equation}{section}
\numberwithin{prop}{section}

\newcommand{\rg}{\mathrm{Rg}}
\renewcommand{\spn}{\mathrm{span}}

\begin{document}
\begin{center}
{\fontsize{15}{15}\fontseries{b}\selectfont{Pushed-to-pulled front transitions: continuation, speed scalings, and hidden monotonicty}}\\[0.2in]
Montie Avery$^1$, Matt Holzer$^2$, and Arnd Scheel$^1$ \\[0.1in]
\textit{\footnotesize 
$^1$University of Minnesota, School of Mathematics,   206 Church St. S.E., Minneapolis, MN 55455, USA\\
$^2$Department of Mathematical Sciences, George Mason University, Fairfax, VA, USA
}
\end{center}

\begin{abstract}
  We analyze the transition between pulled and pushed fronts both analytically and numerically from a model-independent perspective. Based on minimal conceptual assumptions, we show that pushed fronts bifurcate from a branch of pulled fronts with an effective speed correction that scales quadratically in the bifurcation parameter. Strikingly, we find that in this general context without assumptions on comparison principles, the pulled front loses stability and gives way to a pushed front when monotonicity in the leading edge is lost. Our methods rely on far-field core decompositions that identify explicitly asymptotics in the leading edge of the front. We show how the theoretical construction can be directly implemented to yield effective algorithms that determine spreading speeds and bifurcation points with exponentially small error in the domain size. Example applications considered here include an extended Fisher-KPP equation, a Fisher-Burgers equation, negative taxis in combination with logistic population growth, an autocatalytic reaction, and a Lotka-Volterra model. 
\end{abstract}


\section{Introduction}

Propagation into unstable states is often mediated by invasion fronts. One is interested both in the speed of propagation and in the selected state in the wake of the invasion fronts. Examples of the role of such fronts in experiment, computation, and analysis abound and we refer to \cite{vanSaarloosReview} for a review. At small amplitude, growth of disturbances is determined by the linearization at the unstable state. Assuming that the dynamics in the leading edge are effectively governed by this linearized equation, one can then derive a spreading speed for disturbances from a \emph{linear marginal stability} criterion, that is, finding the supremum of all speeds at which disturbances grow in a comoving frame; see again \cite{vanSaarloosReview}, \cite{bers1983handbook} for background from plasma physics where such questions were first studied, and \cite{HolzerScheelPointwiseGrowth} for a more recent and detailed mathematical approach. Such linear predictions are clearly useful and allow at times for explicit, algebraic characterizations of the invasion speed. Concluding that such linear predictions are accurate for the nonlinear equation is usually very difficult analytically, crucially because linear predictions are often incorrect: states selected in the wake are of finite, non-small amplitude and nonlinearity can cause instabilities of an invasion process that propagates at the linearly predicted spreading speed and lead to faster propagation. Such acceleration is usually attributed to non-convex nonlinearities, weakly subcritical bifurcations, or generally competing nonlinear driving and saturation. Fronts that mediate invasion at the linear speed are commonly referred to as \emph{pulled fronts}; fronts that mediate the invasion at faster speeds due to instability of fronts at the linear spreading speeds are referred to as \emph{pushed fronts}. The aim of this paper is to analyze the transition from pulled to pushed front invasion from a bifurcation perspective, focusing on minimal assumptions on existence of fronts and spectral properties of the linearization at the front. 

Our motivation originates in difficulties with practical attempts at determining spreading speeds. On the one hand, computing linear spreading speeds can often be accomplished with high accuracy and without actually tackling a nonlinear PDE in an unbounded domain but rather an (often quite challenging) algebraic problem. In fact, measuring the front speed in direct simulations is quite difficult because of the slow convergence of speeds $c(t)$ towards the predicted linear speed $c_\mathrm{lin}$, $c(t)- c_\mathrm{lin}\sim t^{-1}$; see \cite{Bramson1,Bramson2,Lau,Comparison1,EbertvanSaarloos,averyscheelselection}. Finding fronts and their speeds in a bounded domain of size $L$ directly using for instance a Newton method where a phase condition that centers the front profile in the middle of the computational domain is compensated by the front speed as a Lagrange multiplier leads to errors of size $L^{-2}$. Since computations need to resolve the exponentially small tails in the leading edge, underflow and round-off errors put effective limits on the size $L$ of the computational domain in this scenario and lead to non-negligible errors for the speed \cite{adss}. 

On the other hand, convergence to pushed fronts is exponential both in time $t$ when performing direct simulations and exponential in space $L$ when using the Newton-approach described above \cite{HadelerRothe,SattingerWeightedNorms,adss}. 

Taken together, one would wish to compute 
\begin{itemize}
 \item pulled front speeds from linear (algebraic) information;
 \item pushed front speeds from nonlinear boundary value problems. 
\end{itemize}
For this strategy to work reliably, one clearly needs to 
\begin{itemize}
 \item determine transitions from pulled to pushed front propagation;
 \item find predictions for pushed front speeds near the pulled-to-pushed transition.
\end{itemize}
The latter part is necessary since exponential convergence of pushed fronts speeds either in time $t$ or in domain size $L$ is slow, with exponential rate converging to zero near the transition to pulled front propagation. 

Our results in this regard can be summarized as follows:
\begin{itemize}
 \item we develop numerical continuation for both pushed and pulled fronts that continues past pushed-to-pulled transitions;
 \item we identify a computable criterion for a pushed-to-pulled transition;
 \item we predict pushed speeds near the transition via leading-order corrections to linear speeds.
\end{itemize}
The main analytical result, precisely formulated in Theorems~\ref{t: pulled pushed unfolding}--\ref{t: pulled pushed stability}, can be informally stated as follows. 

\textbf{Theorem} (Pushed-to-pulled transition) \emph{
The pushed-to-pulled transition is a codimension-1 bifurcation. For a suitable orientation of the generic parameter $\mu$, pulled fronts with speed $c_\mathrm{lin}(\mu)$ are marginally stable for $\mu>0$ and unstable for $\mu<0$. Pushed fronts exist and are marginally stable for $\mu<0$ with leading-order speed $c_\mathrm{pushed}=c_\mathrm{lin}(\mu)+c_2 \mu^2$ for some $c_2>0$. Pulled fronts are monotone in the leading edge when stable, $\mu>0$, and non-monotone when unstable, $\mu<0$.
}

Numerical algorithms are a natural consequence of our analysis, in which we identify ``explicit'' tail behavior of pulled fronts so that we can find pulled and pushed fronts analytically as strongly localized corrections to this tail-behavior. In the remainder of this introduction, we describe the general setup we use to formulate our results and formulate conditions for a pushed-to-pulled transition and for a generic unfolding. We then state our main analytical results on existence and marginal stability, including a brief discussion of front selection. The remainder of this paper is occupied by proofs of these main results, their applications to numerical algorithms, and implications for several concrete model PDEs. 

\paragraph{Setup.}

To fix ideas, consider the semilinear  parabolic-elliptic system
\begin{align}\label{e:pe}
	M u_t = \mathcal{P}(\partial_x) u + f(u; \mu), \quad u \in \R^n,
\end{align}
with parameter $\mu \in \R$. We assume that $M \in \R^{n \times n}$ is a diagonal matrix whose first $k$ diagonal entries are equal to 1, with all other diagonal entries equal to zero; we notably allow $k=n$, the purely parabolic case but of course neeed $k\geq 1$. For the differential operator, we assume the ellipticity condition
\begin{equation}\label{e: ellipticity}
 \mathcal{P}(\partial_x)=\sum_{j=1}^{2m} P_j \partial_x^j,\   P_j\in \R^{n\times n}, \  \Re(-1)^m \lambda<0 \text{ for all eigenvalues }\lambda \text{ of } P_{2m}.
\end{equation}
Note that we choose $P_0=0$, absorbing constant terms into $f$. 

The nonlinearity is assumed to be smooth, of class $C^2$, and allow for a trivial equilibrium, $f(0;0)=0$, and a state $u_-\in\R^n$ that is selected as a result of the invasion process, $f(u_-;0)=0$. For simplicity, we assume that these states are independent of the parameter $\mu$, $f(0;\mu)=f(u_-;\mu)=0$, possibly first changing coordinates in $\R^n$ in a $\mu$-dependent fashion. The analysis and results presented here will hold true for nonlinearities $f=f(u,\partial_x u,\ldots,\partial_x^{2m-1}u;\mu)$ and when allowing for $\mu$-dependence in $\mathcal{P}$ , or in cases where the order of $\mathcal{P}$ changes in different components, and we choose the current setup for notational simplicity, only.

\paragraph{Instability in the leading edge, spreading speeds, and leading-edge profiles.}

We assume that the trivial state $u=0$ is unstable. We therefore assume that the linearization 
in the leading edge
\begin{align*}
   Mu_t= (\mathcal{P}(\partial_x) + f_u (0; \mu)) u,
\end{align*}
possesses exponentially growing solutions $\mathrm{exp}(\lambda t+ i kx)u_0$ for some $u_0\in \R^n$, $k\in\R$, $\Re\lambda>0$. In a comoving frame, this instability may be convective in nature, that is, the solution to 
\begin{align}\label{e:lin}
   Mu_t= (\mathcal{P}(\partial_x) + cMu_x+ f_u (0; \mu)) u, \ x\in\R, \qquad u(t=0,x)=I \delta(x),
\end{align}
with identity matrix $I$ and Dirac-delta $\delta(x)$ may decay exponentially in any finite interval $x\in[-L,L]$. In fact, the results in \cite{HolzerScheelPointwiseGrowth} show that pointwise exponential decay holds for all sufficiently large speeds $c\in (c_\mathrm{lin},\infty)$. At the critical speed $c_\mathrm{lin}$, pointwise exponential decay is obstructed by the presence of a singularity of the resolvent Green's function on the imaginary axis. We assume here that this singularity is located at $\lambda=0$ and, as was shown in \cite{HolzerScheelPointwiseGrowth} to be generically the case, is given by a simple pinched double root of the dispersion relation; see Hypothesis~\ref{h:sdr} for further details. 

Much of the discussion until now comprises marginal stability information. The assumptions that we need for the proof of our main bifurcation result are slightly weaker and we shall formulate those now. We will later comment on the relation to marginal stability as discussed here. 

First, consider  the linearization at $u=0$  in a comoving frame with speed $c$ \eqref{e:lin}, take Fourier-Laplace transform $u(t,x)=\mathrm{exp}({\lambda t + \nu x})$, and find the symbol
\begin{align}\label{e:sym}
	A(\lambda,\nu, c; \mu) = \mathcal{P}(\nu) + c \nu M + f_u (0; \mu) - \lambda M. 
\end{align} 
Associated with this linearization is the dispersion relation 
\begin{equation}\label{e:disp}
 d_c(\lambda,\nu;\mu)=d_0(\lambda-c\nu,\nu;\mu)=\mathrm{det}\,A( \lambda,\nu, c; \mu).
\end{equation}
The next hypothesis describes a generic singularity of the pointwise Green's function at $\lambda=0$. It roughly states that two solutions $\mathrm{exp}(\lambda t+\nu_j(\lambda)x)u_j(\lambda)$, $j=1,2$ collide at $\lambda=0$, $\nu_{1/2}(0)=\nu_0$, $u_{1/2}(0)=u_0$, with generic unfolding in $\lambda$. 

\begin{hyp}[Simple double root]\label{h:sdr}
 We assume that two spatial roots $\nu$ of the dispersion relation collide at $\lambda=0$ for the critical speed $c_0$ at $\mu=0$:
 \[
  d_{c_0}(0,\nu_0;0)=0, \  \partial_\nu d_{c_0}(0,\nu_0;0)=0, \  \partial_{\nu\nu} d_{c_0}(0,\nu_0;0)\partial_{\lambda} d_{c_0}(0,\nu_0;0)< 0,
 \]
for some $\nu_0<0$.\footnote{The terminology ``simple double root'' is motivated by the fact that $(0,\nu_0)$ is ``simple'' in a degree counting sense as a solution to the double root equation $d=\partial_\nu d=0$ assuming that $\partial_{\nu\nu}d,\partial_\lambda d\neq0$.}
\end{hyp}
We briefly note that such double roots at $\lambda=0$ are robust and can be continued in parameters.

\begin{lemma}[Robustness of simple double roots]\label{l: robustness sPDR}
    There exist smooth $\lambda_\mathrm{dr}(\mu,c)$ and $\nu_\mathrm{dr}(\mu,c)$, $\lambda_\mathrm{dr}(0,c_0)=0$ and $\nu_\mathrm{dr}(0,c_0)=\nu_0$, such that  
\begin{align*}
	                    d_{c}(\lambda_\mathrm{dr}(\mu,c),\nu_\mathrm{dr}(\mu,c);\mu)&=0, \qquad  
	  &\partial_\nu      d_{c}(\lambda_\mathrm{dr}(\mu,c),\nu_\mathrm{dr}(\mu,c);\mu)=0, \\  
	  \partial_{\nu\nu} d_{c}(\lambda_\mathrm{dr}(\mu,c),\nu_\mathrm{dr}(\mu,c);\mu)&\neq 0,\qquad    
	  &\partial_\lambda  d_{c}(\lambda_\mathrm{dr}(\mu,c),\nu_\mathrm{dr}(\mu,c);\mu)\neq 0.
	\end{align*}
Moreoever, adjusting $c$ as a function of the parameter $\mu$, the double root is located at the origin. That is, there exist smooth $c_\mathrm{lin} (\mu)$ and $\nu_\mathrm{lin}(\mu)$ with 
$c_\mathrm{lin}(0)=c_0$, $\nu_\mathrm{lin}(0)=\nu_0$, 
as given in  Hypothesis~\ref{hyp: double root}, solving
\begin{align*}
	  d_{c_\mathrm{lin}(\mu)}(0,\nu_\mathrm{lin}(\mu);\mu)&=0, \qquad  & \partial_\nu  d_{c_\mathrm{lin}(\mu)}(0,\nu_\mathrm{lin}(\mu);\mu)=0, \\  \partial_{\nu\nu}  d_{c_\mathrm{lin}(\mu)}(0,\nu_\mathrm{lin}(\mu);\mu)&\neq 0, \qquad  &\partial_{\lambda}  d_{c_\mathrm{lin}(\mu)}(0,\nu_\mathrm{lin}(\mu);\mu)\neq 0.
	\end{align*}
\end{lemma}
Double roots induce Jordan block type spatial behavior, as made precise in the following corollary. 
\begin{corollary}\label{c: asymptotic solutions}
	The linearization in the leading edge
	\begin{align*}
		(\mathcal{P}(\partial_x) + c_\mathrm{lin}(\mu) M \partial_x + f_u (0; \mu)) u = 0
	\end{align*}
	has solutions
	\begin{align}\label{e: le}
		u(x, \mu) = (\alpha (u_0 (\mu) x + u_1 (\mu)) + \beta u_0 (\mu)) e^{\nu_\mathrm{lin}(\mu) x},\qquad \alpha, \beta \in \R,
	\end{align}
	where $u_0(\mu),u_1(\mu)\in \R^n$ are smooth.
\end{corollary}

Proofs will be given in Section~\ref{s: preliminaries}.

\paragraph{Existence of fronts, characterization of pushed-to-pulled transition.} 
The transition between pushed and pulled fronts crucially relies on nonlinear contributions. Rather than making those explicit in a specific example, we make conceptual assumptions on existence and leading-edge asymptotics of a nonlinear front profile. 

\begin{hyp}[Existence of front at a pushed-to-pulled transition]\label{hyp: front existence}
	For $\mu = 0$, there exists a front $q_0 (x)$, that is, a stationary solution in the frame moving with the speed $c_0$ from Hypothesis~\ref{h:sdr},  solving
	\begin{align*}
		(\mathcal{P}(\partial_x) + c_0 M \partial_x) q_0 + f(q_0; 0) = 0,
	\end{align*}
	with asymptotics 
	\begin{align}\label{e: asyle}
		q_0 (x) = u_0(0) e^{\nu_0 x} + \mathrm{O}(e^{(\nu_0 - \eta) x}), \quad x \to \infty
	\end{align}
	and 
	\begin{align}
		q_0 (x) = u_- + \mathrm{O}(e^{\eta x}), \quad x \to -\infty
	\end{align}
	for some $\eta > 0$, with $u_0$ defined in Corollary~\ref{c: asymptotic solutions}. 
\end{hyp}
Note that, in particular, the asymptotics in the leading edge $x\to\infty$ do not include a linear term $x\, \mathrm{exp}(\nu_0 x)$ as generically expected from Corollary~\ref{c: asymptotic solutions}. The absence of this term is the key codimension-one assumption encoding the transition. 
In the hypothesis, we set a non-zero coefficient of $u_0(0) e^{\nu_0 x}$ in \eqref{e: asyle} to 1, which can readily be achieved by scaling $u_0(0)$. 
The asymptotics \eqref{e: asyle} are consistent with the existence of solutions to the linearization of the form \eqref{e: le} shown in Corollary~\ref{c: asymptotic solutions}. Thinking of the existence problem as a shooting problem in an ODE, Hypothesis~\ref{hyp: front existence} corresponds to the existence of a heteroclinic orbit between the equilibria corresponding to $u_-$ and $0$. Associated with this heteroclinic is a dimension counting question: what is the dimension of the unstable manifold of $u_-$, and what is the dimension of the strong stable manifold of $0$ associated with decay rate $\mathrm{exp}(\nu_0 +\eps x)$. Our next assumptions will clarify this dimension counting question via assumptions on Fredholm properties of the linearization at the heteroclinic profile in suitable exponentially weighted spaces. Those assumptions will in particular clarify that the assumption on vanishing of the linear term $x\,\mathrm{exp}(\nu_0 x)$ makes Hypothesis~\ref{hyp: front existence} a codimension-one assumption, corresponding to the codimension-one situation of a transition between pushed and pulled front invasion.

\paragraph{Fredholm indices and pinching conditions.}
We let $B_0 = \mathcal{P}(\partial_x) + c_0 M \partial_x + f_u (q_0; 0)$ denote the linearization at the front described in Hypothesis~\ref{hyp: front existence}. In order to capture solutions with precise leading-edge asymptotics, we introduce exponentially weighted function spaces as follows. 

For $\eta_\pm \in \R$, we let $\omega_{\eta_-, \eta_+}$ be a smooth positive weight function satisfying 
\begin{align}
	\omega_{\eta_-, \eta_+} (x) = \begin{cases}
		e^{\eta_- x}, & x \leq -1, \\
		e^{\eta_+ x}, & x \geq 1. 
	\end{cases}
\end{align}
We let $W^{k, p}_{\mathrm{exp}, \eta_-, \eta_+} (\R)$ denote the corresponding weighted Sobolev space, with norm
\begin{align}
	\| f \|_{W^{k,p}_{\mathrm{exp, \eta_-, \eta_+}}} = \| \omega_{\eta_-, \eta_+} f \|_{W^{k,p}}. 
\end{align}
When $k = 0$, we write $W^{0,p}_{\mathrm{exp},\eta_-, \eta_+} (\R) = L^p_{\mathrm{exp}, \eta_-, \eta_+} (\R)$ with corresponding notation for the norms. Ellipticity \eqref{e: ellipticity} guarantees that $B_0$ is closed with domain  $W^{2m,p}_{\mathrm{exp},\eta_-, \eta_+}$ on $L^p_{\mathrm{exp}, \eta_-, \eta_+} (\R)$.  We denote $\omega_{0, -\nu_0} =: \omega_0$, and let $\mcl_0 = \omega_0 B_0 \omega_0^{-1}$. 
\begin{hyp}[Fredholm properties]\label{hyp: fredholm properties}
	For all  $\eps > 0$ sufficiently small, we assume that 
	$B_0: H^{2m}_{\mathrm{exp}, 0,  -\nu_0 - \eps} \to L^2_{\mathrm{exp}, 0,  -\nu_0 - \eps}$  is Fredholm of index $i=+1$ with (minimal) one-dimensional kernel spanned by the derivative of the front $q_0'$. Furthermore, we assume that 
	$B_0: H^{2m}_{\mathrm{exp}, 0,  -\nu_0 + \eps} \to L^2_{\mathrm{exp}, 0,  -\nu_0 + \eps}$ is Fredholm of index $i=-1$ with trivial kernel and we denote by $\Phi$ a vector spanning the one-dimensional cokernel. We let $\phi = \omega_{0, -\nu_0} \Phi$ denote the corresponding basis for the cokernel of $\mcl_0$. 
\end{hyp}
Some comments are in order here. First, the relation between Freholm indices with weights $-\nu_0\pm\eps$ relates to the fact that $\nu_0$ is double as a root of the dispersion relation, so that the stronger exponential weight adds two boundary conditions at $+\infty$, eliminating solutions with asymptotics as given in Corollary~\ref{c: asymptotic solutions}. The assumption on the change of Fredholm index from $+1$ to $-1$ upon enforcing stronger exponential decay then implies that there are no further roots $d_{c_0}(0,ik;0)=0$ on the imaginary axis.  Second, rewriting the existence equation as a first-order ODE, we claim that this hypothesis guarantees a standard transversality of intersection  of unstable and strong stable manifolds. In fact, the Fredholm index of the differential operators is given by the difference in Morse indices at the asymptotic states in such a first-order ODE formulation; see \cite{ssmorse} for an account relevant to our situation. If we let $i_u$ denote the dimension of the unstable manifold of $u_-$ and $i_{ss}$ the dimension of the strong stable manifold at $0$ comprising solutions with decay at least $\mathrm{exp}((\nu_0+\eps)x)$, we find that the difference of associated Morse indices equals the Fredholm index, $i_u-(2mn-i_{ss})=1.$ Since the kernel is one-dimensional, by assumption, we conclude that these manifolds intersect transversely, that is, the dimension of the sum of their tangent spaces is $2mn$. Note that the Fredholm assumption implicitly implies hyperbolicity of $u_-$ and that there are no eigenvalues (or roots of the dispersion relation $d_{c_0}(0,\nu))$ with $\Re\nu=\nu_0$.

The dimension counting we presented thus far implies that fronts with exponential rate of decay no weaker than $\nu_0$ are robust. We shall however see that for those fronts a linear growth term $x\,\mathrm{exp}(\nu_0 x)$ is generically present with nonzero coefficient that depends smoothly on parameters; see Theorem~\ref{t: pulled pushed unfolding}, below.


\paragraph{Main results.} We are now ready to state our main results. Throughout we assume Hypothesis~\ref{h:sdr}  guaranteeing a simple double root with spatial decay rate $\nu_0$ at the linear spreading speed $c_0$, Hypothesis~\ref{hyp: front existence} on existence of a critical profile of a pulled front propagating with speed $c_0$ at $\mu=0$, with pure exponential asymptotics $\mathrm{exp}(\nu_0x)$ in the leading edge, and Hypothesis~\ref{hyp: fredholm properties} on Fredholm properties and minimality of spectrum at the origin. We also recall the robustness of linear spreading speeds $c_\mathrm{lin}(\mu)$ in the parameter $\mu$ from Lemma~\ref{l: robustness sPDR}.

\begin{thm}[pulled-to-pushed unfolding --- existence]\label{t: pulled pushed unfolding}
	For $\mu\sim 0$, there exist two  families of fronts $q_\mathrm{pl} (x; \mu)$ and $q_\mathrm{ps} (x; \mu)$,  stationary solutions to \eqref{e:pe} in a frame with speeds $c=c_\mathrm{pl}(\mu)=c_\mathrm{lin}(\mu)$ and $c=c_\mathrm{ps}(\mu)$, respectively.
	Both families $q_\mathrm{pl/ps}(\cdot,\mu)$ are continuous in $\mu$ in the local topology of $C^{2m}$, both bifurcate from $q_*$, that is, 
	$q_\mathrm{pl/ps}(x,0)=q_*(x)$ and converge to $u_-$ in their wake 
	\begin{align}
		q_\mathrm{pl/ps} (x; \mu) & =  u_-+\mathrm{O}(e^{\delta x}) ,  \quad  x \to -\infty,
	\end{align}
	for some $\delta>0$. 
	
	In the leading edge, we have double-root asymptotics for $q_\mathrm{pl}$,
	\begin{align}
		q_\mathrm{pl} (x; \mu) &= [\alpha(\mu) (u_0 (\mu) x + u_1 (\mu)) + \beta(\mu) u_0(\mu)] e^{\nu_\mathrm{lin} (\mu) x} + \mathrm{O}(e^{(\nu_0-\delta) x}) ,  \quad   x \to +\infty  
	\end{align}
	where $\beta(\mu) = 1 + \mathrm{O}(\mu)$ and $\alpha(\mu) = \mathrm{O}(\mu)$ are smooth, and $\delta>0$. We have pure exponential asymptotics for $q_\mathrm{ps}$,
	\begin{align}
		q_\mathrm{ps} (x; \mu) &= a(\mu) u^\mathrm{ps}_- (\mu, \sigma(\mu)) e^{\nu^\mathrm{ps}_- (\mu, \sigma(\mu)) x}+ \mathrm{O}(e^{(\nu_0-\delta) x}) , \quad x \to +\infty, 
	\end{align}
	for some smooth functions $\sigma(\mu) \in \R$, $\nu_-^\mathrm{ps}(\mu, \sigma) \in \R$, $u^\mathrm{ps}_-(\mu, \sigma) \in \R^n$, $a(\mu) = 1 + \mathrm{O}(\mu)$, and $\delta>0$.   
	
	Moreoever, if $\alpha'(0)\neq 0$, 
	then 
	\begin{align}
		c_\mathrm{ps}(\mu)= c_\mathrm{lin}(\mu) + c_2\mu^2+\mathrm{O}(\mu^3),\qquad c_2=\frac{\partial_{\nu\nu}d_{c_0}(0,\nu_0;0)}{2\partial_{\lambda}d_{c_0}(0,\nu_0;0)\nu_0} \alpha'(0)^2 >0. \label{e: pushed speed expansion}
	\end{align} 
\end{thm}

\begin{figure}
 \centering \includegraphics[width=0.9\textwidth]{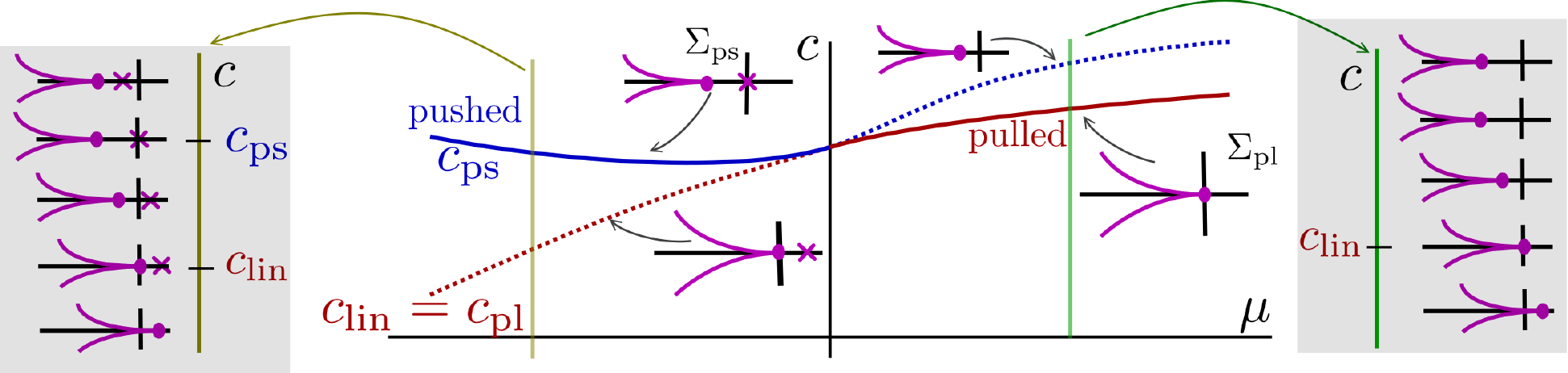}
 \caption{Center: Schematics of the pushed-to-pulled transition in the case $\alpha'(0)>0$: speeds $c_\mathrm{ps/pl}(\mu)$ with a quadratic tangency at $\mu=0$, pushed $c_\mathrm{ps}\geq c_\mathrm{pl}$. Solid lines show ``selected'' fronts, pushed and pulled, dashed lines the continuation of those selected fronts in parameter space. Insets show spectra in exponentially weighted spaces with pulled fronts marginally stable due to essential spectrum and pushed fronts marginally stable due to point spectrum. The continuation of pulled fronts is unstable against point spectrum, the continuation of pushed fronts is strongly stable. Left and Right: spectra of fronts as $c$ is increased for $\mu<0$ (left) and $\mu>0$ (right). Fronts at the linear speed are marginally stable for $\mu>0$ with marginally stable essential spectrum and fronts at the pushed speed are marginally stable for $\mu<0$ with marginally stable point spectrum.}\label{f:schem}
\end{figure}

As stated, the theorem continues the front $q_0$ preserving either the fact that the leading edge has double root asymptotics, or the fact that the front itself has pure exponential asymptotics. The significance of these two families becomes apparent when discussing selection of fronts. In fact, the families $q_\mathrm{pl}$ and  $q_\mathrm{ps}$ should be thought of as the smooth extension of families of pulled and pushed  fronts, respectively. The terminology of pulled and pushed front refers to their selection property, that is, one requires that open classes of initial data that vanish in $x>0$, say, converge to the fronts in a suitable topology. To clarify this interpretation, we inspect stability properties of the fronts identified in Theorem~\ref{t: pulled pushed unfolding}.

Consider therefore  the linearization about a  front $q_*$,
\begin{align}\label{e:lin2}
	B_* = \mathcal{P}(\partial_x) + c M \partial_x + f_u (q_*(\cdot; \mu); \mu) 
\end{align}
We write $B_\mathrm{pl/ps}(\mu)$  when $q_*=q_\mathrm{pl/ps}$ and $c=c_\mathrm{pl/ps}$ from Theorem~\ref{t: pulled pushed unfolding}. 
Recall the decay rate $\nu_\mathrm{dr}(\mu;c)$ associated with the pinched double root $\lambda_\mathrm{dr}(\mu;c)$ and define 
\begin{equation}
 \mathcal{L}_\mathrm{pl/ps}(\mu) = \omega_{0, -\nu_\mathrm{dr}(\mu;c_\mathrm{pl/ps}(\mu))} B_\mathrm{pl/ps}(\mu) \omega_{0, -\nu_\mathrm{dr}(\mu;c_\mathrm{pl/ps}))}
\end{equation}
the linearization conjugated with the critical exponential weight, in which the essential spectrum is pushed as far left as possible. 

We focus on the spectrum $\Sigma_\mathrm{pl/ps}(\mu)$ of $\mathcal{L}_\mathrm{pl/ps}(\mu)$ in a small ball centered at the origin $|\lambda|\leq \delta$. We distinguish between point spectrum $\Sigma^\mathrm{pt}_\mathrm{pl/ps}(\mu)$, where $\mathcal{L}_\mathrm{pl/ps}(\mu)-\lambda M$ is Fredholm index 0 but not invertible, and the essential spectrum $\Sigma^\mathrm{ess}_\mathrm{pl/ps}(\mu)$, where 
 $\mathcal{L}_\mathrm{pl/ps}(\mu)-\lambda M$ is not Fredholm of index 0.

\begin{thm}[pulled-to-pushed unfolding --- stability]\label{t: pulled pushed stability} Consider the spectra $\Sigma_\mathrm{pl/ps}(\mu)$  of $\mathcal{L}_\mathrm{pl/ps}(\mu)$ in a $\delta$-neighborhood of the origin. Recall the leading-order linear term $\alpha(\mu)$ in the leading edge of $q_\mathrm{pl}$ and assume  $\alpha'(0) \neq 0$. Then, for $\delta>0$ and $|\mu|$ sufficiently small, we have
\begin{align*}
 \Re\Sigma^\mathrm{ess}_\mathrm{pl}(\mu)\leq 0, \quad 0\in \Sigma^\mathrm{ess}_\mathrm{pl}(\mu), \qquad \qquad \text{and} \quad 
 \Re\Sigma^\mathrm{ess}_\mathrm{ps}(\mu)< 0 \text{ for } |\mu|\neq 0.
\end{align*}
Moreover, 
\begin{itemize}
 \item for $\alpha'(0)\mu>0$, we have $\Sigma^\mathrm{pt}_\mathrm{pl}(\mu)=\emptyset$, $\Sigma^\mathrm{pt}_\mathrm{ps}(\mu)=\emptyset$;
 \item for $\alpha'(0)\mu<0$, we have $\Sigma^\mathrm{pt}_\mathrm{pl}(\mu)\ni \lambda(\mu)>0$, $\Sigma^\mathrm{pt}_\mathrm{ps}(\mu)=\{0\}$.
\end{itemize}
In particular, fronts $q_\mathrm{ps}$ are marginally stable precisely when  $\alpha'(0)\mu<0$ and fronts  $q_\mathrm{pl}$ are marginally stable precisely when $\alpha'(0)\mu>0$; fronts $q_\mathrm{pl}$ are unstable when the leading edge behavior is non-monotone. 
\end{thm}

%
%
%

\begin{remark} 
The fronts $q_\mathrm{ps}$ possess a resonance at $\lambda=0$ when $\alpha'(0)\mu>0$, that is, a pole of the pointwise resolvent when considered on a Riemann surface with branch cut at the double root. A characterization of the significance of fronts with such a resonance at the origin does not appear to be known.
\end{remark}

The \emph{marginal stability conjecture}, often loosely formulated, states that marginal stability in the leading edge implies selection of fronts, that is those fronts attract open sets of initial conditions including functions with support in $\{x<0\}$. Assuming that the marginal stability conjecture holds and absence of spectrum in $\{\Re\lambda\geq0,|\lambda|>\delta\}$, we can therefore predict that $q_\mathrm{pl}$ is \emph{selected} when $\alpha'(0)\mu>0$ and $q_\mathrm{ps}$ is \emph{selected} when $\alpha'(0)\mu<0$. 
We summarize this conclusion in the following ``result''. 

\textbf{Result --- pushed-to-pulled transition.} \emph{Assume Hypotheses~\ref{h:sdr}--\ref{hyp: fredholm properties} and in addition absence of spectrum of $\mathcal{L}_\mathrm{pl/ps}(0)$ in $\{ \lambda \geq 0 \}\setminus \{0\}$. Assume in addition that the marginal stability conjecture holds. Then we have propagation at the linear spreading speed for monotone tails, $\alpha'(0)\mu>0$, with ``selection'' of the pulled front $q_\mathrm{pl}(\mu)$, and selection of the pushed front $q_\mathrm{ps}(\mu)$ with speed $c_\mathrm{ps}(\mu)>c_\mathrm{lin}(\mu)$ for $\alpha'(0)\mu<0$.}

In the case of pushed fronts, the marginal stability conjecture generally can be established with standard methods. The linearization possesses a simple eigenvalue at the origin associated with translations in a weighted space that allows for perturbations that cut off the front tail. For pulled fronts, the marginal stability conjecture was known in systems with comparison principles starting with \cite{Kolmogorov} and only recently established in a conceptual framework, based only on linear information as provided here, albeit only for scalar higher-order parabolic equations \cite{averyscheelselection}. 

\paragraph{Acknowledgements.} The material here is based on work supported by the National Science Foundation, through through the Graduate Research Fellowship Program under Grant No. 00074041 (MA), as well as NSF-DMS-2007759 (MH) and NSF-DMS-1907391 (AS). Any opinions, findings, and conclusions or recommendations
expressed in this material are those of the authors and do not necessarily reflect the views of the
National Science Foundation.

\section{Preliminaries}\label{s: preliminaries}

\subsection{Double root criteria and robustness of double roots}

We start by providing a reformulation of Hypothesis~\ref{h:sdr} without relying on determinants, preparing also for the proof of robustness, Lemma~\ref{l: robustness sPDR}, and tail expansions, Corollary~\ref{c: asymptotic solutions}. Recall the definition of the family of matrices associated with the leading edge, 
\[	A(\lambda,\nu, c; \mu) = \mathcal{P}(\nu) + c \nu M + f_u (0; \mu) - \lambda M. 
\]

\begin{hyp}[Simple double root]\label{hyp: double root}
	There exist $\nu_0 < 0$, $c_0 > 0$, and $u_0^0$ and $u_1^0 \in \R^n$ such that 
	\begin{align}
		A(0,\nu_0, c_0;  0) u_0^0 &= 0, \label{e:g1}\\
		\partial_\nu A(0,\nu_0, c_0; 0) u_0^0 + A(0,\nu_0, c_0; 0) u_1^0 &= 0. \label{e:g2}
	\end{align}
	We let 
	\[A^{00} = A(0,\nu_0, c_0;  0),\  A^{01} = \partial_\nu A(0,\nu_0, c_0;  0), \  A^{10}=\partial_\lambda A(0,\nu_0, c_0;  0), \  A^{02}=\frac{1}{2}\partial_{\nu \nu} A(0,\nu_0, c_0;  0).
	 	\]
 We then assume that  $\ker A^{00} = \spn (u_0^0)$, we let $\ker (A^{00})^T = \spn (e_\mathrm{ad})$, and we assume that
	\begin{align}
		\langle A^{10} u_0^0, e_\mathrm{ad} \rangle \langle A^{02} u_0^0 + A^{01} u_1^0, e_\mathrm{ad} \rangle < 0. \label{e: generic sn assumption}
	\end{align}
\end{hyp}
In particular, $A^{10} u_0^0 = - M u_0^0$ and $A^{02} u_0^0 + A^{01} u_1^0$ are not in the range of $A^{00}$ since both the projections in \eqref{e: generic sn assumption} are nonzero. 

\begin{lemma}\label{l: disp gen evp equivalent}
 Hypothesis~\ref{hyp: double root} and Hypothesis~\ref{h:sdr} are equivalent.
\end{lemma}
\begin{proof}
    Both formulations express algebraic multiplicities of the eigenvalue zero when $A$ is considered as a matrix pencil in $\lambda$ with $\nu=0$ or a matrix pencil in $\nu$ with $\lambda=0$. Hypothesis~\ref{h:sdr} expresses these multiplicities as orders of the roots of the determinant and Hypothesis~\ref{hyp: double root} as lengths of Jordan chains. Both characterizations agree; see for instance \cite{gohberg2006invariant}. It remains to show that $\partial_\lambda d\partial_{\nu\nu}d<0$ is equivalent to \eqref{e: generic sn assumption}. This in turn follows from a direct computation using Lyapunov-Schmidt reduction to find all values of $\lambda,\nu$ where $A$ has a kernel. The determinant criterion gives this via the expansion $\partial_\lambda d \cdot \lambda + \frac{1}{2}\partial_{\nu\nu}d\cdot \nu^2+\mathrm{O}(\lambda^2,\lambda(\nu-\nu_0),(\nu-\nu_0)^3)=0$. Directly from the matrix kernel, we find 
    $\langle A^{10} u_0^0, e_\mathrm{ad} \rangle \lambda + \frac{1}{2} \langle A^{02} u_0^0 + A^{01} u_1^0, e_\mathrm{ad} \rangle \nu^2 + \mathrm{O}(\lambda^2,\lambda(\nu-\nu_0),(\nu-\nu_0)^3)=0$, establishing our claim. 
\end{proof}
%

\begin{remark} The sign in \eqref{e: generic sn assumption} implies an effective positive diffusivity, when interpreting the expansion $\lambda-(\nu-\nu_0)^2$ of the dispersion relation as stemming from a diffusion equation with exponential weight $\nu_0$. There appear to be no known examples where the most unstable double root has a negative effective diffusivity in this sense. Positive effective diffusivity also implies stability of the absolute spectrum, which governs stability in large bounded domains \cite{AbsoluteSpecArndBjorn}, in a neighborhood of the double root. We caution however that, conversely, stability of double roots and positive effective diffusivity does not imply stability of absolute spectra and refer to \cite{FayeHolzerScheelSiemer} for analysis and a discussion of invasion phenomena in this context. 
\end{remark}

We next establish robustness of double roots, using the formulation from Hypothesis~\ref{hyp: double root}. As an additional benefit, using this approach also provides us with a continuation of the eigenvectors $u_0$ and $u_1$. 

\begin{proof}[Proof of Lemma~\ref{l: robustness sPDR}]
    We only prove the second claim, pinning the double root at the origin for a suitable speed $c_\mathrm{lin}$. The first claim is easier to establish, solving for instance $d_c(\lambda,\nu;\mu)=\partial_\nu d_c(\lambda,\nu;\mu)=0$ with the implicit function theorem for $(\lambda,\nu)$. 
    
	Define $F: \R^n \times \R^n \times \R^2 \times \R \to \R^{2n+2}$ by
	\begin{align}\label{e:dismatrix}
		F(u_0, u_1, c, \nu; \mu) = \begin{pmatrix}
			A(0,\nu, c; \mu) u_0 \\
			\partial_\nu A(0,\nu, c; \mu) u_0 + A(0,\nu, c; \mu)u_1 \\
			\langle u_0 - u_0^0, u_0^0 \rangle \\
			\langle u_1 - u_1^0, u_0^0 \rangle 
		\end{pmatrix}
	\end{align}
	Note that by Hypothesis~\ref{hyp: double root}, $F(u_0^0, u_1^0, c_0, \nu_0; 0) = 0$. Linearizing about this solution, we find
	\begin{align*}
		D_{(u_0, u_1, c, \nu)} F(u_0^0, u_1^0, c_0, \nu_0; 0) = \begin{pmatrix}
			A^{00} & 0 & \partial_c A(\star) u_0^0 & A^{01} u_0^0 \\
			A^{01} & A^{00} & \partial_c \partial_\nu A(\star) u_0^0 + \partial_c A(\star) u_1^0 & 2 A^{02}u_0^0 + A^{01} u_1^0 \\
			\langle \cdot, u_0^0 \rangle & 0 & 0 & 0 \\
			0 & \langle \cdot, u_0^0 \rangle & 0 & 0
		\end{pmatrix},
	\end{align*}
	where $(\star) = (u_0^0, u_1^0, c_0, \nu_0; 0)$. Note that $\partial_c A (\star) = \nu_0 M$, $\partial_c \partial_\nu A(\star) = M$ and by Hypothesis~\ref{hyp: double root} $A^{01} u_0^0 = - A^{00} u_1^0$, so that this expression simplifies to 
	\begin{align*}
		D_{(u_0, u_1, c, \nu)} F(u_0^0, u_1^0, c_0, \nu_0; 0) = \begin{pmatrix} A^{00} & 0 & \nu_0 M u_0^0 & - A^{00}u_1^0 \\
			A^{01} & A^{00} & M u_0^0 + \nu_0 M u_1^0 & 2 A^{02}u_0^0 + A^{01}u_1^0 \\
			\langle \cdot, u_0^0 \rangle & 0 & 0 & 0 \\
			0 & \langle \cdot, u_0^0 \rangle & 0 & 0 
		\end{pmatrix}. 
	\end{align*}
	Assume $(w_0, w_1, \tilde{c}, \tilde{\nu}) \in \ker D_{(u_0, u_1, c, \nu)} F(u_0^0, u_1^0, c_0, \nu_0; 0)$, so that $(w_0, w_1, \tilde{c}, \tilde{\nu})$ satisfy the following system
	\begin{align*}
		A^{00} w_0 + \tilde{c} \nu_0 M u_0^0 - \tilde{\nu} A^{00}u_1^0 &= 0 \\
		A^{01} w_0 + A^{00} w_1 + \tilde{c} (Mu_0^0 + \nu_0 M u_1^0) + \tilde{\nu} (2 A^{02} u_0^0 + A^{01} u_1^0) &= 0, \\
		\langle w_0, u_0^0 \rangle &= 0, \\
		\langle w_1, u_0^0 \rangle &= 0. 
	\end{align*}
	In particular, from the first equation we have $A^{00} (w_0 - \tilde{\nu} u_1^0) = - \tilde{c} \nu_0 M u_0^0$. However, by Hypothesis~\ref{hyp: double root}, $M u_0^0 \notin \rg A^{00}$, so we must have $\tilde{c} = 0$, from which we again use Hypothesis~\ref{hyp: double root} to conclude that $w_0 - \tilde{\nu} u_1^0 = \alpha u_0^0$ for some $\alpha \in \R$. We may then rewrite the second equation as 
	\begin{align*}
		A^{01} (\alpha u_0^0) + A^{00} w_1 = - 2\tilde{\nu} ( A^{02} u_0^0 + A^{01} u_1^0 ). 
	\end{align*}
	By Hypothesis~\ref{hyp: double root}, we have $A^{01} u_0^0 = - A^{00} u_1^0$, and hence 
	\begin{align*}
		A^{00} (w_1 - \alpha u_1^0) = - 2 \tilde{\nu} ( A^{02} u_0^0 + A^{01} u_1^0 ). 
	\end{align*}
	Since by Hypothesis~\ref{hyp: double root}, the right hand side is not in the range of $A^{00}$, we conclude that $\tilde{\nu} = 0$, from which it follows that $w_1 - \alpha u_1^0 = \beta u_0^0$ for some $\beta \in \R$, and $w_0 = \alpha u_0^0$. However, since $\langle w_0, u_0^0 \rangle = 0$, we conclude $\alpha = 0$, and then we have $w_1 = \beta u_0^0$, with $\langle w_1, u_0^0 \rangle = 0$, and so $\beta = 0$ as well. Hence the kernel of the linearization is trivial, and the result follows from the implicit function theorem. 
\end{proof}

\subsection{Projections of tail corrections}
Recall the definition of the linearization $\mathcal{L}_0$ after conjugation with exponential weights, characterized in Hypothesis~\ref{hyp: fredholm properties} and the definition of $\varphi$, there, as a basis of the cokernel. Let $\chi_+$ be a smooth, positive cutoff function satisfying 
\begin{align*}
	\chi_+ (x) = \begin{cases} 1, & x \geq 3, \\
		0, & x \leq 2. 
	\end{cases}
\end{align*}

\begin{lemma}[Projections]\label{l: projections}
	We have 
	\begin{align}
		\langle \mcl_0 (u_0^0 \chi_+), \phi \rangle = 0,  \label{e: projection 1}
	\end{align}
	while 
	\begin{align}
		\langle \mcl_0 [(u_0^0 x + u_1^0)\chi_+], \phi \rangle \neq 0. \label{e: projection 2}
	\end{align}
\end{lemma}
\begin{proof}
	First we prove \eqref{e: projection 1}. We start by rewriting the equation $\mcl_0 u = 0$ with the far-field/core ansatz $u = w + \beta u_0^0 \chi_+$, where $w \in L^2_{\mathrm{exp}, 0, \eta}$ for $\eta$ small. We then let $P$ be the orthogonal projection onto the range of $\mcl_0$ in $L^2_{\mathrm{\exp}, 0, \eta}$, and decompose the resulting equation for $(w, \beta)$ as 
	\begin{align}
		\begin{cases}
			P \mcl_0 (w+\beta u^{0}_0 \chi_+) & = 0, \\
			\langle \mcl_0 (w + \beta u^0_0 \chi_+), \phi \rangle &= 0. 
		\end{cases} \label{e: projection 1 system}
	\end{align} 
	 Rewriting the first equation as $P\mcl_0 w = - \beta P\mcl_0 (u_0^0 \chi_+)$, and exploiting that $-\beta P\mcl_0 (u_0^0 \chi_+) \in L^2_{\mathrm{exp}, 0, \eta} (\R)$ and that $P \mcl_0 : H^{2m}_{\mathrm{exp}, 0, \eta} (\R)\subset L^2_{\mathrm{\exp}, 0, \eta} \to L^2_{\mathrm{\exp}, 0, \eta}$ is invertible by construction, we may solve the first equation for $w = w(\beta)$. The full system therefore has a solution $(w(\beta), \beta)$ if and only if $\langle \mcl_0 (w + \beta u^0_0 \chi_+), \phi \rangle = 0$. Since $w$ is exponentially localized and $\phi$ is in the kernel of $\mcl_0^*$, we have $\langle \mcl_0 w, \phi \rangle = \langle w, \mcl_0^* \phi \rangle = 0$. Hence the system \eqref{e: projection 1 system} has a solution $(w, \beta)$ if and only if $\beta \langle \mcl_0 (u^0_0 \chi_+), \phi \rangle= 0$. However, Hypothesis~\ref{hyp: front existence} gives us a solution to this equation: by translational invariance of the original equation, we have $\mcl_0 (\omega_0 q_0') = 0$. Defining $\beta = \nu_0$ and $w = \omega_0 q_0' - \nu_0 u^0_0 \chi_+$ then gives a solution to \eqref{e: projection 1 system}. Since $\beta = \nu_0 \neq 0$, we conclude that $\langle \mcl (u_0^0 \chi_+), \phi \rangle = 0$, as desired. 
	
	To prove \eqref{e: projection 2}, we modify the far-field core ansatz to incorporate the linearly growing solution captured in Corollary~\ref{c: asymptotic solutions}, writing 
	\begin{align}
		\mcl_0 [w + \alpha (u_0^0 x + u_1^0) \chi_+] = 0. \label{e: projection 2 far-field core}
	\end{align}
	Corollary~\ref{c: asymptotic solutions} guarantees that the result of the left hand side is in $L^2_{\exp, 0, \eta}$, so we again decompose this equation as 
	\begin{align}
		\begin{cases}
			P \mcl_0 [w + \alpha (u_0^0 x + u_1^0) \chi_+ ] &= 0, \\
			\langle \mcl_0 [w + \alpha (u_0^0 x + u_1^0)\chi_+], \phi \rangle & = 0. 
		\end{cases} \label{e: projection 2 system}
	\end{align}
	At $\alpha = 0$, the first equation has the trivial solution $w = 0$. The linearization in $w$ at this solution is $P \mcl_0$, which is invertible on $L^2_{\mathrm{\exp}, 0, \eta}$ by construction, so by the implicit function theorem we find a solution $w(\alpha)$ for $\alpha$ small. Linearity in $\alpha$ and $w$ together with uniqueness of the solution found from the implicit function theorem implies that we may write this solution as $w(\alpha) = \alpha \tilde{w}$ for some $\tilde{w} \in L^2_{\mathrm{\exp}, 0, \eta}$. We may then insert this into the second equation, and find that the system \eqref{e: projection 2 system} has a solution if and only if 
	\begin{align*}
		\alpha \langle \mcl_0 [\tilde{w} + (u_0^0 x + u_1^0) \chi_+], \phi \rangle = 0. 
	\end{align*}
	Again, since $\tilde{w} \in L^2_{\mathrm{\exp}, 0, \eta}$, we have $\langle \mcl_0 \tilde{w}, \phi \rangle = \langle \tilde{w}, \mcl_0^* \phi \rangle = 0$. Hence \eqref{e: projection 2 system} has a solution if and only if $\alpha \langle \mcl_0 [(u_0^0 x + u_1^0) \chi_+], \phi \rangle = 0$. 
	
	We claim that \eqref{e: projection 2 system} has no nontrivial solutions by Hypotheses~\ref{hyp: front existence} and~\ref{hyp: fredholm properties}. Any nontrivial solution would give rise to a solution to \eqref{e: projection 2 far-field core} which is either linearly growing at $+ \infty$ (if $\alpha \neq 0$) or exponentially localized (if $\alpha = 0$). Such a solution would be a solution to $\mcl_0 u = 0$, which is in $L^2_{\mathrm{exp}, 0, -\eta}$, which is linearly independent from $\omega_0 q_*'$, which is excluded by Hypothesis~\ref{hyp: fredholm properties}. Hence \eqref{e: projection 2 system} has no nontrivial solutions. On the other hand, if we had $\langle \mcl_0 [(u_0^0 x + u_1^0) \chi_+], \phi \rangle = 0$, we would obtain a family of solutions of \eqref{e: projection 2 system} for $\alpha \neq 0$, small. Hence we must have $ \langle \mcl_0 [(u_0^0 x + u_1^0) \chi_+], \phi \rangle \neq 0$, as desired. 
\end{proof}

\subsection{Expansions of spatial eigenvalues and eigenspaces}

Fixing $\lambda=0$, we have an eigenvalue of multiplicity 2 in $\nu$ at the origin and an associated Jordan block. Varying $\lambda$, this double eigenvalue splits and eigenvalues and eigenspaces need to be carefully expanded after passing to a Riemann surface. 
\begin{lemma}[Saddle node of eigenspaces]\label{l: saddle node}
	Fix $\mu$ small, and let $\gamma = \sqrt{\lambda}$ with branch cut along the negative real axis. The equation $A(\lambda, \nu, c_\mathrm{lin}(\mu); \mu) u = 0$ has precisely two solutions $(\nu^\mathrm{pl}_\pm(\mu, \gamma), u^\infty_\pm (\mu, \gamma))$ (up to a constant multiple of the $u$ component) for $\nu$ close to $\nu_0$ and $\lambda$ close to zero, with expansions
	\begin{align}
		\nu^\mathrm{pl}_\pm (\mu, \gamma) &= \nu_\mathrm{lin}(\mu) \pm \sqrt{-d_{10}d_{02}^{-1}} \gamma + \mathrm{O}(\gamma^2), \\
		u^\infty_\pm (\mu, \gamma) &= u_0(\mu) \pm \sqrt{-d_{10}d_{02}^{-1}} \left(\langle u_0, u_1 \rangle u_0 + u_1 \right) \gamma + \mathrm{O}(\gamma^2),
	\end{align}
	with remainder terms uniformly small in $\mu$, and $d_{10},d_{02}$ as in \eqref{e:dij}.
\end{lemma}
\begin{proof}
	We let $A_\mathrm{\mu} (\lambda, \nu) = A( \lambda,\nu, c_\mathrm{lin}(\mu); \mu)$, and define 
	\begin{align}
		F(u, \nu; \lambda) = \begin{pmatrix}
			A_\mu (\lambda, \nu_\mathrm{lin}(\mu) + \nu ) (u_0(\mu) + u) \\
			|u_0+u|^2 -1
		\end{pmatrix}.
	\end{align}
	We look for solutions to $F(u, \nu; \lambda) = 0$. The second equation is a normalization condition: since the first equation is linear in $u_0(\mu) + u$, we need to adjoin with a condition that fixes the constant multiple. From now on, we suppress the dependence on $\mu$. By construction, $F(0, 0; 0) = 0$, and we compute the linearization
	\begin{align}
		D_{(u, \nu)} F(0, 0; 0) = \begin{pmatrix}
			A_\mu^{00} & A_\mu^{01} u_0 \\
			2 \langle \cdot, u_0 \rangle & 0 
		\end{pmatrix},
	\end{align}
	where $A_\mu^{00}=A_\mu(0, \nu_\mathrm{lin}(\mu))$ and $A_\mu^{01}=\partial_\nu A_\mu(0, \nu_\mathrm{lin}(\mu))$.
	From a short computation, we see that $D_{(u,\nu)} F(0,0;0)$ has a one dimensional kernel spanned by
	\begin{align}
		w_0 = (\langle u_1, u_0 \rangle u_0 + u_1, 1 ). 
	\end{align}
	We therefore perform a Lyapunov-Schmidt reduction. Let $Q$ denote the orthogonal projection in $\C^{n+1}$ onto the range of $D_{(u,\nu)} F(0,0;0)$. We let $w = (u, \nu)$, and split the solution as $w = w_c + w_h = (u_c + u_h, \nu_c + \nu_h)$, with $w_c \in \ker (D_{(u,\nu)} F(0,0;0))$ and $w_h \in (\ker (D_{(u,\nu)} F(0,0;0)))^\perp$. Our system then becomes 
	\begin{align}
		\begin{cases}
			QF (u_c + u_h, \nu_c + \nu_h; \lambda) &= 0 \\
			(I-Q) F(u_c + u_h, \nu_c + \nu_h; \lambda) &=0. 
		\end{cases}
	\end{align}
	The linearization of the first equation with respect to $w_c$ is $QD_{(u, \nu)} F(0, 0; 0)\big|_{(\ker (D_{(u,\nu)} F(0,0;0)))^\perp}$, which is invertible by construction. We therefore solve the first equation with the implicit function theorem for $w_h (w_c;\lambda) = \mathrm{O}(|\lambda| + |w_c|^2 + |\lambda||w_c|)$. 
	
	To compute the reduced second equation, we find $I-Q$ explicitly by solving for the cokernel of $D_{(u,\nu)} F(0,0;0)$. Indeed, we have
	\begin{align}
		D_{(u, \nu)} F(0,0;0)^* = \begin{pmatrix}
			(A_\mu^{00})^T & 2 u_0  \\
			\langle \cdot, A^{01}_\mu u_0 \rangle & 0. 
		\end{pmatrix}
	\end{align}
	The kernel of $(A^{00}_\mu)^T$ is one dimensional, spanned by a vector $e_\mathrm{ad}(\mu)$. Notice that by Lemma~\ref{l: robustness sPDR} and the formulation of simple pinched double roots in Hypothesis~\ref{hyp: double root}, we have $\langle e_\mathrm{ad}, A_\mu^{01} u_0 \rangle +\langle e_\mathrm{ad}, A_\mu^{00} u_0 \rangle =0$, hence
	\begin{align}
		\langle e_\mathrm{ad}, A_\mu^{01} u_0 \rangle = -\langle e_\mathrm{ad}, A_\mu^{00} u_0 \rangle = - \langle (A_\mu^{00})^T e_\mathrm{ad}, u_0 \rangle = 0,
	\end{align}
	and therefore $(e_\mathrm{ad}, 0)$ spans the kernel of $D_{(u, \nu)} F(0,0;0)^*$. 
	
	Choosing coordinates on the kernel, we let $(u_c, \nu_c) = \alpha (w_0^u, w_0^\nu)$. From a short computation, we find 
	\begin{multline}
		(I-Q) F(\alpha w_0^u + u_h (\alpha w_0; \lambda), \alpha w_0^\nu + \nu_h (\alpha w_0; \lambda); \lambda) \\ = - \lambda \langle M u_0, e_\mathrm{ad} \rangle + \left( \left \langle A^{01}_\mu (\langle u_1, u_0 \rangle u_0 + u_1) + A^{02}_\mu u_2, e_\mathrm{ad} \right \rangle \right) \alpha^2 + \mathrm{O}(|\lambda|^2, |\alpha|^3, |\alpha||\lambda|),
	\end{multline}
	where $A_\mu^{02} = \partial_{\nu \nu} A_\mu (0, \nu_\mathrm{lin}(\mu))$. Note that, as above, $\langle A_\mu^{01} u_0, e_\mathrm{ad} \rangle = 0$. Hence we obtain the reduced equation
	\begin{align}
		0 =  d_{10} \lambda  + d_{02}  \alpha^2 + \mathrm{O}(|\lambda|^2, |\alpha| |\lambda|, |\alpha|^3)
	\end{align}
	where
	\begin{align}\label{e:dij}
		d_{02} = \langle A^{01}_\mu u_1 + A^{02}_\mu u_2, e_\mathrm{ad} \rangle, \quad d_{10} = \langle -Mu_0, e_\mathrm{ad} \rangle. 
	\end{align}
	Solving with the Newton polygon, we find unique solutions
	\begin{align}
		\alpha (\gamma) = \pm \sqrt{-d_{10}d_{02}^{-1}  } \gamma + \mathrm{O}(\gamma^2), 
	\end{align}
	for $\lambda = \gamma^2$, with $\Re \lambda$ to the right of the critical dispersion curve. Returning to $\nu_c = \alpha w_0^c$, we find
	\begin{align}
		\nu_c^\pm (\gamma) = \pm \sqrt{-d_{10} d_{02}^{-1}} \gamma + \mathrm{O}(\gamma^2). 
	\end{align}
	Similarly,
	\begin{align}
		u_c^\pm = \alpha w_0^u = \pm \sqrt{-d_{10}d_{02}^{-1}} \left(\langle u_0, u_1 \rangle u_0 + u_1 \right) \gamma + \mathrm{O}(\gamma^2). 
	\end{align}
	Since $u_h$ and $\nu_h$ are higher order, this proves the desired expansions, with $u^\infty_\pm = u_0 + u_c^\pm$ and $\nu^\mathrm{pl}_\pm = \nu_\mathrm{lin} + \nu^c_\pm$.  
\end{proof}

\subsection{Expansions of pulled fronts in the leading edge}

We show how the fact that the symbol $A$ has a double root in $\nu$ translates into asymptotics of pulled fronts, Corollary~\ref{c: asymptotic solutions}.

\begin{proof}[Proof of Corollary ~\ref{c: asymptotic solutions}]
 We suppress the arguments $\lambda,c,\mu$ in $A$ and write $A(\nu_0+\nu)=A^0+A^1\nu+\mathrm{O}(\nu^2)$. We need to show that 
 \[
  A(\partial_x)u_0^0 e^{\nu_0 x}=0, \qquad A(\partial_x)(u_0^0x+u_1^0) e^{\nu_0 x}=0.
 \]
 This turns out being equivalent to 
 \[
  A^0u_0^0=0,\qquad (A^0+A^1\partial_x)(u_0^0x+u_1^0)=0,
 \]
which is precisely encoded in \eqref{e:g1} and \eqref{e:g2}.
\end{proof}

\section{Pulled unfolding}\label{s: pulled unfolding}

We prove the claims on $q_\mathrm{pl}$ from Theorem~\ref{t: pulled pushed unfolding}. We start by inserting the far-field core ansatz
\begin{align}
	u(x) = u_- \chi_- (x) + w(x) + \chi_+ [\alpha (u_0(\mu)x + u_1 (\mu)) + \beta u_0 (\mu)] e^{\nu_\mathrm{lin}(\mu) x}
\end{align}
into the traveling wave equation
\begin{align}
	(\mathcal{P}(\partial_x) + c_\mathrm{lin}(\mu) M \partial_x) u + f(u; \mu) = 0,
\end{align}
obtaining an equation
\begin{align}
	F^{\mathrm{pl}}(w; \alpha, \beta, \mu) := (\mathcal{P}(\partial_x) + c_\mathrm{lin}(\mu) M \partial_x) (u_- \chi_- + w + \chi_+ \psi) + f(u_- \chi_- + w + \chi_+ \psi; \mu) = 0,
\end{align}
where 
\begin{align}
	\psi(x; \alpha, \beta, \mu) =  [\alpha (u_0(\mu)x + u_1 (\mu)) + \beta u_0 (\mu)] e^{\nu_\mathrm{lin}(\mu) x}. \label{e: pulled unfolding psi def}
\end{align}
Fix $\eps > 0$ small and let $\eta_0 = -\nu_0 + \eps$. We require $w$ to be faster decaying than $e^{\nu_0 x}$, so we consider $F^{\mathrm{pl}}$ as a function $F^{\mathrm{pl}} : H^{2m}_{\mathrm{exp}, 0, \eta_0} \times \R^2 \times (-\mu_0, \mu_0) \to L^2_{\mathrm{exp}, 0, \eta_0}$ for some $\mu_0$ small. 

\begin{lemma}
	The function $F^{\mathrm{pl}} : H^{2m}_{\mathrm{exp}, 0, \eta_0} \times \R^2 \times (-\mu_0, \mu_0) \to L^2_{\mathrm{exp}, 0, \eta_0}$ is well defined and smooth in all variables.
\end{lemma}
\begin{proof}
	The fact that $F^{\mathrm{pl}}$ maps into $L^2_{\mathrm{exp},0, \eta_0} (\R)$ and hence is well defined follows from the fact that $f(u_-; \mu) = 0$, and Corollary~\ref{c: asymptotic solutions}, which guarantee that the far-field terms in the ansatz satisfy the traveling wave equation asymptotically. Smoothness follows from the  fact that $H^1_{\mathrm{\exp,0, \eta_0}} (\R)$ is a Banach algebra. 
\end{proof}

Note that at $\mu = 0$, we have $F(w_0; 0, 1, 0) = 0$, where 
\begin{align}
	w_0 (x) = q_0(x) - u_- \chi_- (x) - \psi(x;0,1,0) \chi_+ (x). \label{e: pulled unfolding w0}
\end{align}
By Hypothesis~\ref{hyp: fredholm properties}, the linearization $D_w F(w_0; 0, 1, 0) = \mcl_0$ is Fredholm with index -1. By the Fredholm bordering lemma, the joint linearization $D_{(w,\alpha, \beta)} F^{\mathrm{pl}}(w_0; 0, 1, 0)$ is Fredholm with index 1, so we augment this system with a phase condition, defining
\begin{align}
	G^{\mathrm{pl}}(w; \alpha, \beta, \mu) = \begin{pmatrix}
		F^{\mathrm{pl}}(w; \alpha, \beta, \mu) \\
		\langle w, e_0 \rangle - \langle w_0, e_0 \rangle
	\end{pmatrix}
\end{align}
where $e_0 \in L^2_{\mathrm{exp}, 0, \eta_0} (\R)$ is a fixed localized function chosen such that 
\begin{equation}\label{e:e_0}
\langle q_0' - \nu_0 u_0^0 \chi_+ e^{\nu_0 \cdot}, e_0 \rangle \neq 0.
\end{equation}
The Fredholm bordering lemma implies that $D_{(w, \alpha, \beta)} G^{\mathrm{pl}}(w_0; 0, 1, 0)$ is Fredholm with index zero, and from a short computation, we find
\begin{align*}
	D_{(w,\alpha, \beta)} G(w_0; 0, 1, 0) = \begin{pmatrix}
		B_0 & B_0 [(u_0^0 x + u_1^0) \chi_+e^{\nu_0x}] & B_0 (u_0^0 \chi_+ e^{\nu_0 x}) \\
		\langle \cdot, e_0 \rangle &0 &0 
	\end{pmatrix}. 
\end{align*}

\begin{prop}
	The linear operator \[D_{(w,\alpha, \beta)} G^{\mathrm{pl}}(w_0, 0, 1, 0) : H^{2m}_{\mathrm{exp}, 0, \eta_0} \times \R^2 \times (-\mu_0, \mu_0) \to L^2_{\mathrm{exp}, 0, \eta_0} (\R) \times \R\] is invertible. 
\end{prop}
\begin{proof}
	Since $D_{(w,\alpha, \beta)} G^{\mathrm{pl}}(w_0, 0, 1, 0)$ is Fredholm index zero, it suffices to prove that the kernel of this operator is trivial. Suppose 
	\begin{align*}
		\begin{pmatrix}
			B_0 & B_0 [(u_0^0 x + u_1^0) \chi_+e^{\nu_0x}] & B_0 (u_0^0 \chi_+ e^{\nu_0 x}) \\
			\langle \cdot, e_0 \rangle &0 &0 
		\end{pmatrix} \begin{pmatrix} v \\ a \\ b \end{pmatrix}
	\end{align*}
	for some $(v,a,b)^T \in H^{2m}_{\mathrm{exp}, 0, \eta_0} \times \R^2$. Then, in particular,
	\begin{align*}
		B_0 [v + a (u_0^0 x + u_1^0)e^{\nu_0 x} + bu_0 e^{\nu_0 x}] = 0. 
	\end{align*}
	By Hypothesis~\ref{hyp: front existence} and translation invariance, we have that $B_0 q_0' = 0$. By the assumption on minimality of the kernel of $B_0$ in Hypothesis~\ref{hyp: fredholm properties}, we conclude that this is the unique solution up to a constant multiple which is localized on the left and decays faster than $e^{(\nu_0 +\eps) x}$ as $x \to \infty$. Hence we must have $a = 0$, and 
	\begin{align*}
		v(x) + b u_0^0 \chi_+(x) e^{\nu_0 x} = c_1 q_0'(x) \sim c_1 \nu_0 u_0^0 e^{\nu_0 x}
	\end{align*}
	for some constant $c_1 \in \R$. Since $v$ decays faster than $e^{\nu_0 x}$, we must have $b = c_1 \nu_0$, and so
	\begin{align*}
		v(x) = c_1 (q_0'(x) - \nu_0 u_0^0 \chi_+ e^{\nu_0 x}).  
	\end{align*}
	The second equation then implies 
	\begin{align*}
		c_1 \langle q_0' - \nu_0 u_0^0 \chi_+ e^{\nu_0 \cdot}, e_0 \rangle = 0.
	\end{align*}
	With the choice of $e_0$ in \eqref{e:e_0}, we conclude that $c_1 = 0$, and hence the kernel of $D_{(w,\alpha,\beta)} G^{\mathrm{pl}}(w_0; 0, 1, 0)$ is trivial, as desired. 
\end{proof}

\begin{proof}[Proof of Theorem~\ref{t: pulled pushed unfolding}  --- the case $q_\mathrm{pl}$]
	Note that $G(w_0; 0, 1, 0) = 0$, where $w_0$ is given by \eqref{e: pulled unfolding w0}, and by the preceding proposition $D_{(w,a,b)} G^{\mathrm{pl}}(w_0; 0, 1, 0)$ is invertible. By the implicit function theorem, we find $(w(\mu); \alpha(\mu), \beta(\mu))$ so that $G^{\mathrm{pl}}(w(\mu); \alpha(\mu),\beta (\mu), \mu) = 0$ for $\mu$ small. The ansatz
	\begin{align}
		q_{\mathrm{pl}}(x;\mu) = u_- \chi_- (x) + w(x; \mu) + \chi_+(x) [\alpha(\mu) (u_0 (\mu) x + u_1 (\mu)) + \beta(\mu) u_0(\mu)] e^{\nu (\mu) x} 
	\end{align}
	then gives the desired pulled front solutions. 
\end{proof}

\section{Pushed unfolding}\label{s: pushed unfolding}

We now prove the claims on $q_\mathrm{ps}$ from Theorem~\ref{t: pulled pushed unfolding}.


\subsection{Existence of $q_\mathrm{ps}$}

\begin{lemma}\label{l: pushed saddle node}
	The equation $A(0,\nu, c_\mathrm{lin}+\sigma^2; \mu) u = 0$ has two solutions $(\nu^\mathrm{ps}_\pm(\mu, \sigma), u^\mathrm{ps}_\pm (\mu, \sigma))$ for $\sigma$ small, $\nu \approx \nu_\mathrm{lin}(\mu)$, with expansions
	\begin{align}
		\nu^\mathrm{ps}_\pm (\mu, \sigma) &= \nu_\mathrm{lin}(\mu) \pm \sqrt{d_{10}(\mu) d_{02}(\mu)^{-1} \nu_\mathrm{lin}(\mu)} \sigma + \mathrm{O}(\sigma^2) \label{e: nu pm expansions}, \\
		u^\mathrm{ps}_\pm(\mu, \sigma) &= u_0^0 \pm \sqrt{d_{10}(\mu) d_{02}(\mu)^{-1} \nu_\mathrm{lin}(\mu)} (\langle u_0, u_1 \rangle u_0 + u_1) \sigma + \mathrm{O}(\sigma^2). 
	\end{align}
\end{lemma}
\begin{proof}
	Note that $A(0,\nu, c_\mathrm{lin}+\sigma^2; \mu) = A( - \sigma^2 \nu,\nu, c_\mathrm{lin}(\mu); \mu)$. The result then follows by applying Lemma~\ref{l: saddle node}. 
\end{proof}

\begin{corollary}\label{c: pushed leading edge}
	The linearization in the leading edge
	\begin{align}
		(\mathcal{P}(\partial_x) + (c_\mathrm{lin}(\mu)+\sigma^2) M \partial_x + f_u(0; \mu)) u = 0
	\end{align}
	has a solution
	\begin{align*}
		u(x; \mu) = u^\mathrm{ps}_- (\mu, \sigma) e^{\nu^\mathrm{ps}_- (\mu, \sigma) x}.
	\end{align*}
\end{corollary}

To construct the bifurcating pushed front solutions, we insert the far-field core ansatz
\begin{align}
	u(x) = u_- \chi_- (x) + w(x) + a u^\mathrm{ps}_- (\mu, \sigma) \chi_+ (x) e^{\nu^\mathrm{ps}_-(\mu, \sigma) x}
\end{align}
into the traveling wave equation to obtain
\begin{multline*}
	F^{\mathrm{ps}}(w; a, \sigma, \mu) := (\mathcal{P}(\partial_x) + (c_\mathrm{lin}(\mu) + \sigma^2) M \partial_x) (u_- \chi_- + w + a u^\mathrm{ps}_- (\mu, \sigma) \chi_+ e^{\nu^\mathrm{ps}_-(\mu, \sigma) \cdot}) \\ + f(u_- \chi_- + w + \alpha u^\mathrm{ps}_- (\mu, \sigma) \chi_+ e^{\nu_-^\mathrm{ps}(\mu, \sigma) \cdot}; \mu) = 0. 
\end{multline*}
It follows from Corollary~\ref{c: pushed leading edge} that $F^{\mathrm{ps}} : H^{2m}_{\mathrm{exp}, 0, \eta_0} \times \R^2 \times (-\mu_0, \mu_0) \to L^2_{\mathrm{exp}, 0, \eta_0} (\R)$ is well defined and smooth for $\mu_0$ sufficiently small. By Hypothesis~\ref{hyp: front existence}, at $\mu = 0$ we have a solution $F(w_0; 1, 0, 0) = 0$, where 
\begin{align*}
	w_0 (x) = q_0 (x) - u_- \chi_- (x) - \chi_+ (x) u_0^0 e^{\nu_0 x}. 
\end{align*}
As in the pulled case, Hypothesis~\ref{hyp: fredholm properties} together with the Fredholm bordering lemma implies that $D_{(w, a, \sigma)} F^\mathrm{ps}(w_0; 1, 0, 0)$ is Fredholm with index 1. We again augment with a phase condition, defining
\begin{align}
	G^{\mathrm{ps}}(w; a, \sigma, \mu) = \begin{pmatrix}
		F^{\mathrm{ps}}(w; a, \sigma, \mu) \\
		\langle w, e_0 \rangle - \langle w_0, e_0 \rangle,
	\end{pmatrix}
\end{align}
where $e_0$ is chosen as in Section~\ref{s: pulled unfolding}. From a short calculation, we find
\begin{align*}
	D_{(w,a, \sigma)} G^{\mathrm{ps}}(w_0; 1, 0, 0) = \begin{pmatrix}
		B_0 & B_0 (u_0^0 \chi_+ e^{\nu_0 x}) & B_0 [ (\partial_\sigma u^\mathrm{ps}_-(0,0) + x \partial_\sigma \nu_-^\mathrm{ps} (0, 0) x) \chi_+ e^{\nu_0 x} ] \\
		\langle \cdot, e_0 \rangle & 0 & 0
	\end{pmatrix}. 
\end{align*}

\begin{prop}
	The linear operator \[D_{(w, a, \sigma)} G^{\mathrm{ps}}(w_0, 1, 0, 0) : H^{2m}_{\mathrm{exp}, 0, \eta_0} \times \R^2 \times (-\mu_0, \mu_0) \to L^2_{\mathrm{exp}, 0, \eta_0 } (\R) \times \R\] is invertible.
\end{prop}
\begin{proof}
	The Fredholm properties of $F^{\mathrm{ps}}$ together with the Fredholm bordering lemma imply that this linearization is Fredholm with index zero, so as in the pulled case we only have to prove that the kernel is trivial. Assume there exists $(v, a, b)^T \in H^{2m}_{\mathrm{exp}, 0, \eta_0} \times \R^2$ such that 
	\begin{align*}
		\begin{pmatrix}
			B_0 & B_0 (u_0^0 \chi_+ e^{\nu_0 x}) & B_0 [ (\partial_\sigma u^\mathrm{ps}_-(0,0) + x \partial_\sigma \nu_-^\mathrm{ps} (0, 0) u_0^0 x) \chi_+ e^{\nu_0 x} ] \\
			\langle \cdot, e_0 \rangle & 0 & 0 
		\end{pmatrix} \begin{pmatrix} v \\ a \\ b \end{pmatrix} = 0. 
	\end{align*} 
	Then in particular
	\begin{align*}
		B_0 \left[ v + a u_0^0 \chi_+ e^{\nu_0 x} + b (\partial_\sigma u^\mathrm{ps}_-(0,0) + x \partial_\sigma \nu_-^\mathrm{ps} (0, 0) u_0^0 x) \chi_+ e^{\nu_0 x} \right] = 0. 
	\end{align*}
	Since again $q_0'(x) \sim \nu_0 u_0^0 e^{\nu_0 x}$ is the unique solution to $B_0 u = 0$ which decays faster than $e^{(\nu_0 + \eps) x}$ as $x \to \infty$ and is localized on the left, and $\partial_\sigma \nu_-^\mathrm{ps}(0, 0) \neq 0$ by \eqref{e: nu pm expansions}, we conclude that $b = 0$, that $a = c_1 \nu_0$ for some constant $c_1$, and 
	\begin{align*}
		v(x) = c_1 (q_0'(x) - \nu_0 \chi_+ e^{\nu_0 x}). 
	\end{align*}
	The equation $\langle v, e_0 \rangle = 0$ then implies $c_1 \langle (q_0'(x) - \nu_0 \chi_+ e^{\nu_0 x}, e_0 \rangle = 0$, but choosing $e_0$ as in Section~\ref{s: pulled unfolding}, this implies $c_1 = 0$, and so the kernel is trivial, as desired. 
\end{proof}

\begin{proof}[Proof of Theorem~\ref{t: pulled pushed unfolding} --- existence of $q_\mathrm{ps}$]
	We have $G^{\mathrm{ps}}(w_0; 1, 0, 0) = 0$, with $D_{(w, a, \sigma)} G^{\mathrm{ps}}(w_0; 1, 0, 0)$ invertible. The existence of $q_\mathrm{ps}$ in Theorem~\ref{t: pulled pushed unfolding} then follows directly from the implicit function theorem, with 
	\begin{align*}
		q_\mathrm{ps} (x; \mu) = u_- \chi_- (x) + w(x; \mu) + a(\mu) u^\mathrm{ps}_-(\mu, \sigma(\mu)) \chi_+(x) e^{\nu^\mathrm{ps}_- (\mu, \sigma(\mu)) x},
	\end{align*}
	where $G^{\mathrm{ps}}(w(\cdot; \mu); a(\mu), \sigma(\mu), \mu) = 0$ by the implicit function theorem. 
\end{proof}

\subsection{Expansion of $c_\mathrm{ps}(\mu)$}
Having established the existence of $q_\mathrm{pl}$ and $q_\mathrm{ps}$, to complete the proof of Theorem~\ref{t: pulled pushed unfolding} it only remains to establish the expansion \eqref{e: pushed speed expansion} for $c_\mathrm{ps}(\mu)$. 

\begin{prop}\label{p: alpha sigma relation}
	Assume that the family of pulled fronts from Theorem~\ref{t: pulled pushed unfolding} satisfies $\alpha'(0) \neq 0$. Then 
	\begin{align}
		\sigma'(0) = - \frac{1}{\sqrt{d_{10}d_{02}^{-1} \nu_0}} \alpha'(0) \neq 0. 
	\end{align}
\end{prop}
\begin{proof}
	 Starting with the solution $G^{\mathrm{ps}}(w(\mu); a(\mu), \sigma(\mu), \mu) = 0$ constructed above, we compute 
	\begin{multline}
		\partial_\mu F^{\mathrm{ps}}(w_0; 1, 0, 0) = c_\mathrm{lin}'(0) M \partial_x (u_- \chi_- + w_0 + u_0^0 \chi_+ e^{\nu_0 x} ) + B_0 (\partial_\mu w(0) + a'(0) u_0^0  \chi_+ e^{\nu^\mathrm{ps}(0, 0)\cdot} \\ + \partial_\mu u^\mathrm{ps}_-(0, 0) \chi_+ e^{\nu_0 \cdot} + \partial_\mu \nu^\mathrm{ps}_- (0, 0) u_0^0 x e^{\nu_0 \cdot}). 
	\end{multline}
	Projecting onto the cokernel, we find 
	\begin{multline}
		0 = \langle \partial_\mu F^{\mathrm{ps}}(w_0; 1, 0, 0), \phi \rangle = c_\mathrm{lin}'(0) \langle M \partial_x (u_- \chi_- + w_0 + u_0^0 \chi_+ e^{\nu_0 x}), \phi \rangle + \langle B_0 (\partial_\mu w(0)), \phi \rangle \\ + a'(0) \langle B_0 (u_0^0 \chi_+ e^{\nu_0 \cdot}), \phi \rangle + \langle B_0 [ (\partial_\mu u^\mathrm{ps}_- (0, \sigma(0)) + \partial_\mu \nu_-^\mathrm{ps} (0, \sigma(0)) u_0^0 x) \chi_+ e^{\nu_0 \cdot}], \phi \rangle 
	\end{multline}
	Note that $\langle B_0 (\partial_\mu w(0)), \phi \rangle = \langle B_0 (u_0^0 \chi_+ e^{\nu_0 \cdot}), \phi \rangle = 0$ by Lemma~\ref{l: projections}. Using Lemma~\ref{l: pushed saddle node} to compute $\partial_\mu u_-^\mathrm{ps}(\mu, \sigma(\mu))|_{\mu = 0}$ and $\partial_\mu \nu^\mathrm{ps}_- (\mu, \sigma(\mu))|_{\mu = 0}$, we simplify to 
	\begin{multline}
		0 = c_\mathrm{lin}'(0) \langle M \partial_x (u_- \chi_- + w_0 + u_0^0 \chi_+ e^{\nu_0 x}), \phi \rangle \\ + \langle B_0 [ (u_0'(0) - \sqrt{d_{10} d_{02}^{-1} \nu_0 } (\langle u_0^0, u_1^0 \rangle u_0^0 + u_1^0) \sigma'(0) + \nu_*'(0) u_0^0 x - \sqrt{d_{10} d_{02}^{-1} \nu_0 } \sigma'(0) u_0^0 x) \chi_+ e^{\nu_0 \cdot}, \phi \rangle. 
	\end{multline}
	Noting again that $\langle B_0 (u_0^0 \chi_+ e^{\nu_0 \cdot}), \phi \rangle = 0$, we obtain 
	\begin{align}
		\sigma'(0) = \frac{c_\mathrm{lin}'(0) \langle M \partial_x (u_- \chi_- + w_0 + u_0^0 \chi_+ e^{\nu_0 x}), \phi \rangle + \langle B_0 (u_0'(0) \chi_+ e^{\nu_0 \cdot}), \phi \rangle + \nu_*'(0) \langle B_0 (u_0^0 x), \phi \rangle }{\sqrt{d_{10} d_{02}^{-1} \nu_0 } \langle B_0 [(u_0^0 x +u_1^0) \chi_+ e^{\nu_0 \cdot}], \phi \rangle}.
	\end{align}
	
	To complete the proof, we compute $\alpha'(0)$ from the analysis of Section~\ref{s: pulled unfolding} and compare the two expressions. We recall the definition from Section~\ref{s: pulled unfolding}
	\begin{align}
		G^\mathrm{pl} (w; \alpha, \beta, \mu) = \begin{pmatrix}
			F^\mathrm{pl} (w; \alpha, \beta, \mu) \\
			\langle w, e_0 \rangle - \langle w_0, e_0 \rangle, 
		\end{pmatrix}
	\end{align}
	where 
	\begin{align}
		F^\mathrm{pl} (w; \alpha, \beta, \mu) = (\mathcal{P}(\partial_x) + c_\mathrm{lin}(\mu) M \partial_x) (u_- \chi_- + w + \chi_+ \psi) + f(u_- \chi_- + w + \chi_+ \psi; \mu) = 0,
	\end{align}
	with $\psi$ given by \eqref{e: pulled unfolding psi def}. From Section~\ref{s: pulled unfolding}, we have a solution $F^\mathrm{pl}(w(\mu); \alpha(\mu), \beta(\mu), \mu) = 0$ for $\mu$ small.  We compute 
	\begin{multline}
		\partial_\mu F^\mathrm{pl}(w(\mu); \alpha(\mu), \beta(\mu), \mu)|_{\mu = 0} = c_\mathrm{lin}'(0) M \partial_x (u_- \chi_- + w_0 + \chi_+ u_0^0 e^{\nu_0 \cdot} ) \\ + B_0 (\partial_\mu w(0) + \chi_+ \partial_\mu \psi(\cdot; \alpha(\mu), \beta(\mu), \mu)|_{\mu = 0} ).
	\end{multline}
	We compute from \eqref{e: pulled unfolding psi def}
	\begin{align}
		\partial_\mu \psi(x; \alpha(\mu), \beta(\mu), \mu)|_{\mu = 0} = [\alpha'(0)(u_0^0 x + u_1^0) + \beta'(0) u_0^0 + u_0'(0)] e^{\nu_0 x} + u_0^0 \nu_*'(0) x e^{\nu_0 x}
	\end{align}
	Projecting onto the span of $\phi$ and using that $\langle B_0 (\partial_\mu w(0)), \phi \rangle = \langle B_0 (u_0^0 \chi_+ e^{\nu_0 \cdot}), \phi \rangle = 0$, we obtain
	\begin{multline}
		0 = c_\mathrm{lin}'(0) \langle M \partial_x (u_- \chi_- + w_0 + \chi_+ u_0^0 e^{\nu_0 \cdot} ), \phi \rangle + \alpha'(0) \langle B_0 [(u_0^0 x + u_1^0) \chi_+ e^{\nu_0 \cdot}], \phi \rangle \\ + \langle B_0 (u_0'(0) \chi_+ e^{\nu_0 \cdot}), \phi \rangle + \nu_*'(0) \langle B_0 (u_0^0 x \chi_+e^{\nu_0 \cdot}), \phi \rangle.
	\end{multline}
	Solving for $\alpha'(0)$, we obtain
	\begin{align}
		\alpha'(0) &= - \frac{c_\mathrm{lin}'(0) \langle M \partial_x (u_- \chi_- + w_0 + \chi_+ u_0^0 e^{\nu_0 \cdot} ) + \langle B_0 (u_0'(0) \chi_+ e^{\nu_0 \cdot}), \phi \rangle + \nu_*'(0) \langle B_0 (u_0^0 x \chi_+e^{\nu_0 \cdot}), \phi \rangle}{\langle B_0 [(u_0^0 x + u_1^0)\chi_+ e^{\nu_0 \cdot}], \phi \rangle} \\
		&= - \sqrt{d_{10}d_{02}^{-1} \nu_0} \sigma'(0).
	\end{align}
	Since $-\sqrt{d_{10} d_{02}^{-1} \nu_0} < 0$, this completes the proof of the proposition. 
\end{proof}

\begin{proof}[Proof of Theorem~\ref{t: pulled pushed unfolding} --- conclusion]
	By Proposition~\ref{p: alpha sigma relation}, the pushed speed is given by 
	\begin{align}
		c_\mathrm{ps}(\mu) = c_\mathrm{lin}(\mu) + \sigma(\mu)^2 &= c_\mathrm{lin}(\mu) + \sigma'(0)^2 \mu^2 + \mathrm{O}(\mu^3) \\
		&= c_\mathrm{lin}(\mu) + \frac{\alpha'(0)^2}{d_{10} d_{02}^{-1} \nu_0} \mu^2 + \mathrm{O}(\mu^3),
	\end{align}
	as desired. 
\end{proof}

\section{Stability criteria}\label{s: pulled stability}

\subsection{Spectral stability of pulled fronts}
We prove the claims on the linearization $\mathcal{L}_\mathrm{pl}$ in Theorem~\ref{t: pulled pushed stability}. The claims on the essential spectrum follow immediately from the dispersion relation, noting that the Fredholm borders are given through roots of $d_c(\lambda,ik)=0$. In order to trace point spectrum we prepare with the following observation.



\begin{corollary}\label{c: pulled stability leading edge solutions}
	The linearization in the leading edge
	\begin{align}
		(\mathcal{P}(\partial_x) + c_\mathrm{lin}(\mu) M \partial_x + f_u (0; \mu) - M \lambda) u = 0 
	\end{align}
	has solutions 
	\begin{align}
		u(x; \mu, \gamma) = u^\pm_\mathrm{pl} (\mu, \gamma) e^{\nu^\mathrm{pl}_\pm (\mu, \gamma) x}
	\end{align}
	for some vectors $u^\pm_\mathrm{pl}(\mu, \gamma) \in \C^n$ for $\mu$ and $\gamma$ sufficiently small. Furthermore, $u^\pm_\mathrm{pl} (0, 0) = u_0^0$, and $\partial_\gamma u^-_\mathrm{pl} (0, 0) = \partial_\gamma \nu^\mathrm{pl}_- (0, 0) (\langle u_0^0, u_1^0 \rangle u_0^0 + u_1^0)$. 
\end{corollary}
\begin{proof}
	This follows directly from Lemma~\ref{l: saddle node}. 
\end{proof}

Let
\begin{align}
	B_\mathrm{pl}(\mu) = \mathcal{P}(\partial_x) + c_\mathrm{lin}(\mu) M \partial_x + f_u (q_\mathrm{pl}(\cdot; \mu); \mu) 
\end{align}
denote the linearization about a pulled front. Note that for $\gamma$ small with $\Re \gamma > 0$, we have $\Re \nu^\mathrm{pl}_- (\nu, \gamma) < \Re \nu_\mathrm{lin}(\mu)$. To capture eigenfunctions with the appropriate localization, we therefore insert the ansatz
\begin{align}
	u(x) = w(x) + \beta u^-_\mathrm{pl} (\mu, \gamma) \chi_+ (x) e^{\nu^\mathrm{pl}_-(\mu, \gamma) x}
\end{align}
into the eigenvalue equation $(B_\mathrm{pl}(\mu) - \gamma^2 M) u = 0$. Letting $P$ denote the orthogonal projection onto the range of $B_\mathrm{pl}(0) = B_0$ as in Section~\ref{s: preliminaries}, we then decompose the resulting equation as 
	\begin{align}
		\begin{cases}
			P(B_\mathrm{pl}(\mu) - \gamma^2 M) \left(w + \beta u^-_\mathrm{pl}(\mu, \gamma) \chi_+ e^{\nu^\mathrm{pl}_-(\mu, \gamma) \cdot } \right) &= 0, \\
			\left\langle (B_\mathrm{pl}(\mu) - \gamma^2 M) \left(w + \beta u^-_\mathrm{pl}(\mu, \gamma) \chi_+ e^{\nu^\mathrm{pl}_-(\mu, \gamma) x} \right), \phi \right\rangle &= 0. 
		\end{cases}\label{e: pulled eigenvalue system}
	\end{align}
The following lemma states that eigenvalues of $\mcl_\mathrm{pl}(\mu)$ in a neighborhood of the origin are completely characterized by solutions to \eqref{e: pulled eigenvalue system}. The proof is identical to that of \cite[proof of Proposition 5.11, step 6]{PoganScheel}.
\begin{lemma}\label{l: pulled eigenvalue reduction}
	The equation $(B_\mathrm{pl}(\mu) - \lambda M) u = 0$ has a solution $u \in L^2_{\mathrm{exp}, 0, -\nu_\mathrm{lin}(\mu)}$ with $\lambda = \gamma^2, \Re \gamma > 0$ if and only if \eqref{e: pulled eigenvalue system} has a solution with $\Re \gamma \geq 0$.
\end{lemma}

At $\mu = \gamma = 0$, the system \eqref{e: pulled eigenvalue system} has a solution resulting from the translational mode $q_0'(x) \sim \nu_0 u_0^0 e^{\nu_0 x}$, given by
\begin{align}
	w_0 = \frac{\beta}{\nu_0} (q_0' - \nu_0 u_0^0 \chi_+ e^{\nu_0 \cdot})
\end{align}
for any $\beta \in \C$. The linearization in $w$ (in the space $H^{2m}_{\mathrm{exp}, 0, \eta_0}$) of the first equation in \eqref{e: pulled eigenvalue system} about this solution is $P B_0$, which is invertible by construction, so we can use the implicit function theorem to solve the first equation for $w(\beta; \mu, \gamma)$ for $\mu$ and $\gamma$ small. Linearity in $\beta$ and the uniqueness of the solution found with the implicit function theorem then guarantees that we can write $w(\beta; \mu, \gamma) = \beta \tilde{w}(\mu, \gamma)$ for some $\tilde{w}(\mu, \gamma) \in H^{2m}_{\mathrm{exp}, 0, \eta_0}$ with
\begin{align}
	\tilde{w}(0,0) = \frac{1}{\nu_0} (q_0' - \nu_0 u_0^0 \chi_+ e^{\nu_0 \cdot}). \label{e: pulled stability tilde w0 expression}
\end{align}
The second equation in \eqref{e: pulled eigenvalue system} may therefore be reduced to 
\begin{align}
	E(\mu, \gamma) := \left \langle (B_\mathrm{pl}(\mu) - \gamma^2 M) \left(\tilde{w}(\mu, \gamma) + u^-_\mathrm{pl}(\mu, \gamma) \chi_+ e^{\nu^\mathrm{pl}_- (\mu, \gamma) \cdot} \right), \phi \right\rangle = 0. 
\end{align}
The linearization $B_\mathrm{pl}(\mu)$ therefore has an eigenvalue (or more precisely a resonance pole) at $\lambda = \gamma^2$ if and only if $E(\mu, \gamma) = 0$, by Lemma~\ref{l: pulled eigenvalue reduction}. Note that $E(0, 0) = 0$ as a consequence of Lemma~\ref{l: projections}. To track how this zero perturbs for $\mu, \gamma \approx 0$, we expand $E(\mu, \gamma)$ to leading order in $\mu$ and $\gamma$. We compute
\begin{align}
	\partial_\gamma E(0, 0) = \left\langle B_0 \left( \partial_\gamma \tilde{w}(0,0) + (\partial_\gamma u^-_\mathrm{pl}(0, 0) \chi_+ e^{\nu_0 \cdot} + u_0^0 \partial_\gamma \nu^\mathrm{pl}_- (0, 0) x) \chi_+ e{^{\nu_0 \cdot}} \right), \phi \right\rangle . 
\end{align}
Note that $\langle B_0 (\partial_\gamma \tilde{w}(0,0)), \phi \rangle = \langle \partial_\gamma \tilde{w}(0,0), B_0^* \phi \rangle = 0$ due to the strong exponential localization of $\tilde{w}$. Together with Corollary~\ref{c: pulled stability leading edge solutions}, we then have 
\begin{align}
	\partial_\gamma E(0,0) &= \partial_\gamma \nu_-^\mathrm{pl}(0,0) \langle B_0 [(u_1^0 + u_0^0 x) \chi_+ e^{\nu_0 x}], \phi \rangle  + \partial_\gamma \nu_-^\mathrm{pl}(0, 0) \langle u_0^0, u_1^0 \rangle \langle B_0 (u_0^0 \chi_+ e^{\nu_0 x}), \phi \rangle \\
	&= \partial_\gamma \nu_-^\mathrm{pl}(0,0) \langle B_0 [(u_1^0 + u_0^0 x) \chi_+ e^{\nu_0 x}], \phi \rangle,  \label{e: pulled stability gamma derivative}
\end{align}
 since $\langle B_0 (u_0^0 \chi_+ e^{\nu_0 x}), \phi \rangle = 0$ by Lemma~\ref{l: projections}. Also by Lemma~\ref{l: projections}, the remaining term on the right hand side is non-zero, and so $\partial_\gamma E(0, 0) \neq 0$. In particular, we can solve the equation $E(\mu, \gamma) = 0$ for $\gamma(\mu)$ with the implicit function theorem. To track $\gamma(\mu)$, we now expand $E(\mu,\gamma)$ in $\mu$. We will use the following auxiliary result.
\begin{lemma}\label{l: pulled stability mu derivative}
	We have 
	\begin{align}
		(\partial_\mu B_\mathrm{pl}(\mu)) q_\mathrm{pl} ' (x; \mu) = -B_\mathrm{pl}(\mu) \partial_\mu q_\mathrm{pl}'(x; \mu). 
	\end{align}
\end{lemma}
\begin{proof}
	Differentiating the traveling wave equation in space, we obtain $B_\mathrm{pl}(\mu) q_\mathrm{pl}'(\cdot; \mu) = 0$. Differentiating again in $\mu$ then implies the desired result. 
\end{proof}
With this lemma in hand, we now compute
\begin{align}
	\partial_\mu E(0,0) = \langle B_\mathrm{pl}'(0) \left( \tilde{w}(0, 0) + u_0^0 \chi_+ e^{\nu_0 \cdot} \right), \phi \rangle + \langle B_0 [(\partial_\mu u^-_\mathrm{pl}(0, 0) + u_0^0 \partial_\mu \nu_-^\mathrm{pl}(0, 0) x) \chi_+ e^{\nu_0 x} ], \phi \rangle \label{e: pulled stability mu expansion}
\end{align}
Using \eqref{e: pulled stability tilde w0 expression} and Lemma~\ref{l: pulled stability mu derivative}, we obtain
\begin{align}
	\langle B_\mathrm{pl}'(0) (\tilde{w}(0, 0) + u_0^0 \chi_+ e^{\nu_0 \cdot}), \phi \rangle = \frac{1}{\nu_0} \langle B_\mathrm{pl}'(0) q_0', \phi \rangle = - \frac{1}{\nu_0} \langle B_0 \partial_\mu q_\mathrm{pl}'(\cdot; 0), \phi \rangle \label{e: pulled stability d mu terms}
\end{align}
Using the asymptotics for $q_\mathrm{pl}(x; \mu)$ for large $x$ from Theorem~\ref{t: pulled pushed unfolding} (recall also $u_0(\mu)$ and $\nu_\mathrm{lin}(\mu)$ defined there), we see that
\begin{align}
	\partial_\mu \partial_x q_\mathrm{pl} (x; 0) = \left[\alpha'(0) u_0^0 + \nu_0 \left( \alpha'(0) (u_0^0 x + u_1^0) + \beta'(0) u_0^0 + u_0'(0) + x \nu_\mathrm{lin}'(0) u_0^0 \right) \right] e^{\nu_0 x} + \tilde{Q} (x)
\end{align}
recalling that $\beta(0) = 1$ and $\alpha(0) = 0$. The error term $\tilde{Q}$ is exponentially localized on the left and decays faster than $e^{(\nu_0 - \eps) x}$ as $x \to +\infty$, and so in particular $\langle B_0 \tilde{Q}, \phi \rangle = \langle \tilde{Q}, B_0^* \phi \rangle = 0$. Also, by Lemma~\ref{l: projections}, we have 
\begin{align}
	\langle B_0 [(\alpha'(0) + \beta'(0)) u_0^0 e^{\nu_0 \cdot} ], \phi \rangle = (\alpha'(0) + \beta'(0)) \langle B_0 (u_0^0 e^{\nu_0 \cdot}), \phi \rangle = 0. \label{e: pulled stability simplify}
\end{align}
Hence, we obtain the simplification
\begin{align*}
	-\frac{1}{\nu_0} \langle B_0 \partial_\mu q_\mathrm{pl}'(\cdot; 0), \phi \rangle = - \alpha'(0) \langle B_0 [(u_0^0 x + u_1^0) e^{\nu_0 \cdot}, \phi \rangle - \langle B_0 [(u_0'(0) + x \nu_\mathrm{lin}'(0) u_0^0) e^{\nu_0 \cdot}, \phi \rangle.
\end{align*}
Returning to \eqref{e: pulled stability mu expansion}, note that $u^-_\mathrm{pl}$ and $\nu_-^\mathrm{pl}$ are smooth functions for which $u^-_\mathrm{pl}(\mu, 0) = u_0 (\mu)$ and $\nu_-^\mathrm{pl}(\mu, 0) = \nu_\mathrm{lin}(\mu)$, where $u_0 (\mu)$ and $\nu_\mathrm{lin}(\mu)$ are as in Theorem~\ref{t: pulled pushed unfolding}. Hence 
\begin{align*}
	\partial_\mu u^-_\mathrm{pl} (0, 0) = u_0'(0), \quad \partial_\mu \nu_-^\mathrm{pl}(0, 0) = \nu_\mathrm{lin}'(0). 
\end{align*}
Combining \eqref{e: pulled stability d mu terms} and \eqref{e: pulled stability simplify} with \eqref{e: pulled stability mu expansion}, we therefore obtain
\begin{align}
	\partial_\mu E(0,0) = - \alpha'(0) \langle B_0 [(u_0^0 x + u_1^0) e^{\nu_0 x}, \phi \rangle. \label{e: pulled stability mu derivative}
\end{align}
Combining with \eqref{e: pulled stability gamma derivative}, we find $E(\mu, \gamma(\mu)) = 0$ with 
\begin{align}
	\gamma(\mu) = \frac{\alpha'(0)}{\partial_\gamma \nu_-^\mathrm{pl}(0, 0)} \mu + \mathrm{O}(\mu^2). \label{e: gamma mu formula}
\end{align}

\begin{proof}[Proof of Theorem~\ref{t: pulled pushed stability}]
	By Lemma~\ref{l: pulled eigenvalue reduction}, and the above reduction, we have such a solution if and only if $\Re \gamma(\mu) \geq 0$, which occurs precisely when $\alpha'(0) \mu < 0$ since $\partial_\gamma\nu^\mathrm{pl}_-(0,0) < 0$. Since $E(\mu, \gamma)$ has precisely one root in a neighborhood of the origin, $\lambda(\mu) = \gamma(\mu)^2$ is the only resonance pole of $\mcl_\mathrm{pl}(\mu)$ for $\mu$ sufficiently small. The resonance pole $\lambda(\mu)$ is an unstable eigenvalue for $\alpha'(0) \mu < 0$ by \eqref{e: gamma mu formula}. 
\end{proof}

\subsection{Marginal stability of pushed fronts}
We first note that the essential spectrum is stable since the speed of pushed fronts is strictly larger than the linear spreading speed. A natural candidate for the point spectrum is the derivative of the front solution, contributing to the kernel of the linearization provided growth conditions at infinity are met. 
In our setting,  these growth conditions are met precisely when $e^{\nu_-(\mu, \sigma(\mu)) x}$ has steeper decay than $e^{\nu_\mathrm{lin}(\mu) x}$, which in turn occurs when $\sigma(\mu) > 0$.

\begin{proof}[Proof of Theorem~\ref{t: pulled pushed stability}]
	Note that $\mcl^\mathrm{ps}(\mu) (\omega_{0, -\nu_\mathrm{lin}(\mu)}q_\mathrm{ps}'(\cdot; \mu)) = 0$. By a stability analysis similar to that of Section~\ref{s: pulled stability}, this is the only solution to $(\mcl^\mathrm{ps}(\mu) - \lambda M) u = 0$ in a neighborhood of $\lambda = 0$. We therefore only need to check whether $\omega_{0, -\nu_\mathrm{lin}(\mu)}q_\mathrm{ps}'(\cdot; \mu) \in L^2 (\R)$. But since 
	\begin{align*}
		\omega_{0, -\nu_\mathrm{lin}(\mu)} q_\mathrm{ps}' (x; \mu) \sim e^{(-\nu_\mathrm{lin}(\mu)+\nu_-^\mathrm{ps}(\mu, \sigma(\mu)) x}
	\end{align*}
	
	 for $x$ large (and is exponentially localized on the left), this occurs precisely when $-\nu_\mathrm{lin}(\mu) + \nu_-^\mathrm{ps}(\mu, \sigma(\mu)) < 0$. By the expansion in Lemma~\ref{l: pushed saddle node}, this occurs precisely when $\sigma(\mu) > 0$. The result then follows from Proposition~\ref{p: alpha sigma relation}. 
\end{proof}

\section{Numerical continuation of pulled fronts and identification of the pushed-to-pulled transition}
Building on the analysis in the previous sections, we now turn our attention to developing a numerical continuation routine to efficiently continue pulled fronts, pushed fronts and pushed-to-pulled transitions.  We  approximate the infinite-domain problem studied until now with a boundary value problem in a large finite domain, using the same decomposition of the invasion front solution into a far-field element that captures the exact decay of the front and a more localized core function, for which we impose additional boundary conditions to reflect the negative Fredholm index of the linearization on the unbounded domain.  As shown in Sections \ref{s: pulled unfolding} and \ref{s: pushed unfolding}, the far-field term includes an explicit exponentially decaying term of the form $(\alpha x+\beta)e^{\nu x}$ (or $\beta e^{\nu x}$ in the pushed case) and the transition between pushed and pulled fronts occurs when $\alpha=0$. We  show that our method  efficiently and accurately locates this transition point by bench-marking our routine against examples where explicit transition points are known.  We then turn our attention to several systems of reaction-diffusion equations where explicit expressions for the transition value are not known.

\subsection{Numerical far-field-core decomposition}

We discuss our numerical strategy in the scalar case and point to modifications for systems for ease of exposition.

\paragraph{Pulled front continuation.}

Our aim is to locate and continue pulled fronts propagating with the linear spreading speed $c$ that are solutions of the traveling wave equation,
\begin{equation}\label{e:tw2}
    \mathcal{P}(\partial_x)u+cu_x+f(u;\mu)=0
.\end{equation}
We seek an approximate solution in the form of a far-field core decomposition,
\begin{equation} 
    u(x)=u_- \chi_-(x)+u_+\chi_+(x) +(\alpha x+\beta) e^{\nu x} \chi_+(x)+w(x)
,\label{eq:ffcore}
\end{equation}  
with cutoff functions $\chi_\pm$ of the form 
\[
\chi_+(x)=(1+e^{m x})^{-1},\qquad \chi_-(x)=1-\chi_+(x),\qquad m\sim 10,
\]
so that derivatives of $\chi_\pm$ are well resolved on our computational grid but derivatives essentially vanish near the boundary of the computational domain, $|x|\gtrsim 10$. 

We insert this ansatz with pre-computed derivatives for $\chi_\pm$ into \eqref{e:tw2} and subtract the identities
\begin{align*}
0&=[\mathcal{P}(\partial_x)+c\partial_x]u_\pm +f(u_\pm;\mu),\\
0&=\chi_+(x)[\mathcal{P}(\partial_x)+c\partial_x+f'(0;\mu)]((\alpha x+\beta)e^{\nu x}),
\end{align*}
Exploiting cancellations when subtracting can be somewhat tedious but helps with roundoff errors when $x$ is large. 

The resulting equation is of the form 
\begin{equation}\label{e:weqn}
0=\mathcal{P}(\partial_x)w+c\partial_x w + g(x,w,\alpha,\beta,\nu,c),
\end{equation}
and needs to be complemented with boundary conditions for $w$. For second-order $\mathcal{P}$ one can use Dirichlet or Neumann boundary conditions with little noticeable difference and we mostly use Neumann boundary conditions.  

We next discretize the interval $x\in[-L,L]$ with $N$ subintervals and $N$ centered grid points $X\in\R^{N}$. We use $N$ values of $W$ at grid points as variables, which gives a total of $N+7$ unknowns,
\[
    W\in\mathbb{R}^N, \quad u_-, u_+\in \mathbb{R}, c,\nu \in \mathbb{R}, \alpha, \beta \in\mathbb{R}, \mu\in\mathbb{R}
.\]
As equations, we require \eqref{e:weqn} to hold at the centered grid points, using finite-difference approximations of $\mathcal{P}$ and $\partial_x$. We thereby obtain $N$ equations, including the boundary conditions,
\begin{equation}\label{e:weqnd}
0=\mathcal{P}^N W+c D_1^N  W + G(X,W,\alpha,\beta,\nu,c), 
\end{equation}
We use  second or fourth order approximations of centered finite differences for $\mathcal{P}^N$ and $D_1^N$. The system \eqref{e:weqnd} is complemented with equations for $u_\pm$
\begin{equation}
\label{e:upm}
0=f(u_-;\mu),\qquad 0= f(u_+;\mu).
\end{equation}

The variables $c$ and $\nu$ are obtained from the dispersion relation, solving 
\begin{equation}
\label{e:disnum}
0 =  d(\nu,c,0;\mu) ,\qquad 
0 = \partial_\nu d(\nu,c,0;\mu),
\end{equation}
or some version of \eqref{e:dismatrix} in the case of systems. We finally add two conditions that account for the variables $\alpha$ and $\beta$. First, we wish to enforce exponential decay of $w$ faster than $e^{\nu x}$, which amounts to adding a transversality condition near the right boundary. In practice, we have observed that a variety of transversality conditions may be employed and we typically default to the condition that the core function satisfies $W_N+W_{N-1}=0$.  Second, our ansatz allows for translation invariance and we add a condition that pins the core to the center of the computational domain. Again, many of such phase conditions will work in practice, and we use a Gaussian. Altogether, we solve \eqref{e:weqnd}, \eqref{e:upm}, \eqref{e:disnum}, 
\begin{equation}\label{e:tr}
0=W_N+W_{N-1},
\end{equation}
and a discretized version of 
\begin{equation}\label{e:ph}
 0= \int_{-L}^L e^{-x^2}\left( u(x)-\frac{u_-+u_+}{2}\right) dx.
\end{equation}
This system of $N+6$ equations for $N+7$ variables, including the parameter $\mu$ can then be augmented by a secant condition to continue solutions and find in particular $\alpha(\mu)$. Detecting values where $\alpha(\mu)=0$ then gives the desired location of the pushed-to-pulled transitions. Adding an equation $\alpha=0$ and a second parameter $\mu_2$, we can similarly located curves of pushed-to-pulled transition in 2-parameter systems. 

Adaptations for systems are straightforward. For $u\in\R^n$, we have $W\in \R^{nN}$, $u_+,u_-\in\R^n$, $\alpha,\beta,c,\nu,\mu\in\R$ for a total of $n(N+2)+5$ variables.  The equations 
\eqref{e:weqnd} and
\eqref{e:upm} are cast for systems, yielding $n(N+2)$ equations, with an additional 4 equations from 
\eqref{e:disnum}, 
\eqref{e:tr}, and
\eqref{e:ph}. Of course, \eqref{e:disnum} could be replaced by \eqref{e:dismatrix}, which would yield the relevant eigenvectors $u_{0/1}(\mu)$ in the tail expansion \eqref{e: le}.

\paragraph{Pushed front continuation.}
The algorithm described above can easily be adapted to located and follow pushed fronts to pushed-to-pulled transitions with minimal modifications. Since far-field expansions should be purely exponential, the far-field ansatz simplifies to  $\chi_+(x)\beta e^{\nu x}$. On the other hand, $\nu$ is simply a root of the dispersion relation, so that we effectively only replace the second equation in \eqref{e:disnum}, $\partial_\nu d_{c}=0$, by the equation $\alpha=0$, otherwise retaining the system of equations 
\eqref{e:weqnd},
\eqref{e:upm},
\eqref{e:disnum},
\eqref{e:tr}, and
\eqref{e:ph}.
%
%
%

We note that pushed fronts can in this fashion be naturally continued through the pushed-to-pulled transition point to continue the family of traveling fronts found in Theorems \ref{t: pulled pushed unfolding}--\ref{t: pulled pushed stability}, whose spatial decay rates are weaker than the linear decay rate.  We illustrate this in the examples considered below. In order to understand exponential convergence rates, one quickly notices that errors from the truncation in the wake decrease with the gap between stable and unstable eigenvalues of the linearization at $u_-$. Errors from truncation in the leading edge contain contributions from first the linear ansatz $(\alpha x + \beta)e^{\nu x}$, neglecting for instance quadratic terms $\mathrm{O}(x^2 e^{2\nu x})$, and from the effect of boundary conditions, with errors related to the gap between the double root and next-nearest eigenvalues. Projecting errors onto the kernel of the adjoint, which grows with exponential rate $-\nu$ in the leading edge, predicts truncation errors with exponential $e^{\nu L}$ from the former contribution, and $e^{\delta_\nu L}$ with gap $\delta_\nu$ between the decay rate $\nu$ and the next nearest spatial eigenvalue.

\paragraph{Convergence aspects and comparisons to other methods.}
The approach to computing pushed and pulled fronts presented here can be compared to more direct methods for computing front speeds \cite{stegemerten,beyn_freeze} or to methods for computing heteroclinic orbits \cite{beyn_hetero,auto_hom}. 

A direct approach to finding wave speeds would be to study the invasion process in a finite domain, appropriately shifting the wave front such that the front interface remains located near the center of the domain, thus minimizing effects from boundaries \cite{beyn_freeze}. Rather than dynamically relaxing the dynamics, one could also employ a Newton method, using the wave speed as a Lagrange multiplier associated with a phase condition that pins the front interface in the center of the domain. One thus arrives at a boundary-value problem for the traveling-wave equation \eqref{e:tw2} together with a phase condition similar to \eqref{e:ph}, but enforced on the full profile $u$, and the wave speed $c$ as parameter. Continuation for this system was systematically used in \cite{stegemerten}, studying variants of the Allen-Cahn equation. The resulting convergence questions were discussed in \cite{adss}, demonstrating that speed and profile converge as $L\to\infty$ with rate $L^{-2}$ in the case of pulled fronts and with exponential rate in the case of pushed fronts. There does not appear to be a way intrinsic to such a method to determine whether a computed front profile is pulled or pushed. In fact, \cite{adss} gives examples of anomalous wave propagation processes where the computed speed converges to an incorrect limiting speed as $L\to\infty$. We comment below on some scenarios where our algorithm breaks down upon encountering such resonances leading to potential different modes of invasion. 

Fronts considered here are of course heteroclinic orbits to the traveling-wave equation, cast as a first-order dynamical system and there is a vast literature on computing such heteroclinic profiles, including algorithms that systematically detect bifurcations \cite{beyn_hetero,auto_hom}. Pushed fronts fall into a standard class of computation of codimension-1 heteroclinic orbits, in this case as intersections between an unstable and a strong stable manifold. Convergence for such heteroclinic orbits is exponential in the size of the domain, with rate for given by the gap between stable and unstable eigenvalues (or strong and weak stable eigenvalues in the leading edge) for the profile and twice that rate for the speed. Clearly, this gap converges to zero when the pushed front approaches the pushed-to-pulled transition, so that the rate of exponential convergence for standard algorithms for finding pushed fronts converges to zero. On the other hand, pulled front speeds are codimension-zero heteroclinic orbits, again with exponential convergence of profiles. The speed is of course best found directly from the dispersion relation.

We do not pursue a full analysis of convergence of our algorithm, here. We do demonstrate below that convergence of speeds is exponential with uniform rate across the pushed-to-pulled transition.

\subsection{Applications}

\paragraph{The Nagumo Equation.}
Perhaps the most familiar example of a pushed front occurs in the Nagumo equation
\begin{equation} u_t=u_{xx}+u(1-u)(\mu+u), \label{eq:nagumo} \end{equation}
see for example \cite{HadelerRothe}. One is interested in the regime $\mu>0$ where $u=0$ is unstable, and then studies the propagation of the stable state $u=1$ into this unstable background. 
For any $\mu>\frac{1}{2}$ the invasion process is  pulled in nature with a monotone front propagating at the linear spreading speed $2\sqrt{\mu}$.  At $\mu=\frac{1}{2}$ a pushed front bifurcates that propagates with speed $c=\sqrt{2}\left(\frac{1}{2}+\mu\right)$.  We note that it is only in this particular cubic nonlinearity (and in a cubic-quintic case) that one is able to explicitly locate the transition to pushed fronts. Known criteria that exclude pushed fronts, such as $f(u)\leq f'(0)u$, are not sharp. 

Our numerical results are illustrated in Figures~\ref{fig:nagumocore}-\ref{fig:nagumopushed}.

In Figure~\ref{fig:nagumocore}, we present an overview of our algorithm showing the far-field core decomposition and continuation of invasion fronts through the pushed-to-pulled transition.  The front is given as a sum of three terms: $\chi_-(x)$ which provides the stable state in the wake of the front (the black curve in Figure~\ref{fig:nagumocore}), $(\alpha x+\beta )e^{-x}$ which gives the far-field term (the red curve) and the core function $w(x)$ (the green curve).

\begin{figure}
    \centering
     \subfigure{\includegraphics[width=0.4\textwidth]{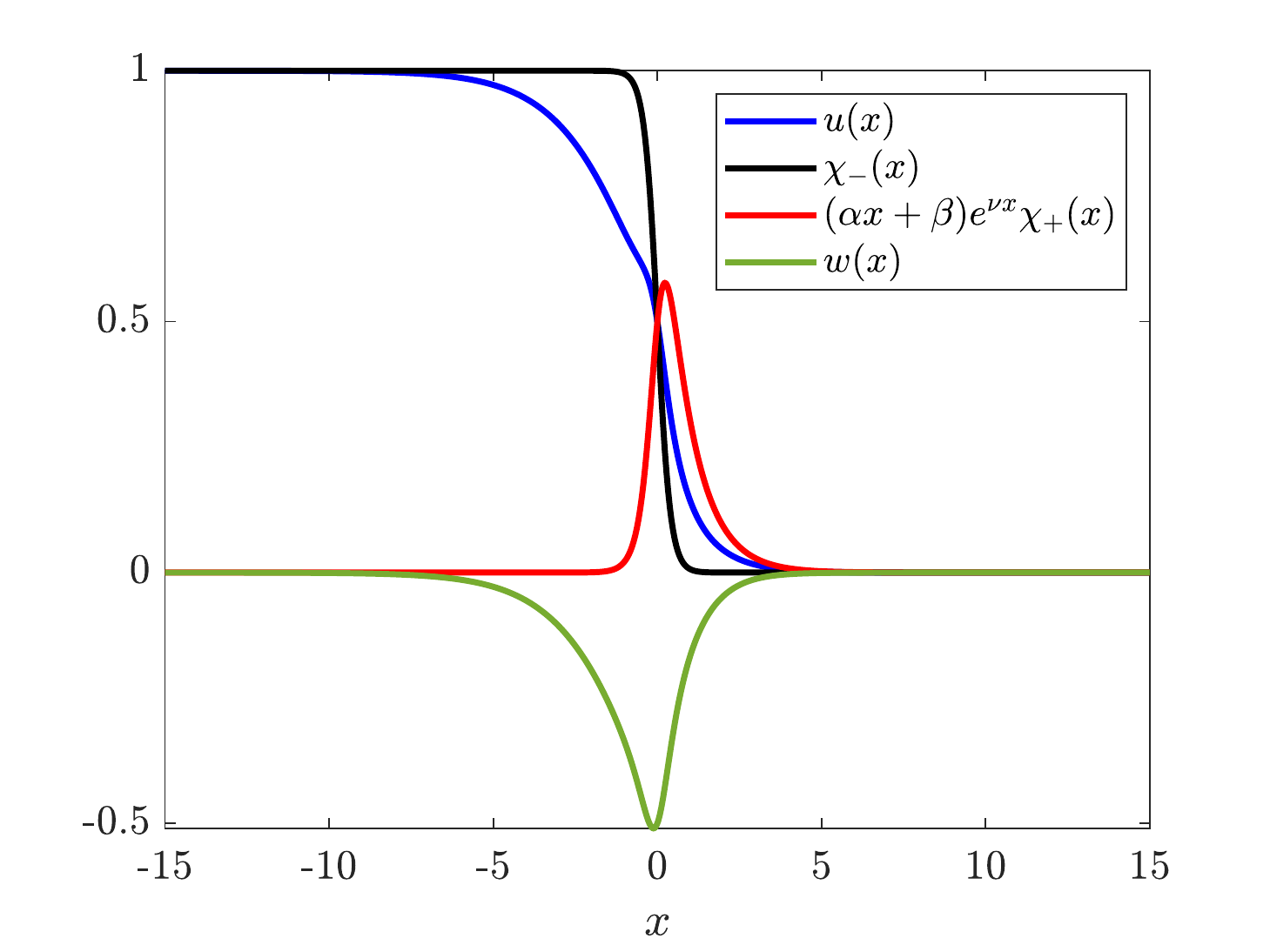}}\qquad	
  \subfigure{\includegraphics[width=0.4\textwidth]{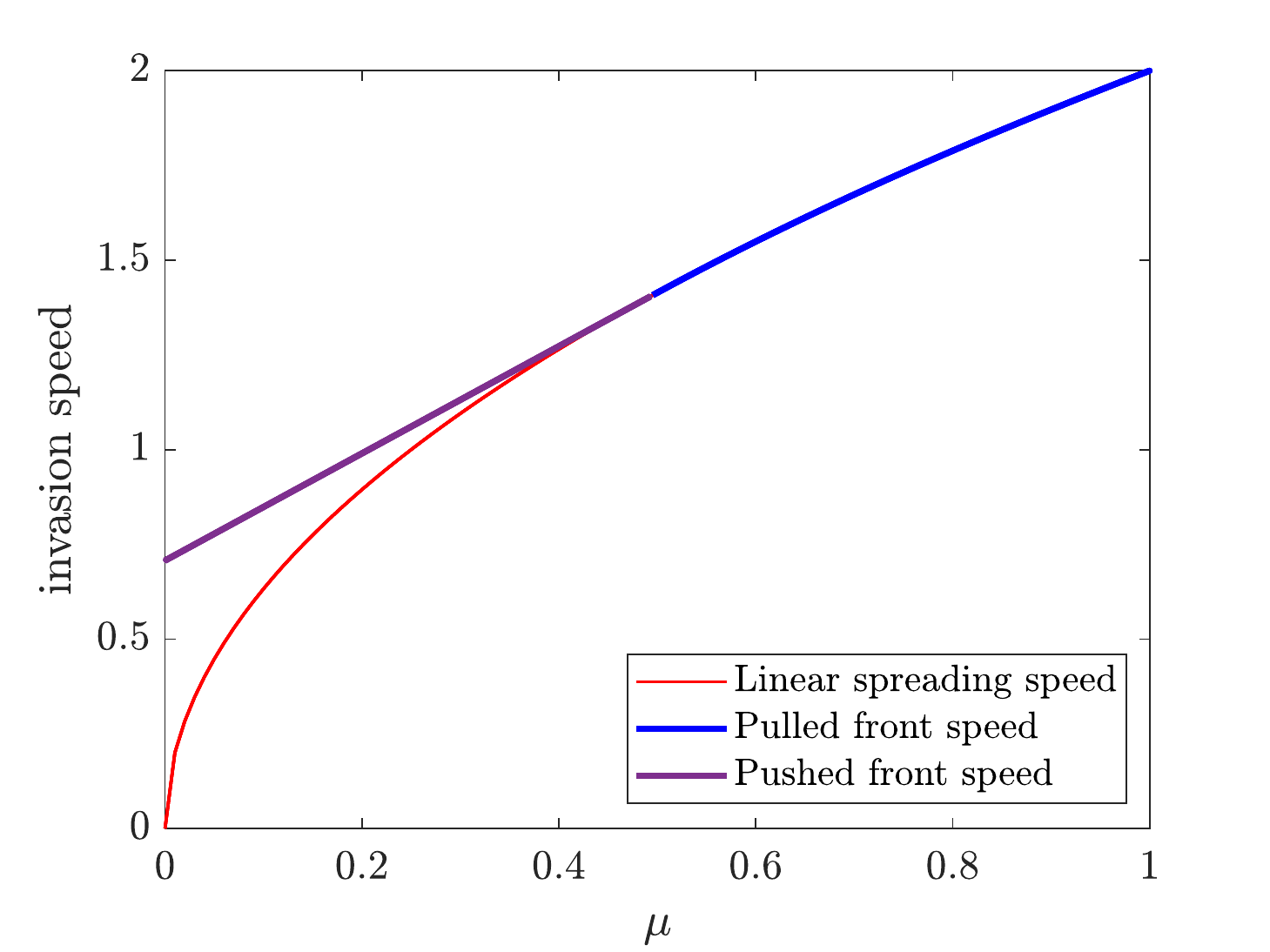}}
   \caption{ Left: Decomposition of the Nagumo front \eqref{eq:nagumo} according to \eqref{eq:ffcore}, with  full front solution (blue), exponential far-field component (red), localized  core (green) and left correction $\chi_-(x)$ (black).  Right: Pushed-to-pulled transition for (\ref{eq:nagumo}).  For $\mu>0.5$ the invasion front is pulled and propagates with the linear spreading speed (red);  at $\mu=0.5$ the invasion front transitions from pulled to pushed;  pushed invasion speed is hown for $\mu<0.5$ (purple).  }
    \label{fig:nagumocore}
\end{figure}

We confirm that the algorithm is able to correctly identify the critical $\mu$ value at which the pushed-to-pulled transition occurs.  This convergence is observed to be exponential in $L$ in Figure~\ref{fig:nagumoLerror}.  We remark that even for a fairly large spatial discretization value of $dx=0.1$, the algorithm with fourth-order accurate discretizations is able to obtain the critical value of $\mu$ to five correct decimal places with only a moderately sized domain ($L=16$).  Thus, only approximately $300$ gridpoints are required and computational times are modest.  In contrast, second order finite differences with $dx=0.1$ are, as expected, only able to obtain two correct decimal places.  This is emphasized further in Figure~\ref{fig:nagumoLerror} where convergence of the observed critical $\mu$ values with respect to changes in $dx$ are shown.  

\begin{figure}
    \centering
     \subfigure{\includegraphics[width=0.4\textwidth]{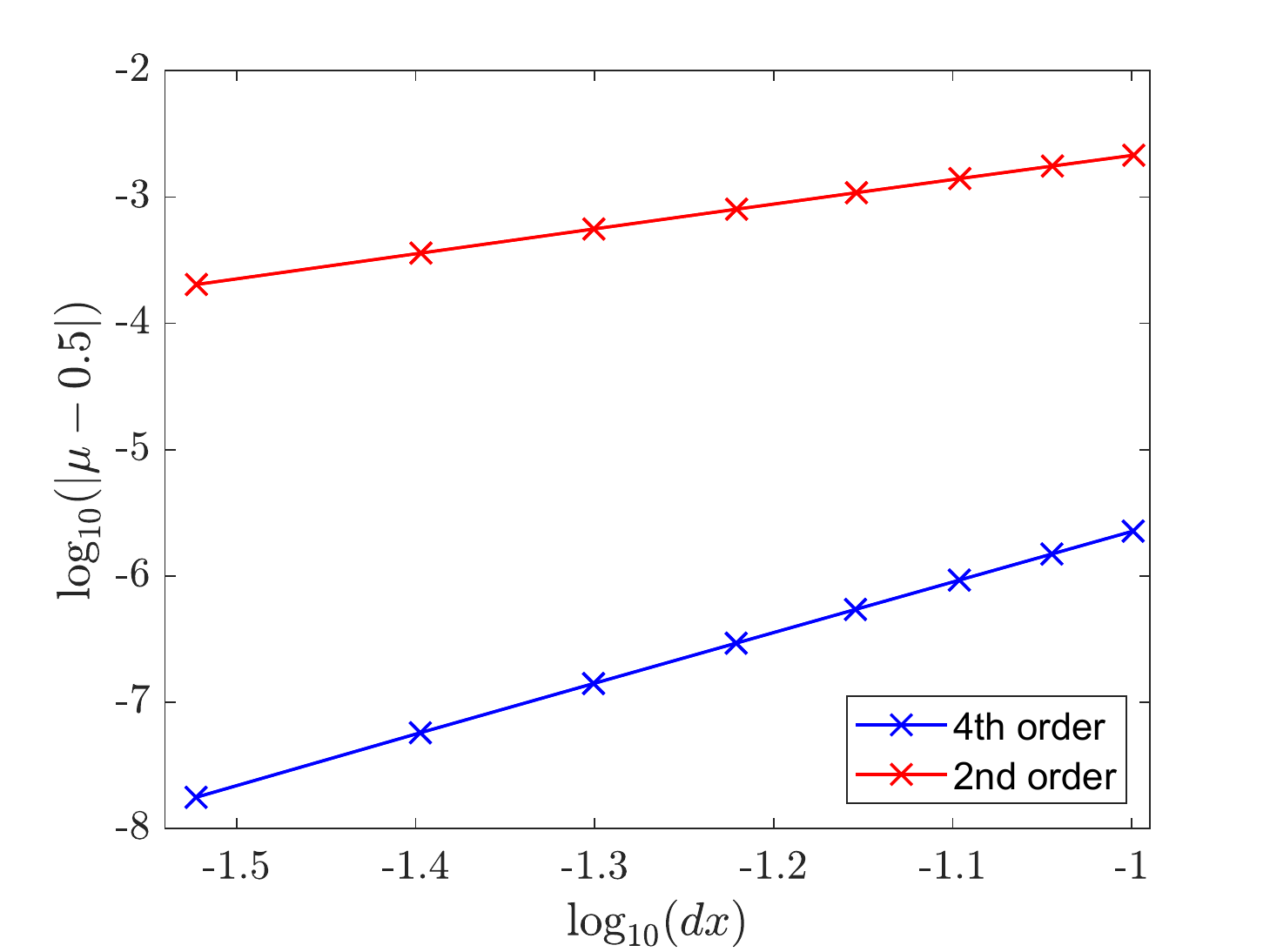}}\qquad
 \subfigure{\includegraphics[width=0.4\textwidth]{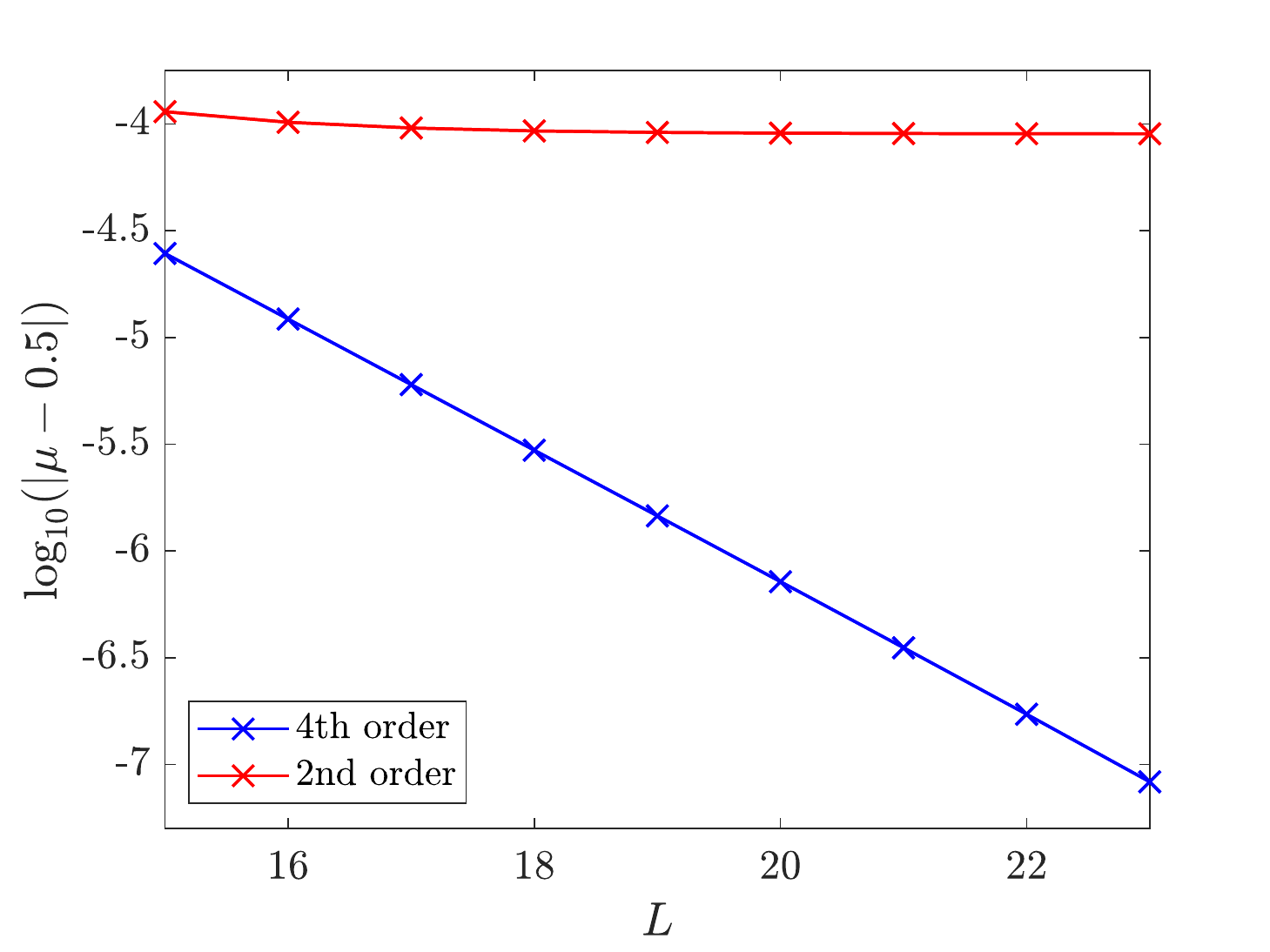}}
   \caption{Left: Convergence of the predicted pulled-to-pushed transition value for Nagumo's equation (\ref{eq:nagumo}) as the spatial discretization is varied for second-order (red) and fourth-order (blue) discretization, with slopes corresponding to the predicted slopes 2 and 4, repectively.  Right: Errors in the computed transition value as the length of the spatial domain $L$ is increased, fixing $dx=0.02$ for both second and fourth order discretizations, with exponential convergence for fourth-order discretization and saturation at discritization error for second-order discretizations.
} 
    \label{fig:nagumoLerror}
\end{figure}


Continuation through the critical value of $\mu=0.5$ demonstrates that the pulled-to-pushed transition is marked by a loss of monotonicity in the leading edge of the traveling wave profile.  This can be observed in Figure~\ref{fig:nagumopulled}.  For $\mu>0.5$ the traveling wave is monotone and positive while for $\mu<0.5$ this traveling front persists but has lost monotonicity due to a change in the sign of $\alpha$.  Taking a different approach we can also switch at $\mu=0.5$ to continue the pushed front out of the transition point.  This continuation is shown in Figure~\ref{fig:nagumopushed} where accurate, linear in $\mu$, values for the pushed invasion speed are obtained.  In the other direction, we demonstrate that the pushed front can be naturally continued through the bifurcation point at $\mu=0.5$ to obtain the family of weakly decaying super-critical fronts that propagate with speeds greater than the linear spreading speed; see again Figure~\ref{fig:nagumopulled}.  
\begin{figure}
    \centering
     \subfigure{\includegraphics[width=0.4\textwidth]{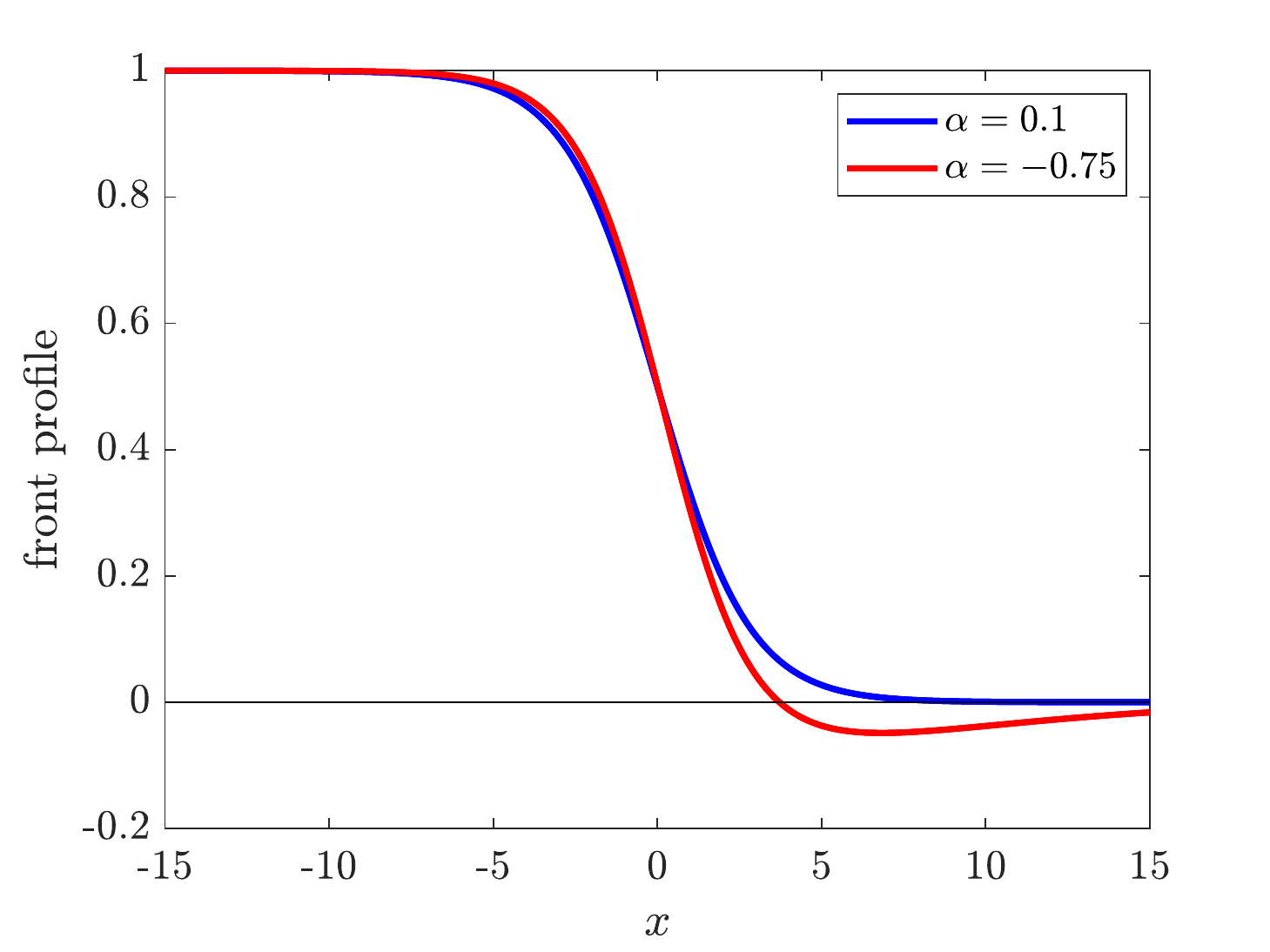}}\qquad
  \subfigure{\includegraphics[width=0.4\textwidth]{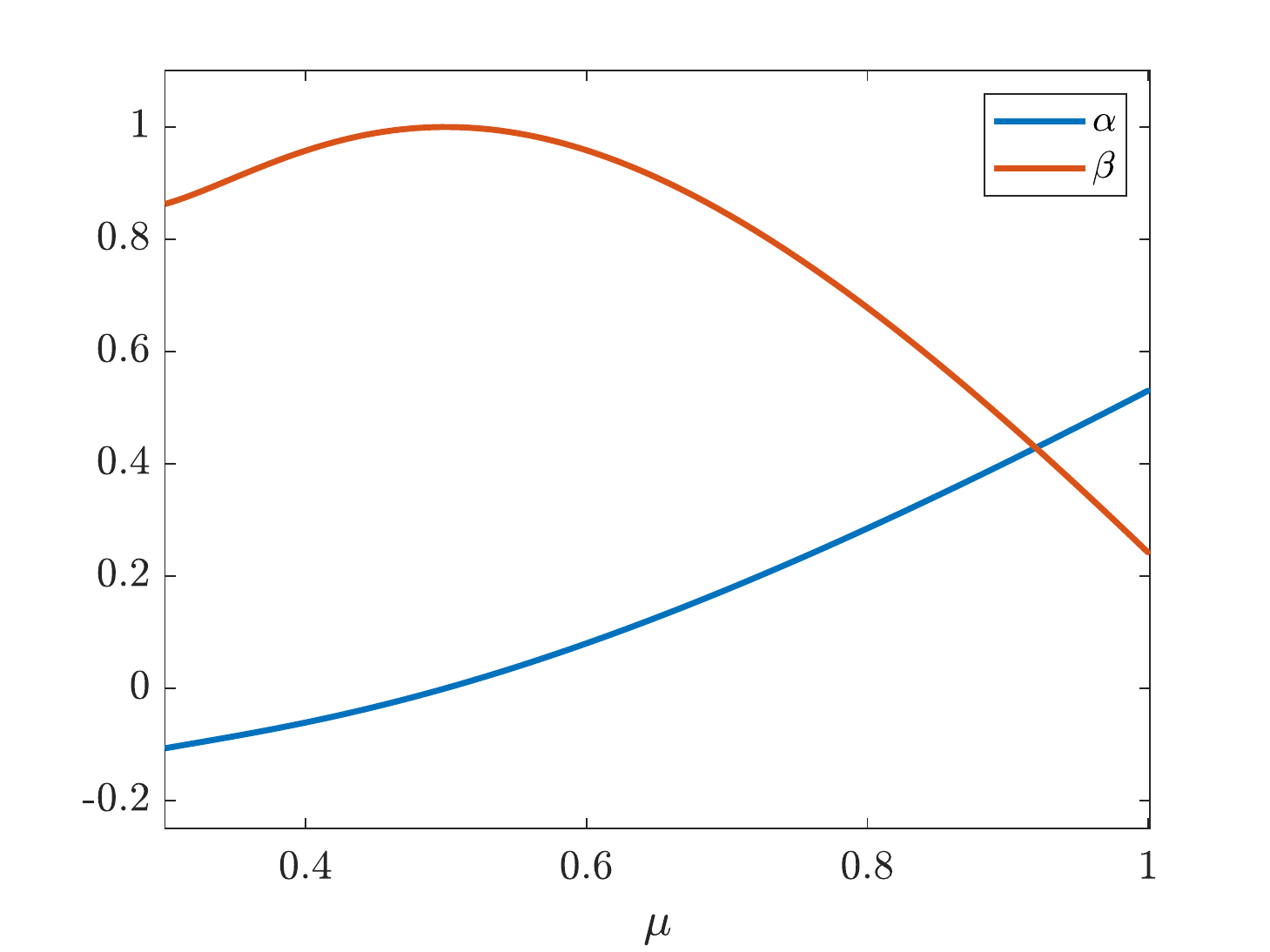}}
   \caption{Left: Front profiles for Nagumo's equation (\ref{eq:nagumo}) for $\alpha=0.1$ ($\mu\approx 0.620$) and $\alpha=0.75$ ($\mu\approx 0.126$) illustrating that the front loses monotonicity as $\alpha$ passes through zero.  Right: Computed $\alpha$ and $\beta$ values as $\mu$ is decreased from $1.0$.  Note that $\alpha=0$ at $\mu=0.5$ makes a transition from pushed-to-pulled fronts.  }
    \label{fig:nagumopulled}
\end{figure}
\begin{figure}
    \centering
     \subfigure{\includegraphics[width=0.4\textwidth]{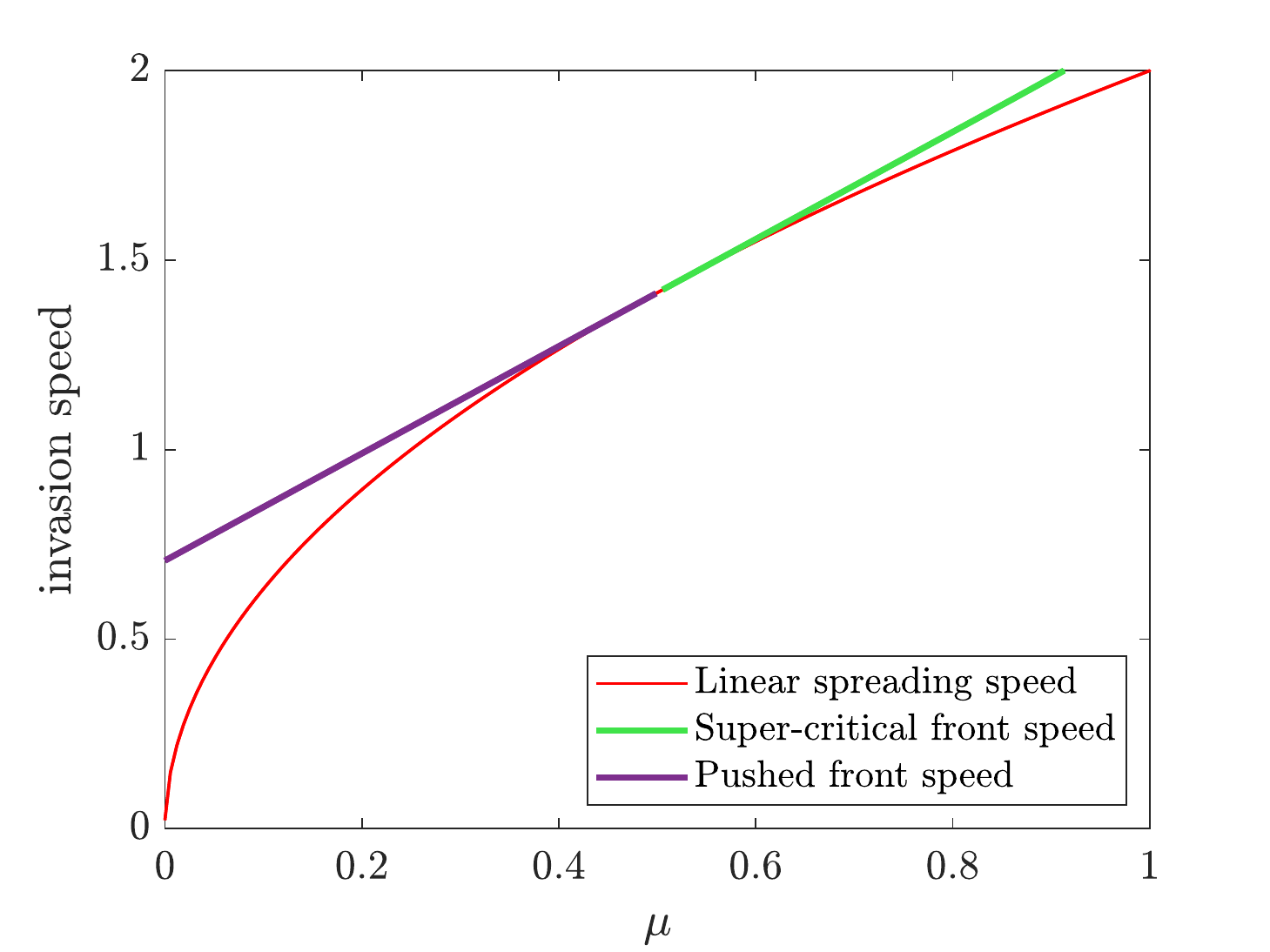}}\qquad
  \subfigure{\includegraphics[width=0.4\textwidth]{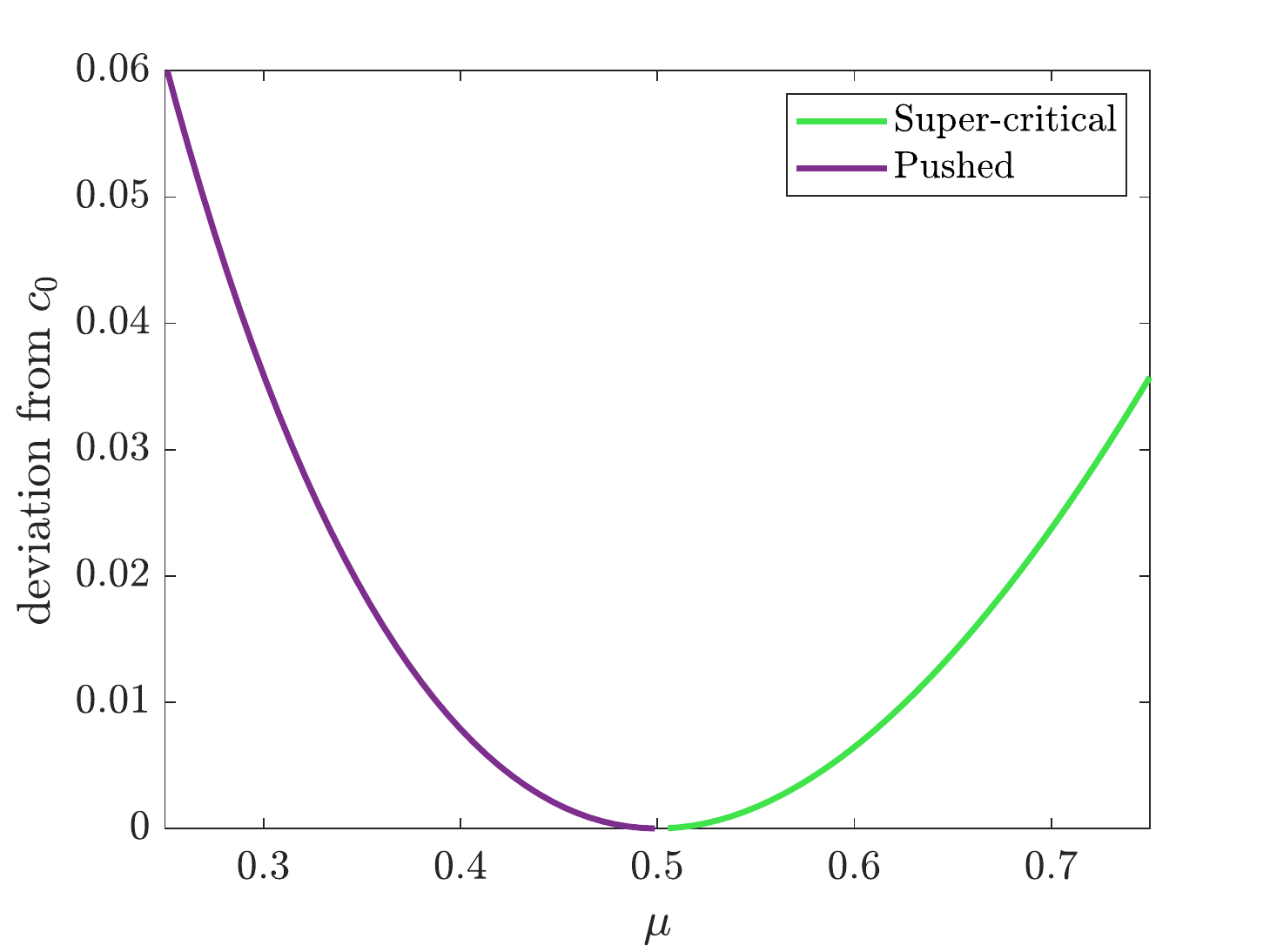}}
   \caption{Left: pushed front continuation through pushed-to-pulled transition follows the family of weakly decaying super-critical fronts for the Nagumo equation (\ref{eq:nagumo}).  Right: Comparison of pushed and  super-critical speeds to the linear spreading speed $2\sqrt{\mu}$ near the transition value $\mu=0.5$.    }
    \label{fig:nagumopushed}
\end{figure}
%
%
%
\paragraph{Fisher-KPP-Burgers equation.}
Next, we consider the Fisher-KPP-Burgers equation,
\begin{equation} u_t+\mu (uu_x)=u_{xx}+u-u^2,  \label{eq:kppburgers} 
\end{equation}
which was recently studied in \cite{an21}  with focus on convergence rates for pushed and pulled fronts, and fronts at the transition point. Front propagation is pulled for $\mu<2$ and pushed for $\mu>2$, with explicit pushed speed $c_\mathrm{ps}=\frac{\mu}{2}+\frac{2}{\mu}$. We focus here on  the numerical continuation of pushed fronts and show numerically determined front speeds and errors in Figure~\ref{fig:KPPburgers}.  Notably, errors are within $10^{-5}$ for $L=15$ and $dx=0.1$, so $N=300$. 

\begin{figure}
    \centering
     \subfigure{\includegraphics[width=0.4\textwidth]{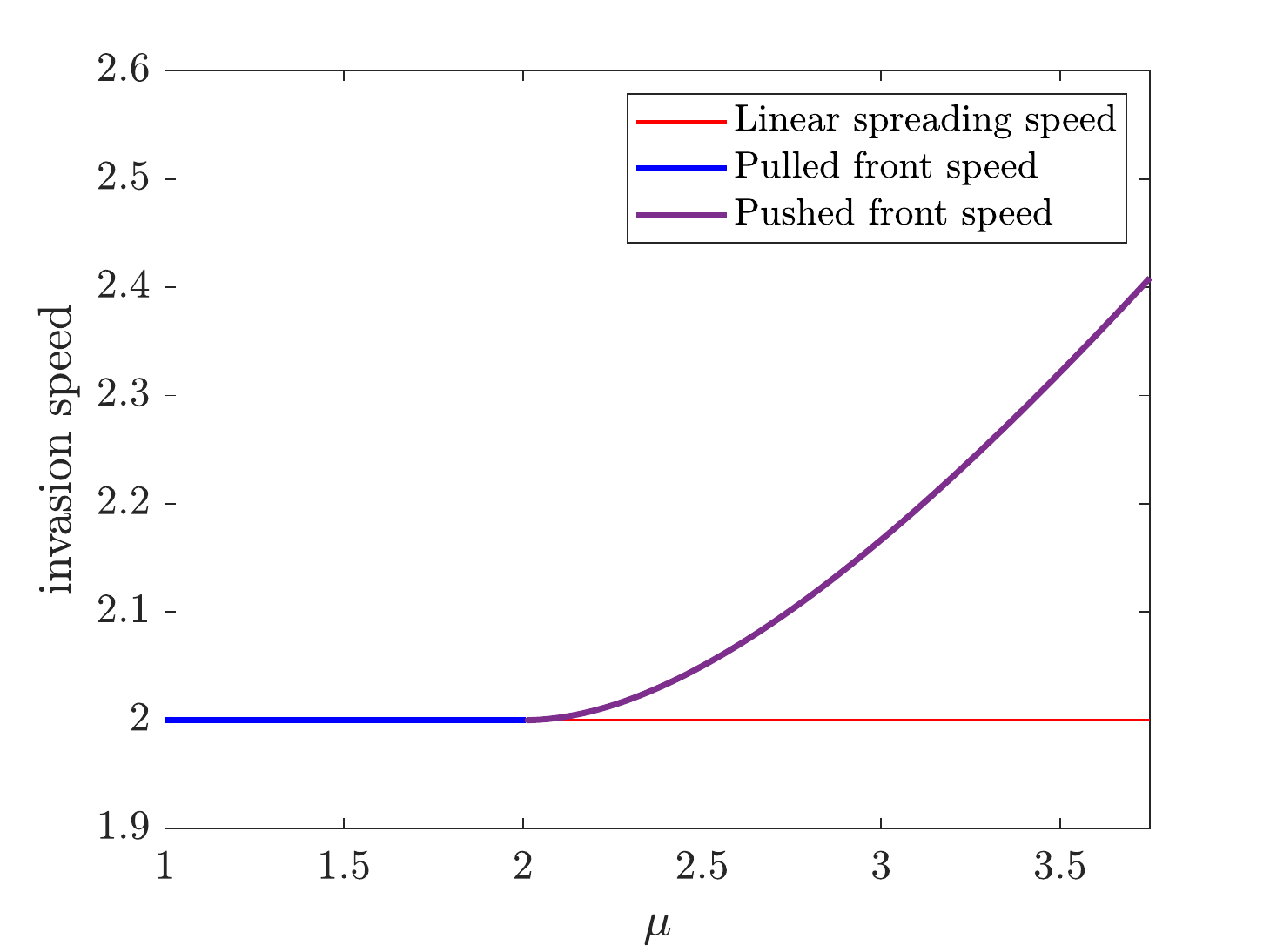}}\qquad
 \subfigure{\includegraphics[width=0.4\textwidth]{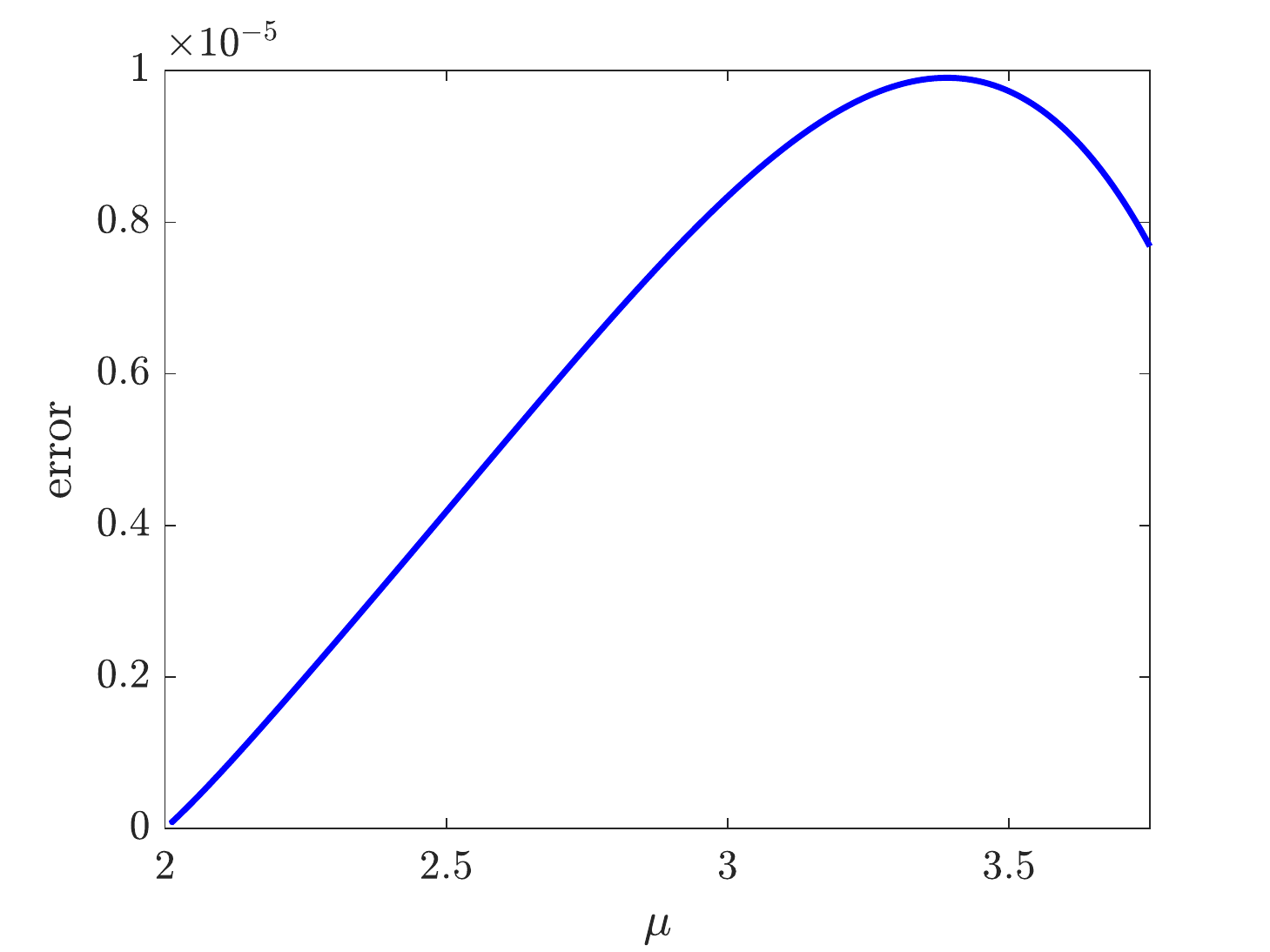}}
   \caption{Left: Pushed-to-pulled transition for Burgers-Fisher-KPP (\ref{eq:kppburgers}) as $\mu$ is varied.  Right: Error between  numerically determined wave speed and analytical speed $\frac{\mu}{2}+\frac{2}{\mu}$. Throughout, $L=15$ and  $dx=0.1$.  } 
    \label{fig:KPPburgers}
\end{figure}

\paragraph{The extended Fisher-KPP equation.}
A generalization of the Nagumo equation that incorporates a fourth-order diffusion term arises in several contexts, particularly as an amplitude equation near certain codimension-two points,
\begin{equation}
u_t = -\gamma u_{xxxx}+u_{xx}+u +\mu u^2-10 u^3; \label{eq:EFKPP}
\end{equation}
see  \cite{vanSaarloosReview} for references particularly in the context of front invasion. Clearly, the equation reduces to Fisher-KPP (or rather Nagumo's equation) at $\gamma=0$. The equation generates intriguing front-invasion dynamics, even for $\mu=0$, with stationary invasion for $\gamma<\frac{1}{12}$ and kink generation for $\gamma>\frac{1}{12}$.  We focus on the case $\gamma<\frac{1}{12}$ and study the transition from pulled to pushed fronts as $\mu$ is increased.  Numerically determined invasion speeds are shown in Figure~\ref{fig:EFKPP}.  

The dispersion relation for (\ref{eq:EFKPP}) 
\[ d(\lambda,\nu)=-\gamma\nu^4+\nu^2+c\nu+1, \]
allows for explicit formulas for speed and exponential decay but we find $(c,\nu)$ numerically as described above. Our continuation routine is able to locate a transition from pulled to pushed fronts as $\mu$ is increased.  Numerical results for two different values of $\gamma$ are provided in Figure~\ref{fig:EFKPP}.  It is worth noticing that the additional diffusion mechanism via a fourth order diffusion operator leads to an effective decrease in spreading speeds. For $\gamma$ small one may view \eqref{eq:EFKPP} as a singular perturbation of the classical Fisher-KPP equation. At $\gamma = 0$, simple scaling to match \eqref{eq:nagumo} predicts a pushed-to-pulled transition at $\mu = \sqrt{5}$. Using the methods of \cite{AveryGarenaux} to regularize the singular perturbation, one can rigorously establish a curve of pushed-to-pulled transitions $\mu_*(\gamma) = \sqrt{5} + \mathrm{o}(1)$ as $\gamma \to 0$, which our numerical algorithm confirms. 
%
%
%
\begin{figure}
    \centering
     \subfigure{\includegraphics[width=0.4\textwidth]{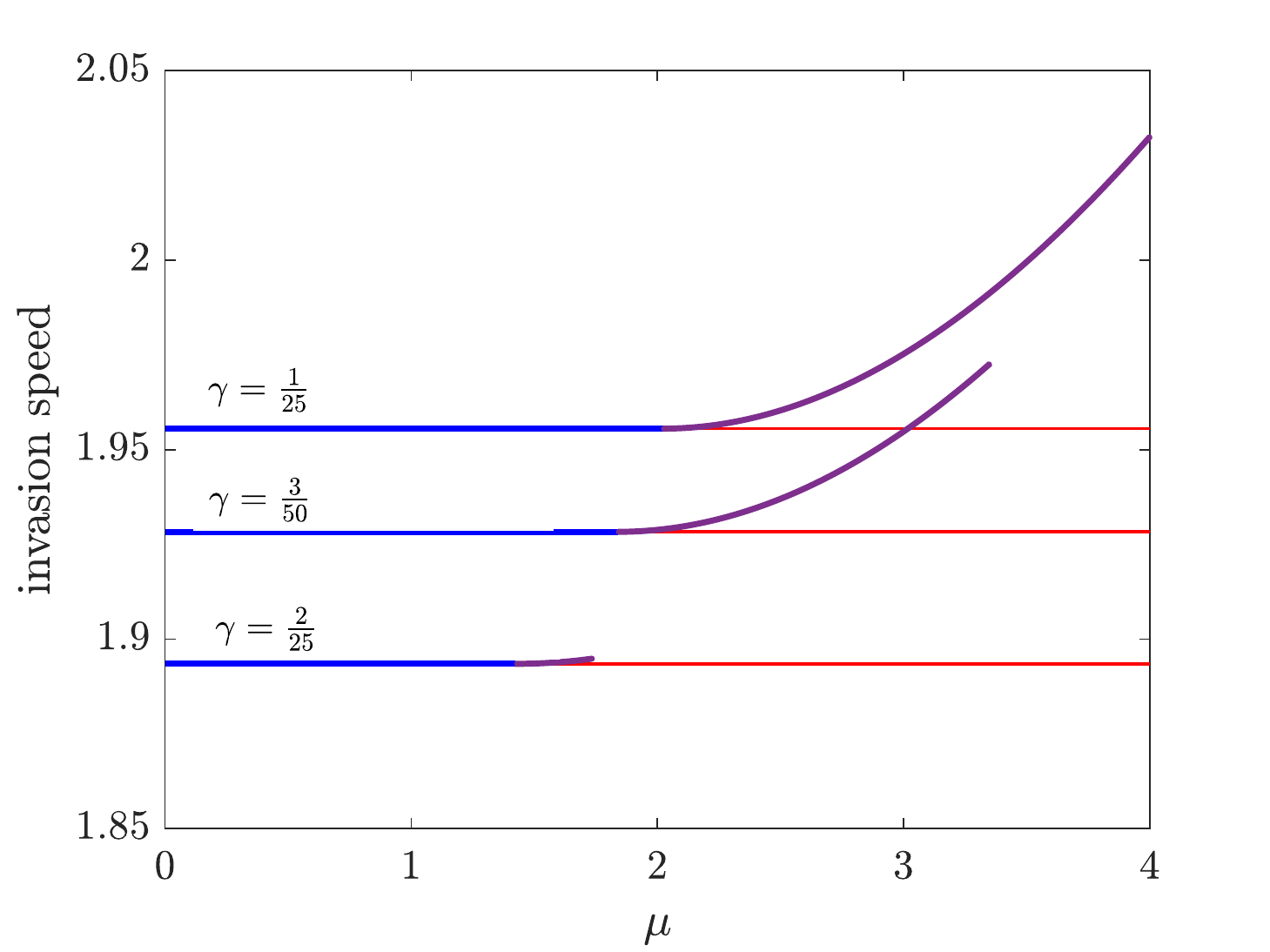}}\qquad
 \subfigure{\includegraphics[width=0.4\textwidth]{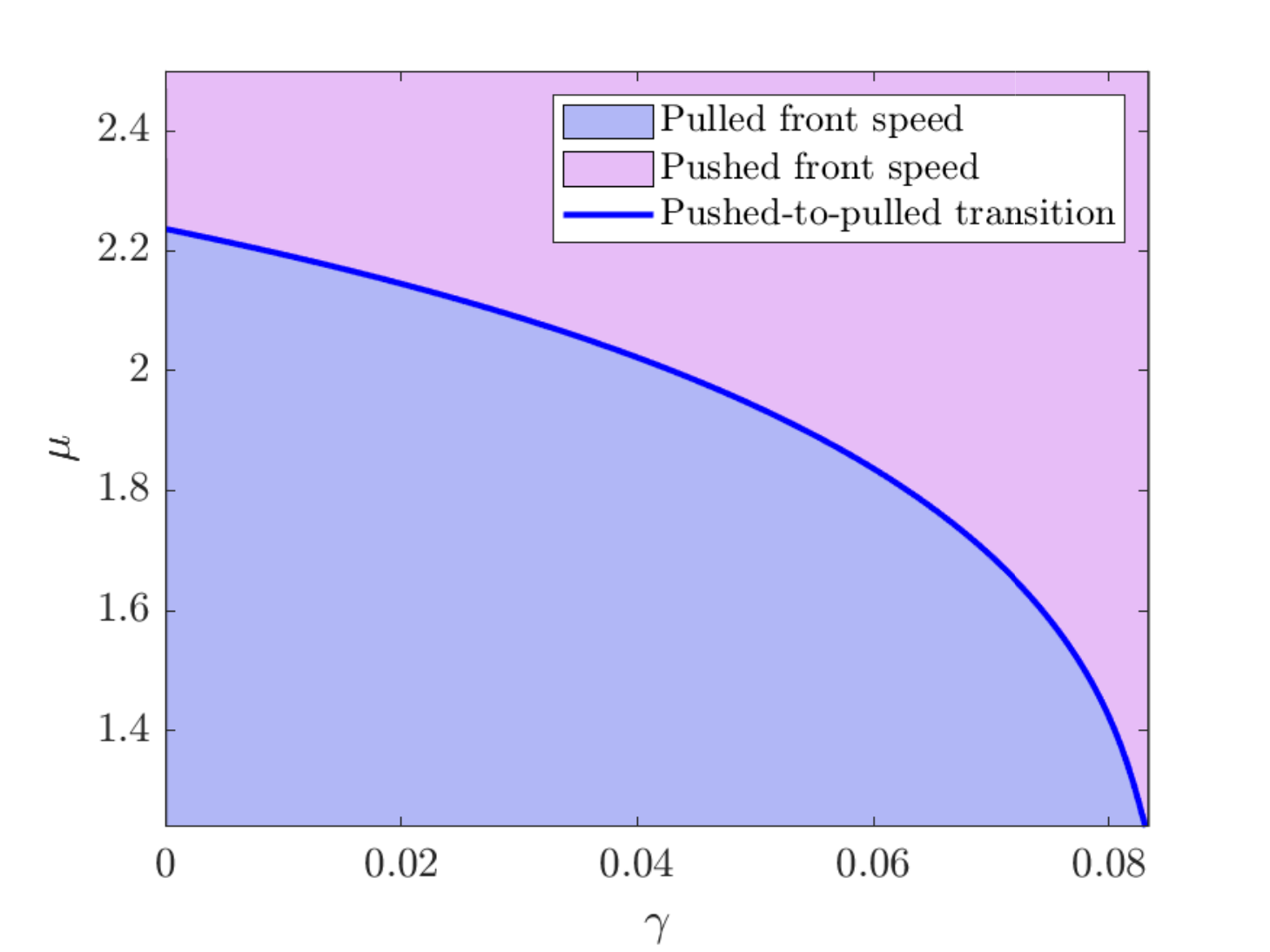}}
   \caption{Left: Pushed-to-pulled speed transitions for the Extended Fisher-KPP equation (\ref{eq:EFKPP}) for three different values of $\gamma$. Pushed fronts are continued until the relevant spatial eigenvalue forms a double root and the selected front loses monotonicity.  Right: Decomposition of $\gamma$-$\mu$ parameter space into pushed and pulled invasions for (\ref{eq:EFKPP}); note the limit of the transition on $\gamma=0$, $\mu=\sqrt{5}$ as predicted. } 
    \label{fig:EFKPP}
\end{figure}

\paragraph{A system modeling autocatalytic reactions.}
The following system of reaction-diffusion equations modeling autocatalytic reactions was considered in \cite{focant98} from a point of view of front propagation, 
\begin{eqnarray}
u_t&=& u_{xx}-uv-kuv^2 \nonumber \\
v_t&=& \sigma v_{xx}+uv+kuv^2. \label{eq:gallay}
\end{eqnarray}
The system possesses two lines of equilibria where $u=0$ or $v=0$.  We focus on fronts connecting the unstable state $(u,v)=(1,0)$ to the marginally stable state $(u,v)=(0,1)$.  Associated with the unstable state, we have the linearization
\[ A(\lambda,\nu,c,k)=\left(\begin{array}{cc} \nu^2+c\nu -\lambda & -1 \\ 0 & \sigma \nu^2+c\nu+1-\lambda \end{array}\right),\]
which is in upper triangular form, so that the  dispersion relation factorizes
\begin{equation} d(\lambda,\nu)=\left( \nu^2+c\nu-\lambda\right)\left(\sigma\nu^2+c\nu+1-\lambda\right),\label{eq:gallaydisp} \end{equation}
leading to a simple expression $ c_0=2\sqrt{\sigma}$ for the linear spreading speed (note however Remark~\ref{rem:anomalous} for some caveats to this consideration).  
Our algorithm splits solutions componentwise according to \eqref{eq:ffcore},
\begin{eqnarray}
u(x)&=& u_- \chi_-(x)+u_+\chi_+(x) +(\delta x+\gamma) e^{\nu x} \chi_+(x)+w_u(x) \nonumber  \\
v(x)&=& v_- \chi_-(x)+v_+\chi_+(x) +(\alpha x+\beta) e^{\nu x} \chi_+(x)+w_v(x) ,\label{eq:gallayansatz}
\end{eqnarray}
where $\delta$ and $\gamma$ are easily found in terms of $\alpha$ and $\beta$ explicitly. 
%
We impose a phase condition on the $u$ component only (either component would suffice) and the transversality condition \eqref{e:tr}.
Our continuation consists of two steps: first we fix $\sigma$ and continue in $k$ until we reach a value where $\alpha=0$ (the push-pulled transition value).  We then perform secant continuation while fixing $\alpha=0$ to continue the pushed-to-pulled transition in parameter space.   Due to a resonance at $\sigma=\frac{1}{2}$ our routine is unable to continue pulled fronts through this value and so continuation must be performed on either side of this value and then matched; see Remark~\ref{rem:anomalous}. Results are shown in Figure~\ref{fig:gallayex} and compare well with the numerical results in \cite{focant98}, where a shooting method was used to detect the transition from monotone to non-monotone tail behavior for fronts propagating at the linear speed. 
%
\begin{figure}
    \centering
     \subfigure{\includegraphics[width=0.4\textwidth]{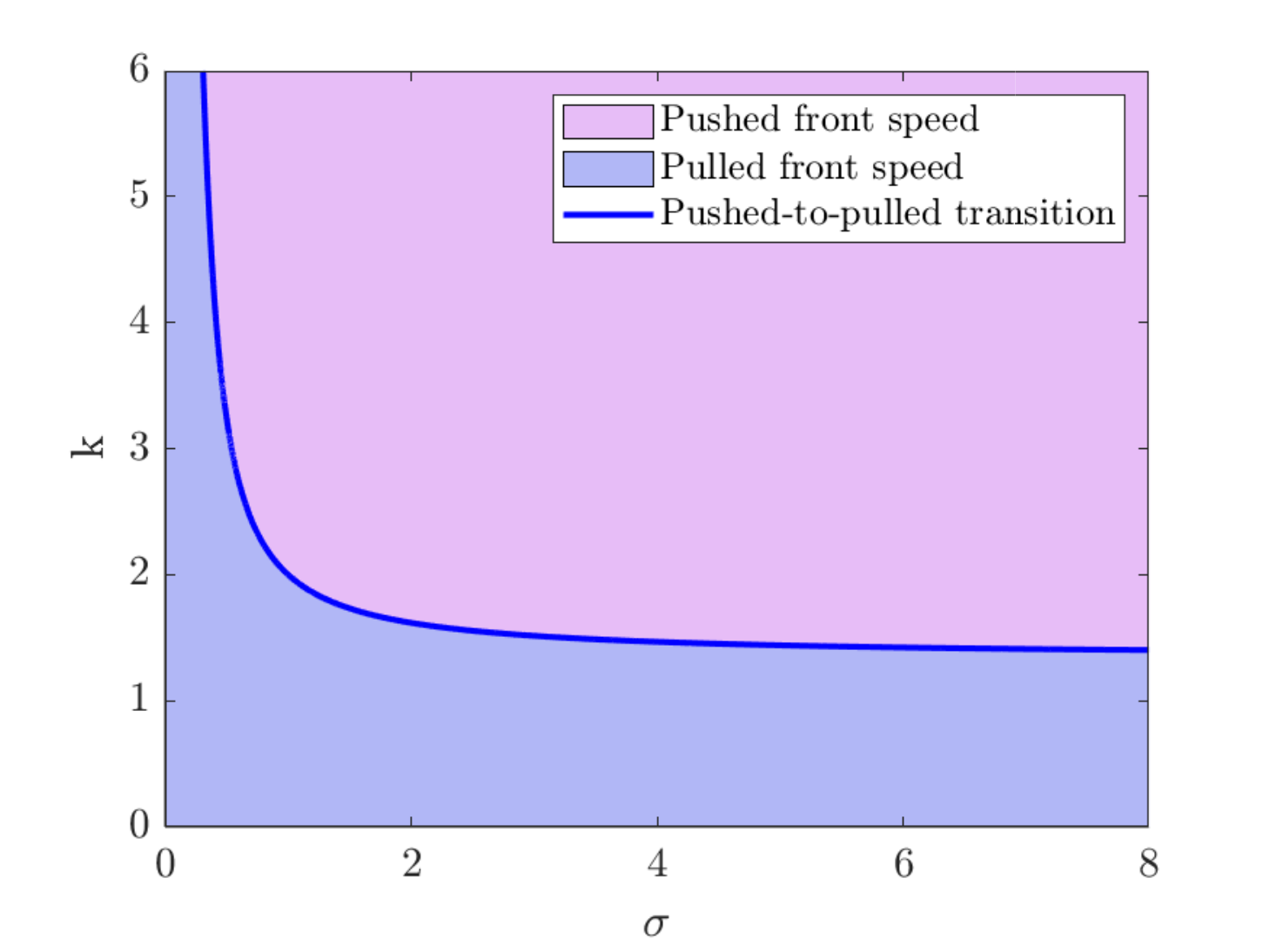}}\qquad
 \subfigure{\includegraphics[width=0.4\textwidth]{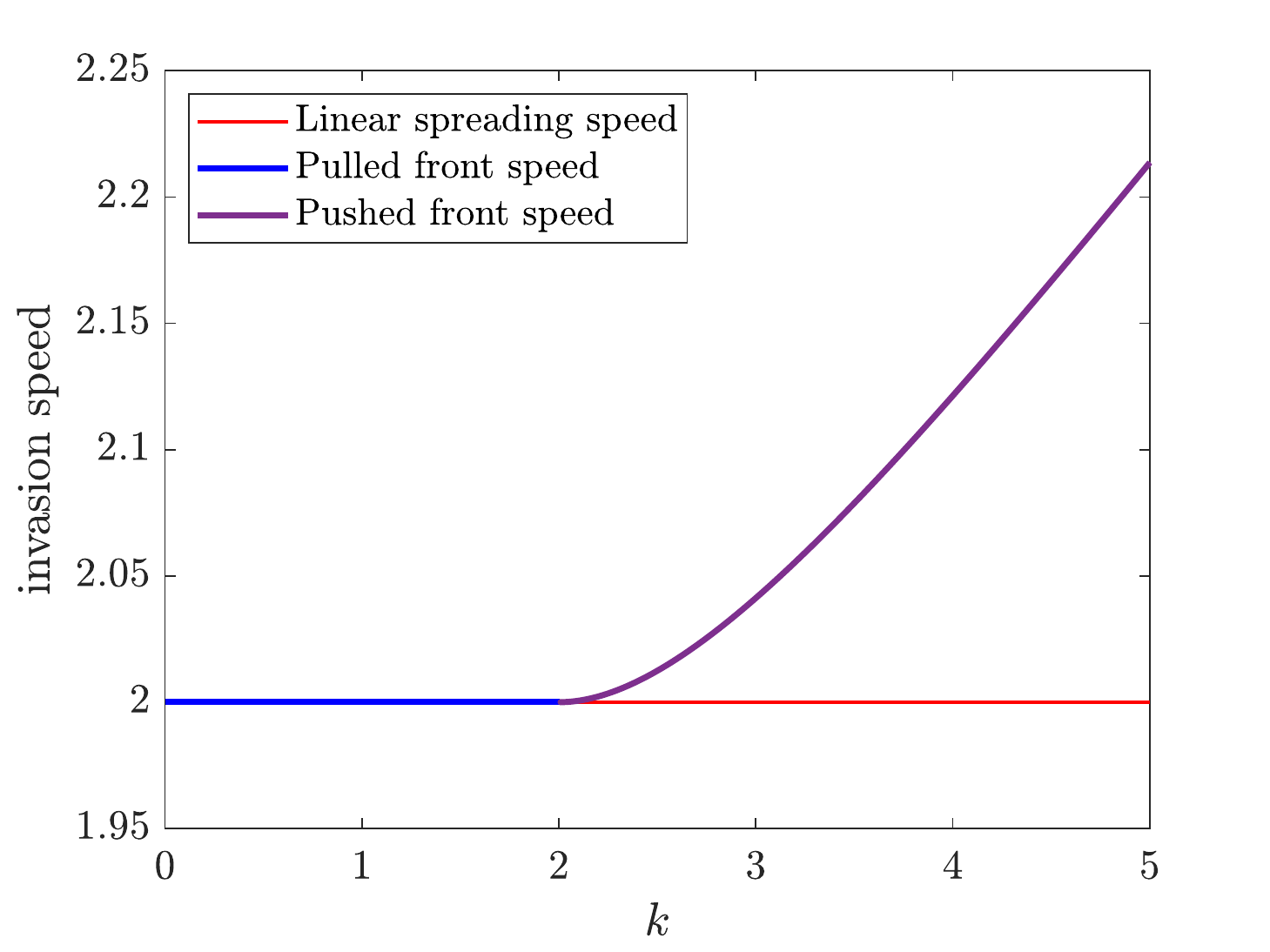}}
   \caption{Left: Pushed-to-pulled transition in $\sigma$-$k$ parameter space for the system of equations modeling autocatlytic reactions given in  (\ref{eq:gallay}).  The data on the left represents two separate continuation runs -- one that starts near $\sigma=4.0$ and decreases to $\sigma=0.5$ and a second that starts at $\sigma=0.45$ and continues down in $\sigma$.  In this example, we use $L=20$ and $N=400$ for an spatial discretization $dx=0.1$.  Right: Pushed-to-pulled transition for (\ref{eq:gallay}) for fixed $\sigma=1.0$ and varying $k$.   }
    \label{fig:gallayex}
\end{figure}

\begin{remark}\label{rem:anomalous}  The dispersion relation (\ref{eq:gallaydisp}) has four roots with two from the linearization of the $u$ component $\nu_u^\pm=-\frac{c}{2}\pm\frac{1}{2}\sqrt{c^2-4\lambda}$ and two from the $v$ component $\nu_v^\pm=-\frac{c}{2\sigma}\pm\frac{1}{2\sigma}\sqrt{c^2-4\sigma\lambda}$.  Note that when $\sigma=\frac{1}{2}$ then $\nu_u^-(0)=\nu_v^\pm(0)=-\sqrt{2}$.  This resonance between the eigenvalues implies that the pulled front ansatz in the $u$ component is insufficient and would require a term that is quadratic in $x$. Thus, the resonance prevents the continuation of pulled fronts through $\sigma=\frac{1}{2}$ as $|\delta|\to\infty$ as $\sigma\to \frac{1}{2}$.  

We also note that when $\sigma<\frac{1}{2}$ the pinched double root no longer contains the weakest decaying stable roots since $\nu_v^\pm(0)<\nu_u^-(0)$.  While this precludes a direct application of our main theorems, one can adapt the functional analytic setting using different exponential weights in $u$- and $v$-components and obtain the same result in this setting. Numerically, the method applies without modification. 

We also remark in passing that when $\sigma<1$ there exists an unstable pinched double root caused by the resonance $\nu_u^-(\lambda)=\nu_v^+(\lambda)$.  This pinched double root is in general related to anomalous spreading speeds but is in the present case {\em irrelevant} in the terminology of \cite{holzeranomalous,HolzerScheelPointwiseGrowth}  so that the selected invasion speed is still the linear speed $2\sqrt{\sigma}$ in this case.  More information on relevant and irrelevant speeds and how these are enabled or disabled by linear or nonlinear coupling terms related to resonances in the linear dispersion relation can be found in \cite{FayeHolzerScheel}.
\end{remark}

\paragraph{A Keller-Segel model with repulsive interaction.}
We next consider the Keller-Segel model for chemotactic motion with logistic population growth,
\begin{eqnarray}
u_t&=& u_{xx}+\chi (uv_x)_x +u(1-u) \nonumber \\
0&=& \sigma v_{xx}+u-v, \label{eq:ks}
\end{eqnarray}
for a population $u$ of bacteria which produce a rapidly diffusing and decaying chemical agent $v$, and which move with speed proportional to the gradient of the chemical signal $v_x$. 
When $\chi=0$ the first equation decouples and is simply the Fisher-KPP equation with pulled invasion speed $c=2$. This linear predicted speed is unaffected by the nonlinear coupling when $\chi\neq 0$.  We focus here on the case $\chi>0$ so that the quadratic term $\chi (uv_x)_x$ is repulsive, modeling negative taxis, that is, motion of bacteria away from locations of high signal concentrations, and thereby away from high concentrations of other bacteria.  Such an effect could clearly be able to accelerate the spreading of bacteria, enhancing the diffusive motion of bacteria from a high-concentration region where $u\sim 1$ to the unstable region where $u\sim 0$. Note however that the chemotactic effect is essentially quadratic in the amplitude of $u$ and it is therefore not a priori clear how it would compete with the quadratic saturation of growth $-u^2$. 

For $\chi$ sufficiently large this repulsive effect does indeed lead to propagation via pushed fronts, as was recently demonstrated in  \cite{henderson21}, particularly when 
%
$\chi>2$ and $\sigma\gg 1$. 

We investigate this pushed-to-pulled transition in $\sigma$-$\chi$ parameter space using our computational approach. 
The linearization at the unstable state $(0,0)$ is in triangular form, 
\[ A(\lambda,\nu,c,\sigma)=\left(\begin{array}{cc} \nu^2+c\nu+1-\lambda & 0 \\ 1 & \sigma\nu^2-1\end{array}\right), \]
with dispersion relation factored as
\[ d(\lambda,\nu)=(\nu^2+c\nu+1-\lambda)(\sigma\nu^2-1). \]
The linear spreading speed is always the Fisher-KPP speed $c_0=2$. The continuation set-up in this case is similar to that of the system (\ref{eq:gallay}).  The main difference is that the $v$ component lacks any advective terms. We first continue in $\chi$ from $\chi=0$ until the pulled-to-pushed transition and then continue the transition in two parameters; see Figure~\ref{fig:ksEX}.  Due to the skew-product nature of the dispersion relation we again find resonances that prevent direct continuation of the pushed-to-pulled transition curve over all values of parameters.  In particular, when $\sigma=1$, one root of the $v$ component aligns with the critical decay rate $\nu=-1$.  Direct continuation of the transition through this value of $\sigma$ is not possible using our routine.
\begin{figure}
    \centering
     \subfigure{\includegraphics[width=0.43\textwidth]{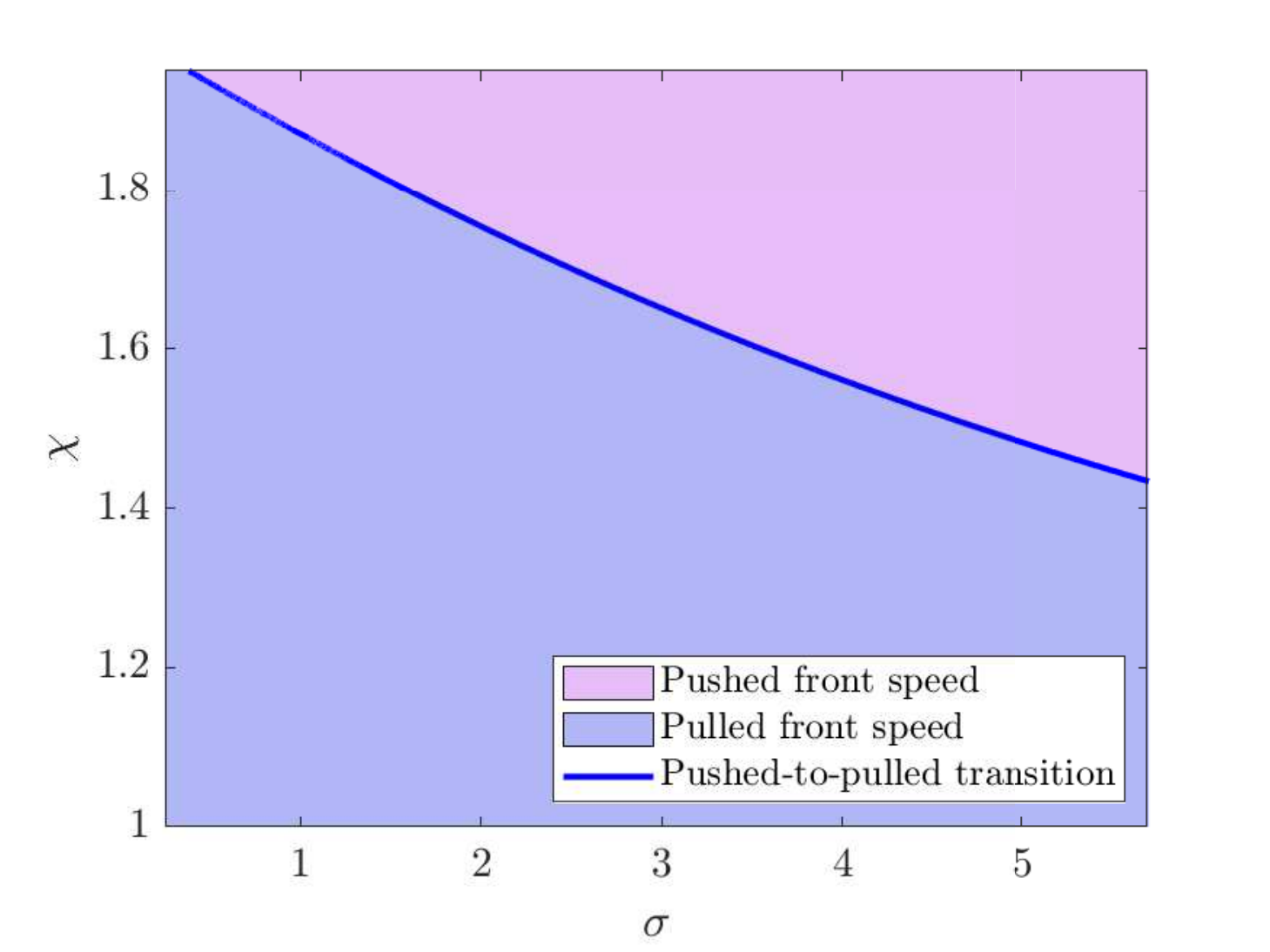}}\qquad
     \subfigure{\includegraphics[width=0.43\textwidth]{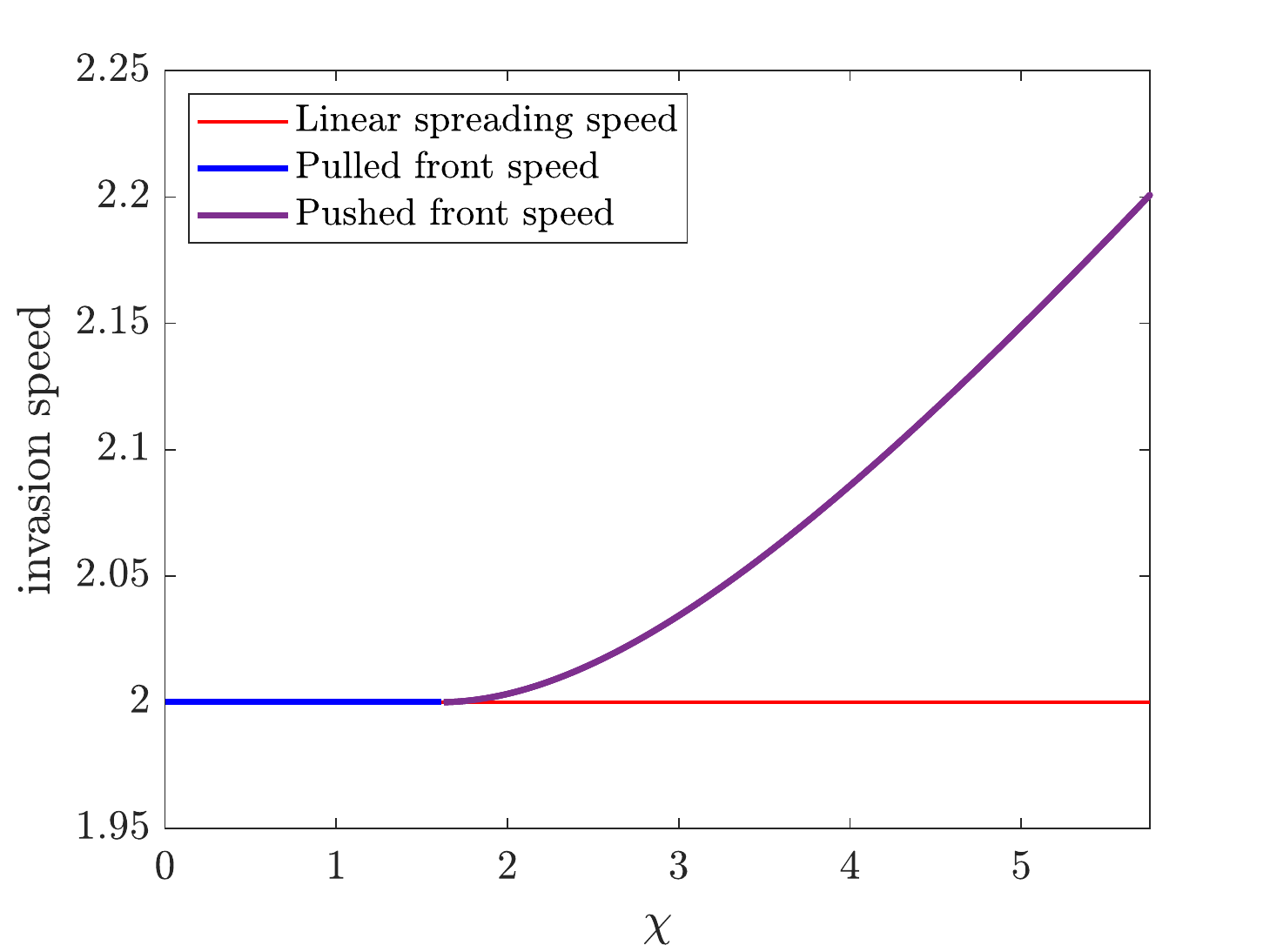}}
   \caption{ Left: Pushed-to-pulled transition in $\sigma$-$k$ parameter space for the Keller-Segel system  (\ref{eq:ks}).  Right: Pulled and pushed invasion speed for (\ref{eq:ks}) with $\sigma=3.50$ fixed.  }
    \label{fig:ksEX}
\end{figure}

\paragraph{Lotka-Volterra competition model.}
As a final example, we briefly consider the Lotka-Volterra competition model,
\begin{eqnarray}
u_t&=& u_{xx}+u(1-u-a_1v) \nonumber \\
0&=& \sigma v_{xx}+rv(1-a_2u-v). \label{eq:LV}
\end{eqnarray}
When $a_1<1<a_2$ the equilibrium point $(0,1)$ is unstable and the dispersion relation is
\[ A(\lambda,\nu,c,\sigma)=\left(\begin{array}{cc} \nu^2+c\nu+1-a_1 -\lambda & 0 \\ -ra_2   & \sigma\nu^2+c\nu-r-\lambda\end{array}\right). \]
The linear spreading speed is $c_0=2\sqrt{1-a_1}$, although we note that there is a faster linear invasion speed which is not relevant; see \cite{holzer12} for a detailed discussion of this phenomena.  

Traveling front solutions of (\ref{eq:LV}) have been studied by many authors; see for example \cite{hosono98}. A particular focus has been on locating regions in parameter space where the invasion fronts are pulled or pushed.  In Figure~\ref{fig:LV}, we demonstrate how our routine can be used to quickly determine numerical approximations to these pushed-to-pulled transition curves.  Here we fix two parameters ($\sigma=1.0$ and $a_1=0.5$) and vary the remaining two parameters ($a_2$ and $r$).  We compare our numerically determined transition curve to known analytical results presented in \cite{lewis02} where it is shown that the invasion front is pulled when $a_2<\frac{2}{r}+2$ (with $\sigma=1.0$ and $a_1=0.5$ fixed).  It was conjectured in \cite{hosono98} that the blue pushed-to-pulled transition curve has a vertical asymptote at $a_2=2$ (for $a_1=0.5$).  Our numerical continuation, while not necessarily accurate for arbitrarily large values of $r$, suggests that this asymptote occurs for some value of $a_2>2$.  


\begin{figure}
    \centering
     \subfigure{\includegraphics[width=0.43\textwidth]{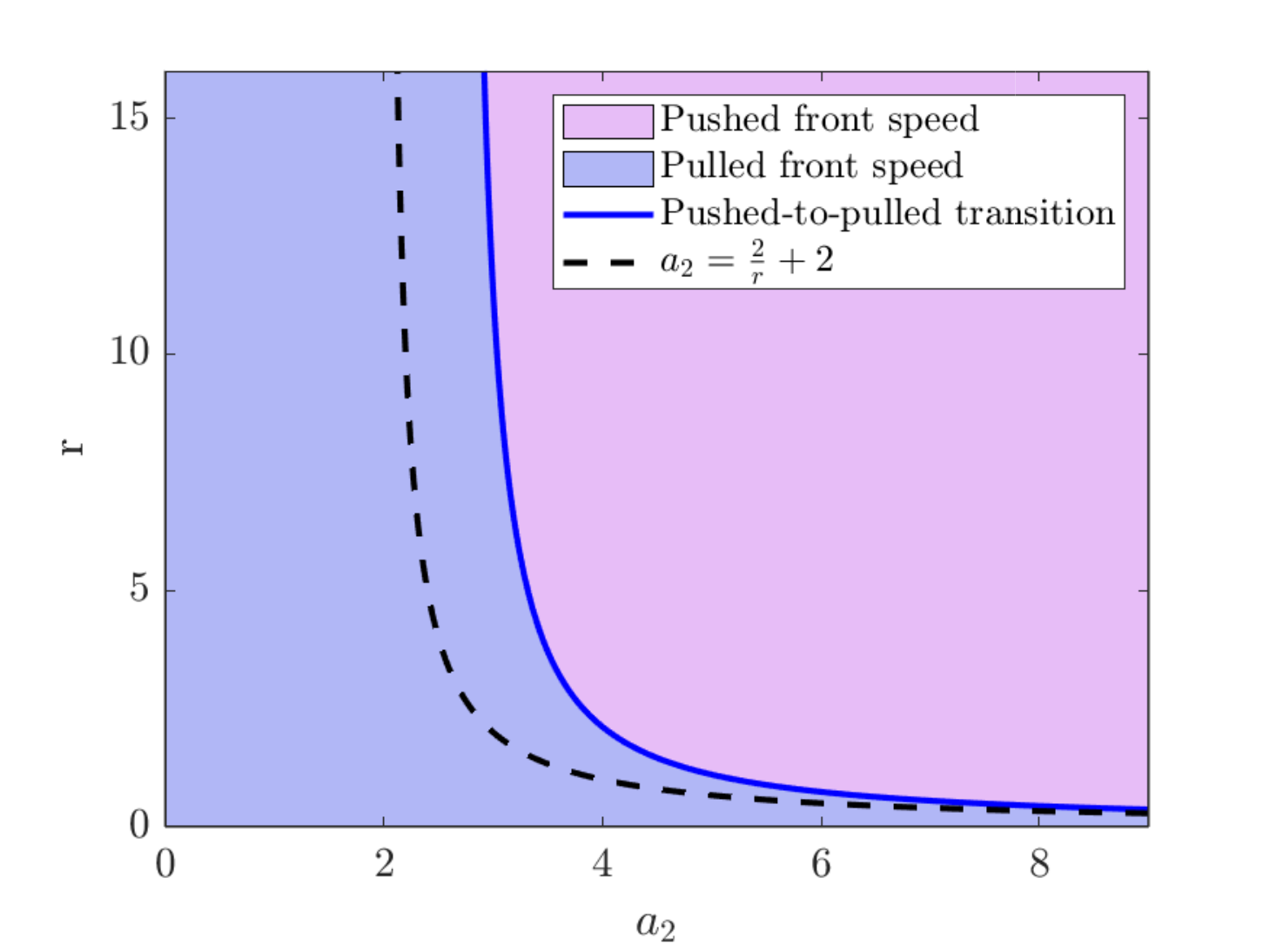}}\qquad
    \subfigure{\includegraphics[width=0.43\textwidth]{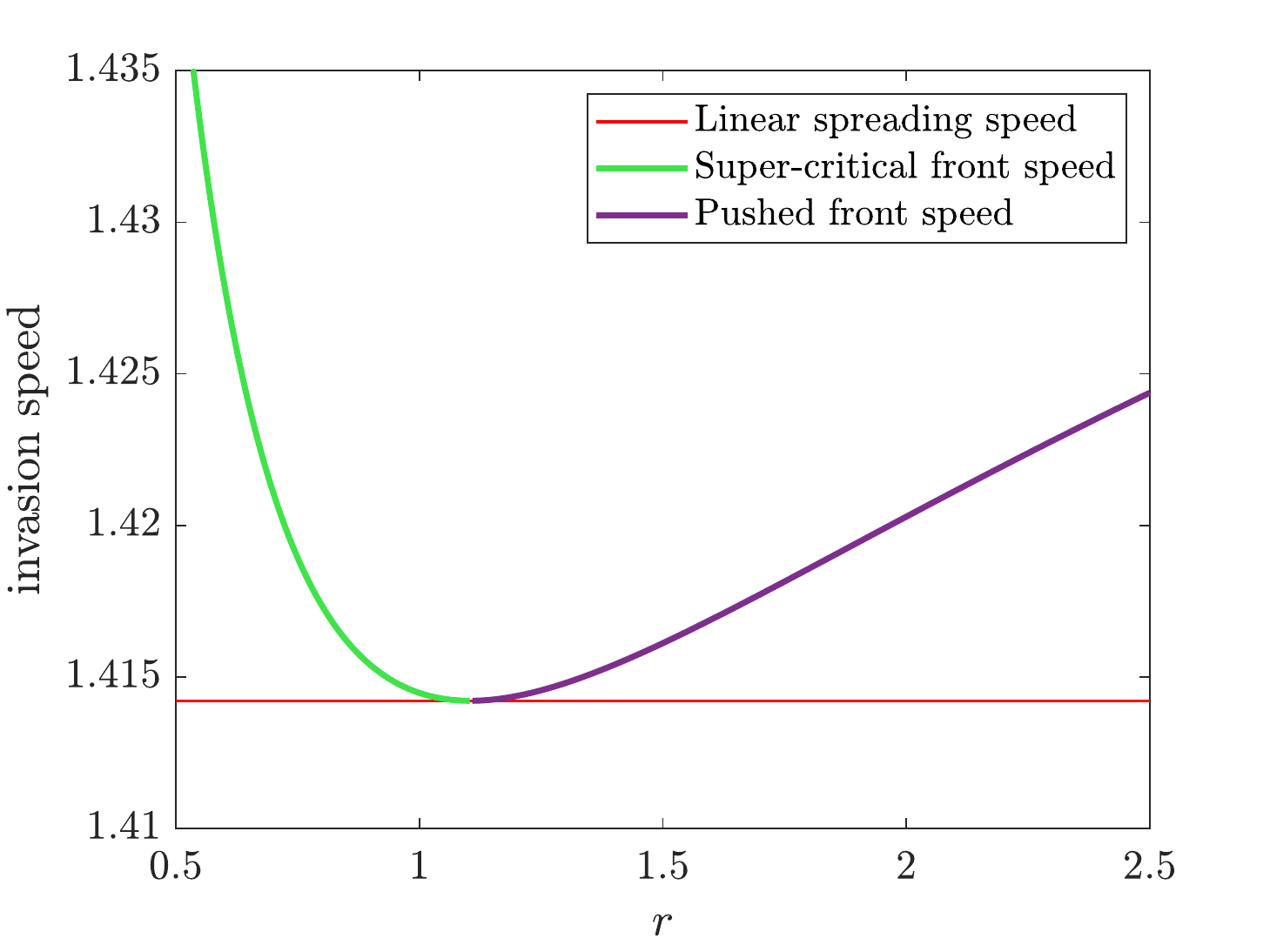}}
   \caption{Left: Pushed-to-pulled transition in $a_2$-$r$ parameter space for the Lotka-Volterra model (\ref{eq:LV}).  All other parameters are fixed to $\sigma=1.0$, $a_1=0.5$ and $a_2=5.0$. For parameters below the black dashed line the invasion process have been proved to be pulled; see \cite{lewis02}. Right: Continuation of pushed and super-critical fronts from the pushed-to-pulled transition point.  The red line indicated the linear spreading speed $2\sqrt{1-a_1}$ with $a_1=0.5$ for all computations.  
   }
    \label{fig:LV}
\end{figure}

%
%
%
%
%
%
%
%
%
%
%
%
\section{Discussion}
We presented an analytical and computational study of the transition from pushed to pulled front invasion. Both analytical and computational results rely on a far-field-core decomposition with explicit tail corrections. Compared to previous numerical techniques, our algorithm converges exponentially in the size of the domain and detects the transition point accurately. We suspect that the analytic approach towards finding eigenvalues and resonances in the linearization could be adapted to numerical algorithms in a similar fashion as the nonlinear existence result was adapted to numerical continuation. 

As a first generalization, one would wish to adapt both analysis and numerics to include oscillatory invasion, that is, cases when marginal linear stability is caused by double roots $(\lambda,\nu)$ on the imaginary axis, that is, $\lambda=i\omega$ with $\omega\neq 0$. A transition to this scenario was in fact the limit of validity of our numerical study in the EFKPP case. The tools developed here can in fact be adapted in a straightforward fashion. One would look for time-periodic solutions in a comoving frame, add a phase condition eliminating temporal shifts, and the frequency as a Lagrange parameter in the case of pushed fronts. For pulled fronts, the parameters in the asymptotics $\alpha$ and $\beta$ are complex, compensating for now two phase conditions for space and time and two transversality conditions. We expect quadratic corrections to both linearly predicted speed and frequency for pushed fronts. Clearly, both analytical and numerical setup now require a PDE rather than an ODE boundary-value problem, although Fourier-expansions in the temporal variable converge rapidly and reduce the problem to one similar to the present case; see for instance \cite{ce,ssmod} for more background on this analogy and an analytical existence result for fronts in this context, \cite{lpss} for numerical aspects. In the particularly interesting case where patterns are selected in the wake, one would adapt the far-field core strategy to insert a periodic solution $u(kx;k)$, $u(y)=u(y+2\pi;k)$, which in turn is computed separately from a periodic boundary-value problem; see \cite{lloydscheel,cdgjs} for implementations of this approach using spectral discretization. 

Including this case of oscillatory invasion, one then is tempted to ask for a broad strategy of predicting, analytically or computationally, invasion speeds. The present work presents one step in this direction, but also detects caveats. Without venturing into an open ended discussion of what could possibly go wrong, we point here to the problems with resonances in some of the computational examples, in particular Remark \ref{rem:anomalous}, to the discussion in \cite{averyscheelselection} for perspective, and \cite{FayeHolzerScheel} for an attempt at broadening the concept of linear spreading speeds to general resonances, possibly comprising all linear mechanisms of invasion and leaving only point spectrum instabilities as the ones discussed here as means of altering linear prediction. 

In a different direction, one can compare the analysis here with bifurcations of coherent structures elsewhere. Detaching of defects from a speed of propagation determined by a background field was identified in \cite[\S6.4]{ssdefect} as a generic bifurcation. Defects are localized structures asymptotic to wave trains. Contact defects travel with the group velocity of the background state, whereas transmission defects travel with a nonlinear speed different from that group velocity. Expanding on the discussion presented there, one finds the same universal quadratic correction to speeds in the bifucation that we identified here. The methods there were geometric rather than functional analytic, but the underlying homoclinic codimension-two bifurcation can be recovered in our context using geometric desingularization of the leading-edge equilibrium as demonstrated in \cite{adss}.

\bibliographystyle{abbrv}

\bibliography{pp}

\begin{thebibliography}{10}

\bibitem{an21}
J.~An, C.~Henderson, and L.~Ryzhik.
\newblock Pushed, pulled and pushmi-pullyu fronts of the {B}urgers-{FKPP}
  equation.
\newblock {\em arXiv preprint arXiv:2108.07861}, 2021.

\bibitem{adss}
M.~Avery, C.~Dedina, A.~Smith, and A.~Scheel.
\newblock Instability in large bounded domains---branched versus unbranched
  resonances.
\newblock {\em Nonlinearity}, 34(11):7916--7937, 2021.

\bibitem{AveryGarenaux}
M.~Avery and L.~Gar\'enaux.
\newblock Spectral stability of the critical front in the extended
  {F}isher-{KPP} equation.
\newblock {\em Preprint}, 2020.

\bibitem{averyscheelselection}
M.~Avery and A.~Scheel.
\newblock Universal selection of pulled fronts.
\newblock {\em Comm. Amer. Math. Soc.}, to appear.

\bibitem{bers1983handbook}
A.~Bers, M.~Rosenbluth, and R.~Sagdeev.
\newblock Handbook of plasma physics.
\newblock {\em MN Rosenbluth and RZ Sagdeev eds}, 1(3.2), 1983.

\bibitem{beyn_hetero}
W.-J. Beyn.
\newblock The numerical computation of connecting orbits in dynamical systems.
\newblock {\em IMA J. Numer. Anal.}, 10(3):379--405, 1990.

\bibitem{beyn_freeze}
W.-J. Beyn and V.~Th\"{u}mmler.
\newblock Freezing solutions of equivariant evolution equations.
\newblock {\em SIAM J. Appl. Dyn. Syst.}, 3(2):85--116, 2004.

\bibitem{Bramson1}
M.~Bramson.
\newblock Maximal displacement of branching {B}rownian motion.
\newblock {\em Comm. Pure Appl. Math.}, 31(5):531--581, 1978.

\bibitem{Bramson2}
M.~Bramson.
\newblock {\em Convergence of solutions of the {K}olmogorov equation to
  traveling waves}.
\newblock Mem. Amer. Math. Soc. American Mathematical Society, 1983.

\bibitem{auto_hom}
A.~R. Champneys, Y.~A. Kuznetsov, and B.~Sandstede.
\newblock A numerical toolbox for homoclinic bifurcation analysis.
\newblock {\em Internat. J. Bifur. Chaos Appl. Sci. Engrg.}, 6(5):867--887,
  1996.

\bibitem{cdgjs}
K.~Chen, Z.~Deiman, R.~Goh, S.~Jankovic, and A.~Scheel.
\newblock Strain and defects in oblique stripe growth.
\newblock {\em Multiscale Model. Simul.}, 19(3):1236--1260, 2021.

\bibitem{ce}
P.~Collet and J.-P. Eckmann.
\newblock {\em Instabilities and fronts in extended systems}.
\newblock Princeton Series in Physics. Princeton University Press, Princeton,
  NJ, 1990.

\bibitem{EbertvanSaarloos}
U.~Ebert and W.~van Saarloos.
\newblock Front propagation into unstable states: universal algebraic
  convergence towards uniformly translating pulled fronts.
\newblock {\em Phys. D}, 146:1--99, 2000.

\bibitem{FayeHolzerScheel}
G.~Faye, M.~Holzer, and A.~Scheel.
\newblock Linear spreading speeds from nonlinear resonant interaction.
\newblock {\em Nonlinearity}, 30(6):2403--2442, 2017.

\bibitem{FayeHolzerScheelSiemer}
G.~Faye, M.~Holzer, A.~Scheel, and L.~Siemer.
\newblock Invasion into remnant instability: a case study of front dynamics.
\newblock {\em Indiana Univ. Math. J., to appear}.

\bibitem{focant98}
S.~Focant and T.~Gallay.
\newblock Existence and stability of propagating fronts for an autocatalytic
  reaction-diffusion system.
\newblock {\em Phys. D}, 120(3-4):346--368, 1998.

\bibitem{gohberg2006invariant}
I.~Gohberg, P.~Lancaster, and L.~Rodman.
\newblock {\em Invariant subspaces of matrices with applications}.
\newblock SIAM, 2006.

\bibitem{HadelerRothe}
K.-P. Hadeler and F.~Rothe.
\newblock Traveling fronts in nonlinear diffusion equations.
\newblock {\em J. Math. Biol.}, 2(1):251--263, 1975.

\bibitem{Comparison1}
F.~Hamel, J.~Nolen, J.-M. Roquejoffre, and L.~Ryzhik.
\newblock A short proof of the logarithmic {B}ramson correction in
  {F}isher-{KPP} equations.
\newblock {\em Netw. Heterog. Media}, 8(1):275--289, 2013.

\bibitem{henderson21}
C.~Henderson.
\newblock Slow and fast minimal speed traveling waves of the fkpp equation with
  chemotaxis.
\newblock {\em arXiv preprint arXiv:2102.06065}, 2021.

\bibitem{holzeranomalous}
M.~Holzer.
\newblock Anomalous spreading in a system of coupled {F}isher–{KPP}
  equations.
\newblock {\em Phys. D}, 270:1--10, 2014.

\bibitem{holzer12}
M.~Holzer and A.~Scheel.
\newblock A slow pushed front in a {L}otka-{V}olterra competition model.
\newblock {\em Nonlinearity}, 25(7):2151--2179, 2012.

\bibitem{HolzerScheelPointwiseGrowth}
M.~Holzer and A.~Scheel.
\newblock Criteria for pointwise growth and their role in invasion processes.
\newblock {\em J. Nonlinear Sci.}, 24(1):661--709, 2014.

\bibitem{hosono98}
Y.~Hosono.
\newblock The minimal speed of traveling fronts for a diffusive lotka-volterra
  competition model.
\newblock {\em Bulletin of Mathematical Biology}, 60(3):435--448, 1998.

\bibitem{Kolmogorov}
A.~Kolmogorov, I.~Petrovskii, and N.~Piskunov.
\newblock Etude de l'equation de la diffusion avec croissance de la quantite de
  matiere et son application a un probleme biologique.
\newblock {\em Bjul. Moskowskogo Gos. Univ. Ser. Internat. Sec. A}, 1:1--26,
  1937.

\bibitem{Lau}
K.-S. Lau.
\newblock On the nonlinear diffusion equation of {K}olmogorov, {P}etrovsky, and
  {P}iscounov.
\newblock {\em J. Differential Equations}, 59(1):44--70, 1985.

\bibitem{lewis02}
M.~A. Lewis, B.~Li, and H.~F. Weinberger.
\newblock Spreading speed and linear determinacy for two-species competition
  models.
\newblock {\em J. Math. Biol.}, 45(3):219--233, 2002.

\bibitem{lloydscheel}
D.~J.~B. Lloyd and A.~Scheel.
\newblock Continuation and bifurcation of grain boundaries in the
  {S}wift-{H}ohenberg equation.
\newblock {\em SIAM J. Appl. Dyn. Syst.}, 16(1):252--293, 2017.

\bibitem{lpss}
G.~J. Lord, D.~Peterhof, B.~Sandstede, and A.~Scheel.
\newblock Numerical computation of solitary waves in infinite cylindrical
  domains.
\newblock {\em SIAM J. Numer. Anal.}, 37(5):1420--1454, 2000.

\bibitem{PoganScheel}
A.~Pogan and A.~Scheel.
\newblock Instability of spikes in the presence of conservation laws.
\newblock {\em Z. Angew. Math. Phys.}, 61(6):979--998, 2010.

\bibitem{AbsoluteSpecArndBjorn}
B.~Sandstede and A.~Scheel.
\newblock Absolute and convective instabilities of waves on unbounded and large
  bounded domains.
\newblock {\em Phys. D}, 145(3):233--277, 2000.

\bibitem{ssmod}
B.~Sandstede and A.~Scheel.
\newblock On the structure of spectra of modulated travelling waves.
\newblock {\em Math. Nachr.}, 232:39--93, 2001.

\bibitem{ssdefect}
B.~Sandstede and A.~Scheel.
\newblock Defects in oscillatory media: toward a classification.
\newblock {\em SIAM J. Appl. Dyn. Syst.}, 3(1):1--68, 2004.

\bibitem{ssmorse}
B.~Sandstede and A.~Scheel.
\newblock Relative {M}orse indices, {F}redholm indices, and group velocities.
\newblock {\em Discrete Contin. Dyn. Syst.}, 20(1):139--158, 2008.

\bibitem{SattingerWeightedNorms}
D.~Sattinger.
\newblock Weighted norms for the stability of traveling waves.
\newblock {\em J. Differential Equations}, 25(1):130--144, 1977.

\bibitem{stegemerten}
F.~Stegemerten, S.~V. Gurevich, and U.~Thiele.
\newblock Bifurcations of front motion in passive and active
  {A}llen-{C}ahn--type equations.
\newblock {\em Chaos}, 30(5):053136, 12, 2020.

\bibitem{vanSaarloosReview}
W.~van Saarloos.
\newblock Front propagation into unstable states.
\newblock {\em Phys. Rep.}, 386:29--222, 2003.

\end{thebibliography}

\end{document}